\documentclass[a4paper,10pt]{amsart}
\usepackage{graphicx}
\usepackage{amssymb}
\usepackage{amsmath}
\usepackage{latexsym}
\usepackage{amsthm}
\usepackage{autograph,epic,latexsym,bezier,amsbsy,color,enumerate,amsfonts,amsmath,amscd,amssymb}
\usepackage[latin1]{inputenc}

\setloopdiam{10}\setprofcurve{7}
\newtheorem{defi}{Definition}[section]
\newtheorem{teo}[defi]{Theorem}

\newtheorem{prop}[defi]{Proposition}
\newtheorem{cor}[defi]{Corollary}

\newtheorem{os}[defi]{Remark}

\begin{document}

\title{Counting dimer coverings on self-similar Schreier graphs}

\author{Daniele D'Angeli}
\address{Departamento de Matem\'{a}ticas, Universidad de los Andes, Cra 1 n. 18A-12 Bogot\'{a}, Colombia}
\email{dangeli@uniandes.edu.co}

\author{Alfredo Donno}
\address{Dipartimento di Matematica, Sapienza Università di Roma, Piazzale A. Moro, 5 \quad 00185 Roma, Italia}
\email{donno@mat.uniroma1.it}

\author{Tatiana Nagnibeda}
\address{Section de Mathématiques, Université de Genève, 2-4, Rue du Lièvre, Case Postale 64 1211 Genève 4, Suisse}
\email{Tatiana.Smirnova-Nagnibeda@unige.ch}

\date{\today}

\begin{abstract}
We study partition functions for the
dimer model on families of finite graphs converging to
infinite self-similar graphs and forming approximation sequences to certain well-known fractals. The graphs that we consider are provided
by actions of finitely generated groups by automorphisms on rooted trees, and thus their edges are naturally labeled by the generators of the group. It is thus natural to consider weight functions on these graphs taking different values according to the labeling. We study in detail the well-known example of the Hanoi Towers group $H^{(3)}$,
closely related to the Sierpi\'nski gasket.
\end{abstract}
\maketitle

\begin{center}
{\it Dedicated to Toni Machì}
\end{center}

\begin{center}
{\footnotesize{\bf Mathematics Subject Classification (2010):}
82B20, 05A15, 20E08.\footnote{This research has been supported by
the Swiss National Science Foundation Grant PP0022$_{-}$118946.}}
\end{center}

\section{Introduction}

The dimer model is widely studied in different areas of mathematics ranging from combinatorics to probability theory to algebraic geometry. It originated in statistical mechanics where it was introduced in the purpose of investigating absorption of diatomic
molecules on surfaces. In particular, one wants to find the number of ways in which
diatomic molecules, called dimers, can cover a doubly periodic
lattice, so that each dimer covers two adjacent lattice points and no lattice point
remains uncovered. First exact results on the dimer model in a finite rectangle of $\mathbb{Z}^2$ were obtained by Kasteleyn \cite{kasteleyn1,kasteleyn2} and independently Temperley and Fisher \cite{TF} in the 60-s. A much more recent breakthrough is the solution of the dimer model on arbitrary planar bipartite periodic graphs by Kenyon, Okounkov and Sheffield \cite{KOS}. We refer to the lecture notes by Kenyon \cite{kenyon2} for an introduction into the dimer model.   \\
\indent Let $Y=(V,E)$ be a finite graph with the vertex set $V$ and
the edge set $E$. A \emph{dimer} is a graph consisting of two vertices
connected by a non-oriented edge. A \textit{dimer covering} $D$ of
$Y$ is an arrangement of dimers on $Y$ such that each vertex of
$V$ is the endpoint of exactly one dimer. In other words, dimer coverings correspond exactly to \emph{perfect matchings} in $Y$.
Let $\mathcal{D}$ denote the set of dimer
coverings of $Y$, and let $w:E\longrightarrow \mathbb{R}_{+}$ be a
weight function defined on the edge set of $Y$. The physical meaning of the weight function can be, for example, the interaction energy between the atoms in a diatomic molecule. We associate with
each dimer covering $D\in \mathcal{D}$ its weight
$$
W(D):=\prod_{e\in D} w(e),
$$
i.e., the product of the weights of the edges belonging to $D$. To
each weight function $w$ on $Y$ corresponds the Boltzmann measure
on $\mathcal{D}$, $\mu=\mu(Y,w)$ defined as
$$
\mu(D)=\frac{W(D)}{\Phi(w)}.
$$
The normalizing constant that ensures that this is a probability measure is one of the central objects in the theory, it is called the  \textit{partition function}:
$$
\Phi(w):=\sum_{D\in \mathcal{D}} W(D).
$$
If the weight function is constant equal to 1, the partition function just counts
 all the dimer coverings on $Y$.\\
\indent For a growing sequence of finite graphs, $\{Y_n\}_n$, one can ask whether the limit
$$
\lim_{n\to \infty}\frac{\log (\Phi_n(w_n))}{|V(Y_n)|}
$$
exists, where $w_n$ is the weight function on $Y_n$, and $\Phi_n(w_n)$ is the associated partition
function. If yes, it is then called the \textit{thermodynamic limit}. For
$w_n\equiv 1$, it specializes to the \textit{entropy} of the absorption of
diatomic molecules per site.\\
\indent Let us recall the method developed by Kasteleyn
\cite{kasteleyn1} to compute the partition function of the dimer
model on finite planar graphs. 
It consists in, given a finite graph $Y=(V,E)$, constructing an anti-symmetric matrix such that the absolute value
of its Pfaffian is the partition function of the dimer model on
$Y$. Recall that the Pfaffian $Pf(M)$ of an anti-symmetric matrix
$M=(m_{ij})_{i,j=1,\ldots,N}$, with $N$ even, is
$$
Pf(M):=\sum_{\pi\in Sym(N)}sgn(\pi) m_{p_1p_2}\cdots
m_{p_{N-1}p_N},
$$
where the sum runs over all permutations $\pi=\begin{pmatrix}
  1 & 2 & \cdots & N \\
  p_1 & p_2 & \cdots & p_N
\end{pmatrix}$
such that
$p_1<p_2$, $p_3<p_4,\ \ldots$, $p_{N-1}<p_N$ and $p_1<p_3<\cdots <p_{N-1}$.
One has $(Pf(M))^2= \det(M)$.\\
\indent Given an orientation on $Y$ and a weight
function $w$ on $E$, consider the oriented adjacency matrix
$A=(a_{ij})_{i,j=1,\ldots,|V|}$ of $(Y,w)$ with this orientation.
It is of course anti-symmetric.
\begin{defi}
A good orientation on $Y$ is an orientation of the edges of $Y$
such that the number of clockwise oriented edges around each face of $Y$ is odd.
\end{defi}
\begin{teo}[\cite{kasteleyn1}]
\begin{enumerate}
\item Let $Y=(V,E)$ be a planar graph with a good orientation, let
$w$ be a weight function on $E$. If $A$ is the associated oriented
adjacency matrix, then
$$
\Phi(w)=|Pf(A)|.
$$
\item If $Y$ is planar, a good orientation on $Y$ always exists.
\end{enumerate}
\end{teo}
In this paper we apply Kasteleyn's method to study dimers partition functions on families of finite graphs that form approximating sequences for some
well-known fractals, and on the other hand converge in local convergence to interesting self-similar graphs. The graphs that we consider are Schreier graphs of certain finitely generated groups and thus come naturally endowed with a labeling of the edges of the graph by the generators of the group. It is therefore natural to think about the edges with different labels as being of different type, and to consider weight functions on them that take different values according to the type of the edge.\\
\indent We now turn to Schreier graphs of self-similar groups and recall some basic facts and definitions.
Let $T$ be the infinite regular rooted tree of degree $q$, i.e.,
the rooted tree where each vertex has $q$ offsprings. Every vertex of the $n$-th
level of the tree can be regarded as an element of the set $X^n$ of words of length $n$ over the
alphabet $X=\{0,1,\ldots, q-1\}$; the
set $X^{\omega}$ of infinite words over $X$ can be identified with the set $\partial
T$ of infinite geodesic rays in $T$ emanating from the root. Now let
$G<Aut(T)$ be a group acting on $T$ by automorphisms, generated by
a finite symmetric set $S\subset G$. Throughout the paper we will assume that the action
of $G$ is transitive on each level of the tree (note that any action by automorphisms is level-preserving).
\begin{defi}\label{defischreiernovembre}
The $n$-th {\it Schreier graph} $\Sigma_n$ of the action of $G$ on
$T$, with respect to the generating set $S$, is a (labeled, oriented) graph with
$V(\Sigma_n)=X^n$, and edges $(u,v)$ between vertices $u$ and $v$ such that $u$ is moved to $v$ by the action of some generator $s\in S$. The edge $(u,v)$ is then labeled by $s$. \\
\indent For an infinite ray $\xi\in\partial T$, the {\it orbital Schreier graph} $\Sigma_\xi$
has vertex set $G\cdot \xi$ and the edge set determined by the action
of generators on this orbit, as above.
\end{defi}
The graphs $\Sigma_n$ are Schreier graphs of stabilizers of
vertices of the $n$-th level of the tree and the graphs
$\Sigma_\xi$ are Schreier graphs of stabilizers of infinite rays.
It is not difficult to see that the orbital Schreier graphs are
infinite and that the finite Schreier graphs
$\{\Sigma_n\}_{n=1}^\infty$ form a sequence of graph coverings.
Finite Schreier graphs converge to infinite Schreier graphs in the
space of rooted (labeled) graphs with local convergence (rooted
Gromov-Hausdorff convergence \cite{Grom}, Chapter 3). More
precisely, for an infinite ray $\xi\in X^\omega$ denote by $\xi_n$
the $n$-th prefix of the word $\xi$. Then the sequence of rooted
graphs $\{(\Sigma_n,\xi_n)\}$ converges to the infinite rooted
graph $(\Sigma_\xi, \xi)$ in the space  $\mathcal{X}$ of (rooted
isomorphism classes of ) rooted graphs endowed with the following
metric: the distance between two rooted graphs $(Y_1,v_1)$ and
$(Y_2,v_2)$ is
$$
Dist((Y_1,v_1),(Y_2,v_2)):=\inf\left\{\frac{1}{r+1};\textrm{$B_{Y_1}(v_1,r)$
is isomorphic to $B_{Y_2}(v_2,r)$}\right\}
$$
where $B_Y(v,r)$ is the ball of radius $r$ in $Y$ centered in $v$.
\begin{defi}\label{defiselfsimilar}
A finitely generated group $G<Aut(T)$ is {\it self-similar} if,
for all $g\in G, x\in X$, there exist $h\in G, y\in X$ such that
$$
g(xw)=yh(w),
$$
for all finite words $w$ over the alphabet $X$.
\end{defi}
Self-similarity implies that $G$ can be embedded into the wreath
product $Sym(q)\wr G$, so that any automorphism $g\in G$ can be
represented as
\begin{eqnarray}\label{tauleftformula}
g=\tau(g_0,\ldots,g_{q-1}),
\end{eqnarray}
where $\tau\in Sym(q)$ describes the action of $g$ on the first
level of $T$ and $g_i\in G, i=0,...,q-1$ is the restriction of $g$
on the full subtree of $T$ rooted at the vertex $i$ of the first
level of $T$ (observe that any such subtree is isomorphic to $T$).
Hence, if $x\in X$ and $w$ is a finite word over $X$, we have
$$
g(xw)=\tau(x)g_x(w).
$$
See \cite{volo} and references therein for more
information about this interesting class of groups, also known as
{\it automata groups}.\\
\indent In many cases, self-similarity of a group action
allows to formulate a number of rules that allow to construct inductively the
sequence of Schreier graphs $\{\Sigma_n\}_{n\geq 1}$
\cite{fractal, volo} and thus to
describe inductively the action of the group on
the $n$-th level of the tree. More precisely, the action of $g\in G$ on
the $n$-th level can be represented by a permutation matrix of
size $q$, whose entries are matrices of size $q^{n-1}$. If $g$ is
as in \eqref{tauleftformula}, the nonzero entries of the matrix
are at position $(i,\tau(i))$ and correspond to the action of the
restriction $g_i$ of $g$ on the subtree of depth $n-1$ rooted at
$i$, for each $i=0,\ldots, q-1$.\\
\indent In the next sections we will systematically use this description.
Our idea is to define recursively an oriented adjacency matrix
associated with the action of the generators on the $n$-th level,
with some prescribed signs. The rows and columns of this matrix
are indexed by the words of length $n$ over the alphabet
$\{0,1,\ldots, q-1\}$, in their lexicographic order. The signs can
be interpreted as corresponding to a good orientation of the graph
$\Sigma_n$, in the sense of Kasteleyn. This allows to compute the
partition function and the number of dimer coverings by studying
the Pfaffian of this matrix.\\
\indent In this paper we compute the partition function of the dimer model on the following
examples of planar Schreier graphs associated with self-similar actions:\\ \indent- the first Grigorchuk's
group of intermediate growth (see \cite{grigorchuk} for a detailed
account and further references);\\ \indent - the Basilica group
that can be described as the iterated monodromy group of the
complex polynomial $z^2-1$ (see \cite{volo}
for connections of self-similar groups to complex dynamics);\\
\indent - the Hanoi Towers group $H^{(3)}$ whose action on the
ternary tree models the famous Hanoi Towers game on three pegs,
see \cite{hanoi}, and whose Schreier graphs are closely related to the Sierpi\'{n}ski gasket. Let us mention that counting dimers on the Schreier graphs of the Hanoi towers group $H^{(3)}$ is related to the computation of the partition function for the Ising model on the Sierpi\'{n}ski triangle, via Fisher's correspondence \cite{fisher}, see Subsection 4.5 in our paper \cite{ddn}, devoted to the Ising model on the self-similar Schreier graphs.\\
\indent Finally we also compute the partition function of the dimer model on the (finite approximations of) the Sierpi\'{n}ski triangle. These graphs cannot be labeled so as to become Schreier graphs of a self-similar group, but they are very similar to the Schreier graphs of the group $H^{(3)}$.
They have a few natural labeling of the edges in three different types, of which we describe three, and provide computations for two of those.

\subsection{Plan of the paper}
The rest of the paper is structured as follows. In Section
\ref{SECTION2} we study the dimer model on the Schreier graphs
associated with the action of the Grigorchuk's group and of the
Basilica group on the rooted binary tree. Even if the model on
these graphs can be easily computed directly, we prefer to apply
the general Kasteleyn theory: the partition function at each
finite level is described, the thermodynamic limit and the entropy
are explicitly computed. In Section \ref{SECTION4} the dimer model
is studied on the Schreier graphs of the Hanoi Towers group
$H^{(3)}$. First, we follow a combinatorial approach using
recursion and the property of self-similarity of these graphs (see
Section \ref{SECTIONCOMBINATORIAL}). A recursive description of
the partition function is given in Theorem \ref{numerohanoi}. The
thermodynamic limit is not explicitly computed, although its
existence is proven in two particular cases (see Proposition
\ref{a=b=c} and Proposition \ref{COROLLARYEXISTENCE}). Then, the
problem is studied by using Kasteleyn method (Section \ref{2210}):
the Pfaffian of the oriented adjacency matrix is recursively
investigated via the Schur complement. The description of the
partition function that we give in Theorem
\ref{PROPOSITIONPARTITION} uses iterations of a rational map. In
Section \ref{SECTIONSIERPINSKI}, the dimer model is studied on
finite approximations of the well-known Sierpi\'{n}ski gasket:
these are self-similar graphs closely related to the Schreier
graphs of the group $H^{(3)}$. Two different weight functions on
the edges of these graphs are considered and for both of them the
partition function, the thermodynamic limit and the entropy are
computed. In Section \ref{statistiques} we perform, for the
Schreier graph of $H^{(3)}$ and the Sierpi\'{n}ski gasket, a
statistical analysis of the random variable defined as the number
of occurrences of a fixed label in a random dimer covering.


\section{The partition function of the dimer model on the Schreier graphs of the
Grigorchuk's group and of the Basilica group}\label{SECTION2}

In this section we study the dimer model on two examples of
Schreier graphs: the Schreier graphs of the Grigorchuk's group and
of the Basilica group. Even if in these cases the problem can be
easily solved combinatorially, we prefer to apply
here the Kasteleyn theory because we will follow the same strategy
in the next sections to solve the problem on more complicated
graphs.
\subsection{The Grigorchuk's group}

Let us start with the Grigorchuk's group: this is the self-similar
group acting on the rooted binary tree generated by the
automorphisms:
$$
a=\epsilon(id,id), \qquad b=e(a,c), \qquad  c=e(a,d), \qquad
d=e(id,b),
$$
where $e$ and $\epsilon$ are, respectively, the trivial and the
non-trivial permutations in $Sym(2)$ (observe that all the
generators are involutions). The following substitutional rules
describe how to construct recursively the graph $\Sigma_{n+1}$
from $\Sigma_n$, starting from the Schreier graph of the first
level $\Sigma_1$ \cite{hecke, grigorchuk}. More precisely, the
construction consists in replacing the labeled subgraphs of
$\Sigma_{n}$ on the top of the picture by new labeled graphs (on
the bottom).
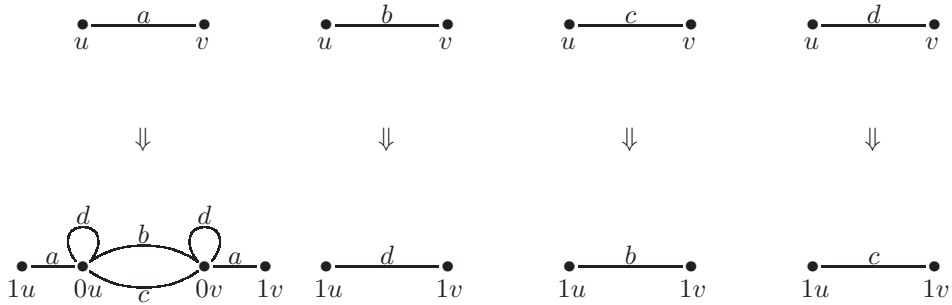
\begin{center}
\begin{picture}(400,110)
\letvertex A=(65,100)\letvertex B=(105,100)\letvertex C=(145,100)
\letvertex D=(185,100)\letvertex E=(225,100)\letvertex F=(265,100)\letvertex G=(305,100)
\letvertex H=(345,100)\letvertex I=(45,20)\letvertex L=(65,20)\letvertex M=(105,20)\letvertex N=(125,20)
\letvertex c=(145,20)
\letvertex d=(185,20)\letvertex e=(225,20)\letvertex f=(265,20)\letvertex g=(305,20)
\letvertex h=(345,20)

\put(82,60){$\Downarrow$}\put(162,60){$\Downarrow$}\put(242,60){$\Downarrow$}\put(322,60){$\Downarrow$}

\put(62,92){$u$} \put(102,92){$v$}\put(142,92){$u$}
\put(182,92){$v$}\put(222,92){$u$} \put(262,92){$v$}
\put(302,92){$u$}\put(342,92){$v$}

\put(40,10){$1u$} \put(62,10){$0u$}\put(102,10){$0v$}
\put(122,10){$1v$}

\put(141,10){$1u$} \put(181,10){$1v$}\put(221,10){$1u$}
\put(261,10){$1v$} \put(301,10){$1u$}\put(341,10){$1v$}

\drawvertex(A){$\bullet$}\drawvertex(B){$\bullet$}
\drawvertex(C){$\bullet$}\drawvertex(D){$\bullet$}
\drawvertex(E){$\bullet$}\drawvertex(F){$\bullet$}
\drawvertex(G){$\bullet$}\drawvertex(H){$\bullet$}
\drawvertex(I){$\bullet$}\drawvertex(L){$\bullet$}
\drawvertex(M){$\bullet$}\drawvertex(N){$\bullet$}
\drawvertex(c){$\bullet$}\drawvertex(d){$\bullet$}
\drawvertex(e){$\bullet$}\drawvertex(f){$\bullet$}
\drawvertex(g){$\bullet$}\drawvertex(h){$\bullet$}

\drawundirectedloop(L){$d$}\drawundirectedloop(M){$d$}

\drawundirectededge(A,B){$a$}\drawundirectededge(C,D){$b$}
\drawundirectededge(E,F){$c$} \drawundirectededge(G,H){$d$}
\drawundirectededge(c,d){$d$}\drawundirectededge(e,f){$b$}
\drawundirectededge(g,h){$c$}
\drawundirectededge[r](L,I){$a$}\drawundirectededge(M,N){$a$}
\drawundirectedcurvededge(L,M){$b$}\drawundirectedcurvededge[b](L,M){$c$}
\end{picture}
\end{center}
starting from
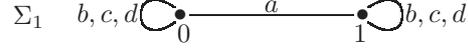
\begin{center}
\begin{picture}(200,40)
\letvertex A=(70,25)\letvertex B=(130,25)

\put(69,16){$0$}\put(127,16){$1$}\put(15,22){$\Sigma_1$}

\drawvertex(A){$\bullet$}\drawvertex(B){$\bullet$}

\drawundirectedloop[l](A){$b,c,d$}\drawundirectedloop[r](B){$b,c,d$}
\drawundirectededge(A,B){$a$}
\end{picture}
\end{center}
In the study of the dimer model on these graphs, we consider the
graphs without loops. We keep the notation $\Sigma_n$ for these
graphs. The following pictures give examples for $n=1,2,3$.
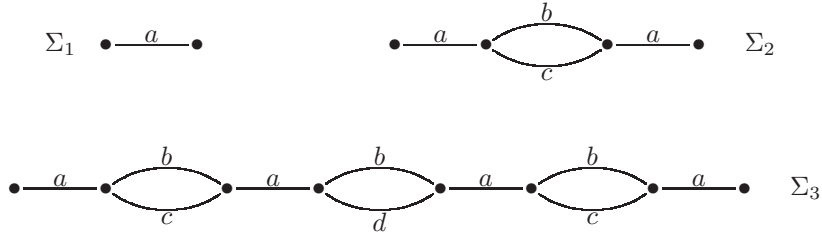
\begin{center}
\begin{picture}(300,40)
\letvertex A=(60,20)\letvertex B=(90,20)

\letvertex C=(155,20)\letvertex D=(185,20)
\letvertex E=(225,20)\letvertex F=(255,20)

\drawundirectededge(A,B){$a$} \drawundirectededge(C,D){$a$}
\drawundirectededge(E,F){$a$}
\drawundirectedcurvededge(D,E){$b$}\drawundirectedcurvededge(E,D){$c$}

\put(40,18){$\Sigma_1$} \put(270,18){$\Sigma_2$}

\drawvertex(A){$\bullet$}\drawvertex(B){$\bullet$}
\drawvertex(C){$\bullet$}\drawvertex(D){$\bullet$}
\drawvertex(E){$\bullet$}\drawvertex(F){$\bullet$}
\end{picture}
\end{center}

\begin{center}
\begin{picture}(300,40)
\letvertex A=(30,20)\letvertex B=(60,20)\letvertex C=(100,20)\letvertex D=(130,20)
\letvertex E=(170,20)\letvertex F=(200,20)\letvertex G=(240,20)\letvertex H=(270,20)

\drawundirectededge(A,B){$a$} \drawundirectededge(C,D){$a$}
\drawundirectededge(E,F){$a$} \drawundirectededge(G,H){$a$}

\put(285,18){$\Sigma_3$}

\drawvertex(A){$\bullet$}\drawvertex(B){$\bullet$}
\drawvertex(C){$\bullet$}\drawvertex(D){$\bullet$}
\drawvertex(E){$\bullet$}\drawvertex(F){$\bullet$}
\drawvertex(G){$\bullet$}\drawvertex(H){$\bullet$}

\drawundirectedcurvededge(B,C){$b$}
\drawundirectedcurvededge(C,B){$c$}
\drawundirectedcurvededge(D,E){$b$}
\drawundirectedcurvededge(E,D){$d$}
\drawundirectedcurvededge(F,G){$b$}
\drawundirectedcurvededge(G,F){$c$}
\end{picture}
\end{center}
In general, the Schreier graph $\Sigma_n$, without loops, has a
linear shape and it has $2^{n-1}$ simple edges, all labeled by
$a$, and $2^{n-1}-1$ cycles of length 2 whose edges are labeled by $b,c,d$.\\
\indent What we need in order to apply the Kasteleyn theory is an
adjacency matrix giving a good orientation to $\Sigma_n$. We start
by providing the (unoriented weighted) adjacency matrix
$\Delta_n$ of $\Sigma_n$, which refers to the graph with loops, that one
can easily get by using the self-similar definition of the
generators of the group. It
is defined by putting
$$
a_1 = \begin{pmatrix}
  0 & 1 \\
  1 & 0
\end{pmatrix} \qquad b_1 = c_1 = d_1 = \begin{pmatrix}
  1 & 0 \\
  0 & 1
\end{pmatrix}
$$
and, for every $n\geq 2$,
$$
a_n = \begin{pmatrix}
  0 & I_{n-1} \\
  I_{n-1} & 0
\end{pmatrix}, \qquad b_n = \begin{pmatrix}
  a_{n-1} & 0 \\
  0 & c_{n-1}
\end{pmatrix}, \qquad c_n = \begin{pmatrix}
  a_{n-1} & 0 \\
  0 & d_{n-1}
\end{pmatrix}, \qquad d_n = \begin{pmatrix}
  I_{n-1} & 0 \\
  0 & b_{n-1}
\end{pmatrix},
$$
where $a_n,b_n,c_n,d_n$ and $I_n$ are matrices of size $2^n$. If
we put $A_n=aa_n, B_n=bb_n, C_n=cc_n$ and $D_n = dd_n$, then
$\Delta_n$ is given by
$$
\Delta_n = A_n + B_n + C_n + D_n =
\begin{pmatrix}
  ba_{n-1}+ca_{n-1}+dI_{n-1} & aI_{n-1} \\
  aI_{n-1} & bc_{n-1} + cd_{n-1} + db_{n-1}
\end{pmatrix}.
$$
We want now to modify $\Delta_n$ in order to get an oriented
adjacency matrix $\Delta_n'$ for $\Sigma_n$, corresponding to a
good orientation in the sense of Kasteleyn. To do this, it is
necessary to delete the nonzero diagonal entries in $\Delta_n$
(this is equivalent to delete loops in the graph) and to construct
an anti-symmetric matrix whose entries coincide, up to the sign,
with the corresponding entries of $\Delta_n$. Finally, we have to
verify that each cycle of $\Sigma_n$, with the orientation induced
by $\Delta_n'$, has an odd number of edges clockwise oriented. So
let us define the matrices
$$
a_1' = \begin{pmatrix}
  0 & 1 \\
  -1 & 0
\end{pmatrix} \qquad \ b_1' = c_1' =d_1' = \begin{pmatrix}
  1 & 0 \\
  0 & 1
\end{pmatrix}.
$$
Then, for every $n\geq 2$, we put
$$
a_n' = \begin{pmatrix}
  0 & I_{n-1} \\
  -I_{n-1} & 0
\end{pmatrix}, \qquad b_n' = \begin{pmatrix}
  a_{n-1}' & 0 \\
  0 & c_{n-1}'
\end{pmatrix}, \qquad c_n' = \begin{pmatrix}
  a_{n-1}' & 0 \\
  0 & d_{n-1}'
\end{pmatrix}, \qquad  d_n' = \begin{pmatrix}
  I_{n-1} & 0 \\
  0 & b_{n-1}'
\end{pmatrix}.
$$
For each $n$, we put $A_n'=aa_n', B_n'=bb_n', C_n'=cc_n'$ and
$D_n' = dd_n'$. Moreover, set
$$
J_1 = \begin{pmatrix}
  b+c+d & 0 \\
  0 & b+c+d
\end{pmatrix} \quad \mbox{and } \ \ J_n =\begin{pmatrix}
  dI_{n-1} & 0 \\
  0 & \overline{J_{n-1}}
\end{pmatrix} \ \ \mbox{ for } n\geq 2,
$$
where, for every $n\geq 1$, the matrix $\overline{J_n}$ is
obtained from $J_n$ with the following substitutions:
$$
b\mapsto d \qquad c\mapsto b \qquad d\mapsto c.
$$
Define
$$
\Delta_1' = A_1' + B_1' + C_1' + D_1' - J_1 = \begin{pmatrix}
  0 & a \\
  -a & 0
\end{pmatrix}
$$
and, for each $n\geq 2$,
$$
\Delta_n' = A_n' + B_n' + C_n' + D_n' - J_n = \begin{pmatrix}
  ba_{n-1}'+ca_{n-1}' & aI_{n-1} \\
  -aI_{n-1} & bc_{n-1}' + cd_{n-1}' + db_{n-1}' - \overline{J_{n-1}}
\end{pmatrix}.
$$
Note that the matrix $J_n$ is introduced to erase the nonzero
diagonal entries of $\Delta_n$, corresponding to loops.
\begin{prop}
The matrix $\Delta_n'$ induces a good orientation on the Schreier
graph $\Sigma_n$ of the Grigorchuk's group.
\end{prop}

\begin{proof}
It is easy to show by induction that $\Delta_n'$ is anti-symmetric
and that each entry of $\Delta_n'$ coincides, up to the sign, with
the corresponding entry of the adjacency matrix $\Delta_n$ of
$\Sigma_n$, where loops are deleted. Finally, we know that all
cycles in the Schreier graph have length $2$ and this ensures that
each cycle has a good orientation in the sense of Kasteleyn.
\end{proof}
\begin{teo}
The partition function of the dimer model on the Schreier graph
$\Sigma_n$ of the Grigorchuk's group is
$$
\Phi_n(a,b,c,d)= a^{2^{n-1}}.
$$
\end{teo}

\begin{proof}
It is easy to check, by using the self-similar definition of the
generators of the group, that
$$
a(1^{n-1}0)=01^{n-2}0 \qquad b(1^{n-1}0)=c(1^{n-1}0)=d(1^{n-1}0)
=1^{n-1}0
$$
and
$$
a(1^n)=01^{n-1} \qquad  b(1^n)=c(1^n)=d(1^n)=1^n.
$$
This implies that the vertices $1^{n-1}0$ and $1^n$ are the (only)
two vertices of degree $1$ of $\Sigma_n$, for each $n$. This
allows us to easily compute $\det(\Delta_n')$ by an iterated
application of the Laplace expansion. We begin from the element $a$ at the entry
$(2^{n-1},2^n)$, which is the only nonzero element of the column
$2^n$. So we can \lq\lq burn\rq\rq the row $2^{n-1}$ and the
column $2^n$. Similarly, row $2^n$ and column $2^{n-1}$ can be
deleted and a second factor $a$ appears in $\det(\Delta_n')$. With
these deletions, we have \lq\lq deleted\rq\rq in the graph all
edges going to and coming from the vertex $01^{n-1}$
(corresponding to the row and column $2^{n-1}$). So the vertex
$001^{n-2}$ (which is adjacent to $01^{n-1}$ in $\Sigma_n$) has
now degree $1$ and on the lines corresponding to it there is just
a letter $a$ (or $-a$) corresponding to the edge joining it to
$101^{n-2}$. Hence, the Laplace expansion can be applied again with
respect to this element, and so on. Observe that each
simple edge labeled $a$ of $\Sigma_n$ contributes $a^2$ to
$\det(\Delta_n')$. The assertion follows since the number of
simple edges is $2^{n-1}$.
\end{proof}

\begin{cor}
The thermodynamic limit is $\frac{1}{2}\log a$. In particular, the
entropy of absorption of diatomic molecules per site is zero.
\end{cor}

\begin{proof}
A direct computation gives
\begin{eqnarray*}
\lim_{n\to +\infty}
\frac{\log(\Phi_n(a,b,c,d))}{|V(\Sigma_n)|}=\lim_{n\to +\infty}
\frac{\log(\Phi_n(a,b,c,d))}{2^n}= \frac{1}{2}\log a.
\end{eqnarray*}
By putting $a=1$, we get the entropy.
\end{proof}
\subsection{The Basilica group}

The Basilica group \cite{primo} is the self-similar group
generated by the automorphisms:
$$
a=e(b,id), \qquad b=\epsilon(a,id).
$$
It acts level-transitively on the binary tree, and the following substitutional rules \cite{ddn2}
 allow to construct inductively $\Sigma_{n+1}$ from $\Sigma_n$,
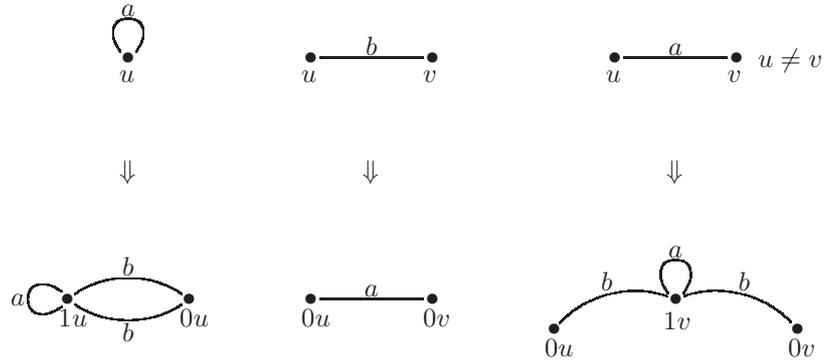
\begin{center}
\begin{picture}(400,110)
\letvertex A=(120,100)\letvertex B=(100,20)\letvertex C=(140,20)

\letvertex D=(180,100)\letvertex E=(220,100)\letvertex F=(180,20)\letvertex G=(220,20)
\letvertex H=(280,100)\letvertex I=(320,100)\letvertex L=(260,10)\letvertex M=(300,20)\letvertex N=(340,10)
\put(117,60){$\Downarrow$}\put(197,60){$\Downarrow$}\put(297,60){$\Downarrow$}

\put(117,92){$u$}\put(97,11){$1u$}\put(137,11){$0u$}

\put(177,92){$u$}\put(217,92){$v$}\put(177,11){$0u$}\put(217,11){$0v$}

\put(277,92){$u$}\put(317,92){$v$}
\put(257,1){$0u$}\put(296,10){$1v$}\put(337,1){$0v$}

\put(327,97){$u\neq v$}

\drawvertex(A){$\bullet$}\drawvertex(B){$\bullet$}
\drawvertex(C){$\bullet$}\drawvertex(D){$\bullet$}
\drawvertex(E){$\bullet$}\drawvertex(F){$\bullet$}
\drawvertex(G){$\bullet$}\drawvertex(H){$\bullet$}
\drawvertex(I){$\bullet$}\drawvertex(L){$\bullet$}
\drawvertex(M){$\bullet$}\drawvertex(N){$\bullet$}

\drawundirectedloop(A){$a$}\drawundirectedloop[l](B){$a$}
\drawundirectedcurvededge(B,C){$b$}\drawundirectedcurvededge(C,B){$b$}

\drawundirectededge(D,E){$b$} \drawundirectededge(F,G){$a$}
\drawundirectededge(H,I){$a$}
\drawundirectedcurvededge(L,M){$b$}\drawundirectedloop(M){$a$}
\drawundirectedcurvededge(M,N){$b$}
\end{picture}
\end{center}
starting with the Schreier graph $\Sigma_1$ on the first level.
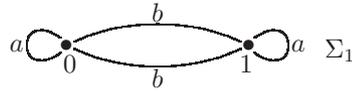
\begin{center}
\begin{picture}(200,40)
\letvertex A=(70,25)\letvertex B=(130,25)

\put(69,16){$0$}\put(127,16){$1$}\put(155,21){$\Sigma_1$}

\drawvertex(A){$\bullet$}\drawvertex(B){$\bullet$}

\drawundirectedloop[l](A){$a$}\drawundirectedloop[r](B){$a$}
\drawundirectedcurvededge(A,B){$b$}\drawundirectedcurvededge(B,A){$b$}
\end{picture}
\end{center}

We consider here the dimer model on the Schreier graphs of the
Basilica group without loops, as in the following pictures, for
$n=1,...,5$. \unitlength=0,3mm
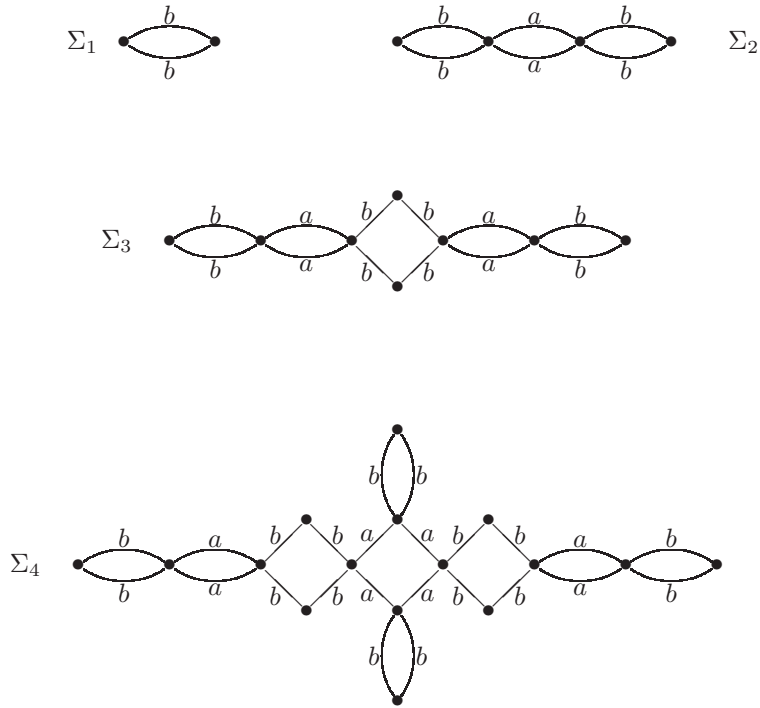
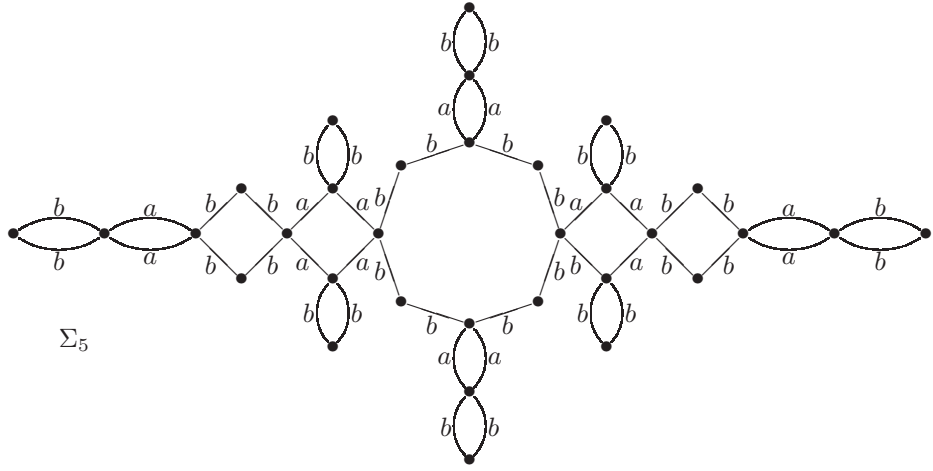
\begin{center}
\begin{picture}(300,30)
\letvertex A=(30,15)\letvertex B=(70,15)\letvertex C=(150,15)\letvertex D=(190,15)
\letvertex E=(230,15)\letvertex F=(270,15)

\drawvertex(A){$\bullet$}\drawvertex(B){$\bullet$}
\drawvertex(C){$\bullet$}\drawvertex(D){$\bullet$}
\drawvertex(E){$\bullet$}\drawvertex(F){$\bullet$}

\drawundirectedcurvededge(A,B){$b$}
\drawundirectedcurvededge(B,A){$b$}
\drawundirectedcurvededge(C,D){$b$}
\drawundirectedcurvededge(D,C){$b$}
\drawundirectedcurvededge(D,E){$a$}
\drawundirectedcurvededge(E,D){$a$}
\drawundirectedcurvededge(E,F){$b$}
\drawundirectedcurvededge(F,E){$b$} \put(5,12){$\Sigma_1$}
\put(295,12){$\Sigma_2$}
\end{picture}
\end{center}

\begin{center}
\begin{picture}(300,60)
\letvertex A=(50,30)\letvertex B=(90,30)\letvertex C=(130,30)\letvertex D=(150,50)
\letvertex E=(150,10)\letvertex F=(170,30)\letvertex G=(210,30)\letvertex H=(250,30)

\drawvertex(A){$\bullet$}\drawvertex(B){$\bullet$}
\drawvertex(C){$\bullet$}\drawvertex(D){$\bullet$}
\drawvertex(E){$\bullet$}\drawvertex(F){$\bullet$}
\drawvertex(G){$\bullet$}\drawvertex(H){$\bullet$}

\drawundirectededge(C,D){$b$} \drawundirectededge(E,C){$b$}
\drawundirectededge(F,E){$b$} \drawundirectededge(D,F){$b$}

\drawundirectedcurvededge(A,B){$b$}
\drawundirectedcurvededge(B,A){$b$}
\drawundirectedcurvededge(B,C){$a$}
\drawundirectedcurvededge(C,B){$a$}
\drawundirectedcurvededge(F,G){$a$}
\drawundirectedcurvededge(G,F){$a$}
\drawundirectedcurvededge(G,H){$b$}
\drawundirectedcurvededge(H,G){$b$} \put(20,27){$\Sigma_3$}
\end{picture}
\end{center}

\begin{center}
\begin{picture}(300,140)
\letvertex A=(10,70)
\letvertex B=(50,70)\letvertex C=(90,70)\letvertex D=(110,90)\letvertex E=(110,50)
\letvertex F=(130,70)\letvertex G=(150,90)\letvertex H=(150,50)\letvertex I=(170,70)

\letvertex J=(150,130)\letvertex K=(150,10)

\letvertex L=(190,90)
\letvertex M=(190,50)\letvertex N=(210,70)
\letvertex O=(250,70)\letvertex P=(290,70)

\drawvertex(A){$\bullet$}\drawvertex(B){$\bullet$}
\drawvertex(C){$\bullet$}\drawvertex(D){$\bullet$}
\drawvertex(E){$\bullet$}\drawvertex(F){$\bullet$}
\drawvertex(G){$\bullet$}\drawvertex(H){$\bullet$}
\drawvertex(I){$\bullet$}\drawvertex(L){$\bullet$}
\drawvertex(M){$\bullet$}\drawvertex(N){$\bullet$}
\drawvertex(O){$\bullet$}\drawvertex(P){$\bullet$}
\drawvertex(J){$\bullet$}\drawvertex(K){$\bullet$}

\drawundirectedcurvededge(A,B){$b$}\drawundirectedcurvededge(B,A){$b$}
\drawundirectedcurvededge(B,C){$a$}\drawundirectedcurvededge(C,B){$a$}
\drawundirectededge(C,D){$b$} \drawundirectededge(D,F){$b$}
\drawundirectededge(F,E){$b$} \drawundirectededge(E,C){$b$}
\drawundirectededge(F,G){$a$} \drawundirectededge(G,I){$a$}
\drawundirectededge(I,H){$a$} \drawundirectededge(H,F){$a$}

\drawundirectededge(I,L){$b$} \drawundirectededge(L,N){$b$}
\drawundirectededge(N,M){$b$} \drawundirectededge(M,I){$b$}

\drawundirectedcurvededge(G,J){$b$}\drawundirectedcurvededge(J,G){$b$}
\drawundirectedcurvededge(H,K){$b$}\drawundirectedcurvededge(K,H){$b$}
\drawundirectedcurvededge(N,O){$a$}\drawundirectedcurvededge(O,N){$a$}
\drawundirectedcurvededge(O,P){$b$}\drawundirectedcurvededge(P,O){$b$}
\put(-20,67){$\Sigma_4$}
\end{picture}
\end{center}

\begin{center}
\begin{picture}(400,210)
\letvertex A=(0,110)\letvertex B=(40,110)\letvertex C=(80,110)\letvertex D=(100,130)
\letvertex E=(100,90)\letvertex F=(120,110)\letvertex G=(140,130)\letvertex H=(140,160)
\letvertex I=(160,110)\letvertex L=(140,90)\letvertex M=(140,60)\letvertex N=(170,140)
\letvertex O=(200,150)\letvertex R=(230,140)\letvertex S=(240,110)\letvertex T=(230,80)
\letvertex U=(200,70)\letvertex V=(170,80)\letvertex P=(200,180)\letvertex Q=(200,210)
\letvertex Z=(200,40)\letvertex J=(200,10)\letvertex K=(260,130)\letvertex X=(280,110)
\letvertex W=(260,90)\letvertex g=(260,160)\letvertex h=(260,60)\letvertex c=(300,130)
\letvertex Y=(300,90)\letvertex d=(320,110)\letvertex e=(360,110)\letvertex f=(400,110)
\drawvertex(A){$\bullet$}\drawvertex(B){$\bullet$}
\drawvertex(C){$\bullet$}\drawvertex(D){$\bullet$}
\drawvertex(E){$\bullet$}\drawvertex(F){$\bullet$}
\drawvertex(G){$\bullet$}\drawvertex(H){$\bullet$}
\drawvertex(I){$\bullet$}\drawvertex(L){$\bullet$}
\drawvertex(M){$\bullet$}\drawvertex(N){$\bullet$}
\drawvertex(O){$\bullet$}\drawvertex(P){$\bullet$}
\drawvertex(J){$\bullet$}\drawvertex(K){$\bullet$}
\drawvertex(Q){$\bullet$}\drawvertex(R){$\bullet$}
\drawvertex(S){$\bullet$}\drawvertex(T){$\bullet$}
\drawvertex(U){$\bullet$}\drawvertex(V){$\bullet$}
\drawvertex(W){$\bullet$}\drawvertex(X){$\bullet$}
\drawvertex(Y){$\bullet$}\drawvertex(Z){$\bullet$}
\drawvertex(g){$\bullet$}\drawvertex(h){$\bullet$}
\drawvertex(c){$\bullet$}\drawvertex(f){$\bullet$}
\drawvertex(d){$\bullet$}\drawvertex(e){$\bullet$}

\drawundirectedcurvededge(A,B){$b$}\drawundirectedcurvededge(B,A){$b$}
\drawundirectedcurvededge(B,C){$a$}\drawundirectedcurvededge(C,B){$a$}
\drawundirectededge(C,D){$b$} \drawundirectededge(D,F){$b$}
\drawundirectededge(F,E){$b$} \drawundirectededge(E,C){$b$}

\drawundirectededge(F,G){$a$} \drawundirectededge(G,I){$a$}
\drawundirectededge(I,L){$a$} \drawundirectededge(L,F){$a$}

\drawundirectedcurvededge(G,H){$b$}\drawundirectedcurvededge(H,G){$b$}
\drawundirectedcurvededge(L,M){$b$}\drawundirectedcurvededge(M,L){$b$}

\drawundirectededge(I,N){$b$} \drawundirectededge(N,O){$b$}
\drawundirectededge(O,R){$b$} \drawundirectededge(R,S){$b$}
\drawundirectededge(S,T){$b$} \drawundirectededge(T,U){$b$}
\drawundirectededge(U,V){$b$} \drawundirectededge(V,I){$b$}

\drawundirectedcurvededge(O,P){$a$}\drawundirectedcurvededge(P,O){$a$}
\drawundirectedcurvededge(Q,P){$b$}\drawundirectedcurvededge(P,Q){$b$}

\drawundirectedcurvededge(U,Z){$a$}\drawundirectedcurvededge(Z,U){$a$}
\drawundirectedcurvededge(Z,J){$b$}\drawundirectedcurvededge(J,Z){$b$}

\drawundirectededge(S,K){$a$} \drawundirectededge(K,X){$a$}
\drawundirectededge(X,W){$a$} \drawundirectededge(W,S){$b$}
\drawundirectededge(X,c){$b$} \drawundirectededge(c,d){$b$}
\drawundirectededge(d,Y){$b$} \drawundirectededge(Y,X){$b$}

\drawundirectedcurvededge(d,e){$a$}\drawundirectedcurvededge(e,d){$a$}
\drawundirectedcurvededge(e,f){$b$}\drawundirectedcurvededge(f,e){$b$}
\drawundirectedcurvededge(K,g){$b$}\drawundirectedcurvededge(g,K){$b$}
\drawundirectedcurvededge(W,h){$b$}\drawundirectedcurvededge(h,W){$b$}
\put(20,60){$\Sigma_5$}
\end{picture}
\end{center}
It follows from the substitutional rules described above that each $\Sigma_n$ is
a cactus, (i.e., a separable graph whose blocks are either cycles or single edges), and that the maximal length of a cycle in
$\Sigma_n$ is $\lceil\frac{n}{2}\rceil$.
We further compute the number of cycles in
$\Sigma_n$, that will be needed later. Denote by $a^i_j$
the number of cycles of length $j$ labeled by $a$ in $\Sigma_i$
and, similarly, denote by $b^i_j$ the number of cycles of length
$j$ labeled by $b$ in $\Sigma_i$.

\begin{prop}\label{computationcycles}
For any $n\geq 4$ consider the Schreier graph $\Sigma_n$ of the
Basilica group. For each $k\geq 1$, the number of cycles of length
$2^k$ labeled by $a$ is
$$
a^n_{2^k} =
\begin{cases}
2^{n-2k-1} & 1\leq k \leq \frac{n-1}{2}-1\\
2 & k = \frac{n-1}{2}
\end{cases}, \mbox{ for } n \mbox{ odd},\qquad
a^n_{2^k}= \begin{cases}
2^{n-2k-1} & 1\leq k \leq \frac{n}{2}-1\\
1 & k = \frac{n}{2}
  \end{cases}, \mbox{ for }  n \mbox{ even}
$$
and the number of cycles of length $2^k$ labeled by $b$ is
$$
b^n_{2^k}=\begin{cases}
2^{n-2k} & 1\leq k \leq \frac{n-1}{2}-1\\
2 &  k = \frac{n-1}{2}\\
1 & k = \frac{n+1}{2}
\end{cases}, \mbox{ for } n \mbox{ odd}, \qquad
b^n_{2^k}=\begin{cases}
2^{n-2k} & 1\leq k \leq \frac{n}{2}-1\\
2& k = \frac{n}{2}
\end{cases}, \mbox{ for }  n \mbox{ even}.
$$
\end{prop}
\begin{proof}
The recursive formulae for the generators imply that, for each
$n\geq 3$, one has
$$
a^n_2 = b^{n-1}_2 \ \ \ \mbox{and } \ \ \ b^n_2 = a^{n-1}_1 =
2^{n-2}
$$
and in general $a^n_{2^k} = a^{n-2(k-1)}_2$ and $b^n_{2^k} =
b^{n-2(k-1)}_2$. In particular, for each $n\geq 4$, the number of
$2$-cycles labeled by $a$ is $2^{n-3}$ and the number of
$2$-cycles labeled by $b$ is $2^{n-2}$. More generally, the number
of cycles of length $2^k$ is given by
$$
a^n_{2^k} = 2^{n-2k-1}, \ \ \ \ b^n_{2^k}=2^{n-2k},
$$
where the last equality is true if $n-2k+2 \geq 4$, i.e., for
$k\leq \frac{n}{2}-1$. Finally, for $n$ odd, there is only one
cycle of length $2^{\frac{n+1}{2}}$ labeled by $b$ and there are four cycles
of length $2^{\frac{n-1}{2}}$, two of them labeled by $a$ and two
labeled by $b$; for $n$ even, there are three cycles of length
$2^{\frac{n}{2}}$, two of them labeled by $b$ and one labeled by
$a$.
\end{proof}
\begin{cor}
For each $n\geq 4$, the number of cycles labeled by $a$ in
$\Sigma_n$ is
$$
\begin{cases}
\frac{2^{n-1}+2}{3}  & n \ \text{odd},\\
\frac{2^{n-1}+1}{3}  & n \ \text{even}
\end{cases}
$$
and the number of cycles labeled by $b$ in $\Sigma_n$ is
$$
\begin{cases}
\frac{2^n+1}{3}  & n \ \text{odd},\\
\frac{2^n+2}{3}  & n \ \text{even}.
\end{cases}
$$
The total number of cycles of length $\geq 2$ is $2^{n-1}+1$ and
the total number of edges, without loops, is $3\cdot 2^{n-1}$.
\end{cor}
Also in this case we construct an adjacency matrix $\Delta_n'$
associated with a good orientation of $\Sigma_n$, in the sense of
Kasteleyn. We first present the (unoriented weighted) adjacency
matrix $\Delta_n$ of the Schreier graph of the Basilica group.
Define the matrices
$$
a_1 = a_1^{-1} = \begin{pmatrix}
  1 & 0 \\
  0 & 1
\end{pmatrix}, \ \ \ \mbox{and } \ b_1 = b_1^{-1} = \begin{pmatrix}
  0 & 1 \\
  1 & 0
\end{pmatrix}.
$$
Then, for every $n\geq 2$, we put
$$
a_n = \begin{pmatrix}
 b_{n-1} & 0 \\
  0 & I_{n-1}
\end{pmatrix}, \qquad a_n^{-1} = \begin{pmatrix}
  b_{n-1}^{-1} & 0 \\
  0 & I_{n-1}
\end{pmatrix}, \qquad b_n = \begin{pmatrix}
  0 & a_{n-1} \\
  I_{n-1} & 0
\end{pmatrix},\qquad  b_n^{-1} = \begin{pmatrix}
  0 & I_{n-1} \\
  a_{n-1}^{-1} & 0
\end{pmatrix}.
$$
If we put $A_n=aa_n, A_n^{-1}=aa_n^{-1}, B_n=bb_n$ and $B_n^{-1} =
bb_n^{-1}$, then $\Delta_n$ is given by
$$
\Delta_n = A_n + A_n^{-1}+B_n +B_n^{-1} =
\begin{pmatrix}
  a(b_{n-1}+b_{n-1}^{-1}) & b(a_{n-1}+I_{n-1}) \\
  b(a_{n-1}^{-1}+I_{n-1}) & 2aI_{n-1}
\end{pmatrix}.
$$
We modify now $\Delta_n$ in order to get the oriented adjacency
matrix $\Delta_n'$. To do this, we need to delete the nonzero
diagonal entries and to construct an anti-symmetric matrix whose
entries are equal, up to the sign, to the corresponding entries of
$\Delta_n$. Finally, we have to check that each elementary cycle
of $\Sigma_n$, with the orientation induced by $\Delta_n'$, has an
odd number of edges clockwise oriented. We define the matrices
$$
a_1' = \begin{pmatrix}
  1 & 0 \\
  0 & 1
\end{pmatrix}, \qquad  a_1'^{-1} = \begin{pmatrix}
  -1 & 0 \\
  0 & -1
\end{pmatrix} \ \mbox{and } \ b_1' = b_1'^{-1}= \begin{pmatrix}
  0 & 1 \\
  -1 & 0
\end{pmatrix}.
$$
Then, for every $n\geq 2$, we put
$$
a_n' = \begin{pmatrix}
  b_{n-1}' & 0 \\
  0 & I_{n-1}
\end{pmatrix}, \quad a_n'^{-1} = \begin{pmatrix}
  b_{n-1}'^{-1} & 0 \\
  0 & -I_{n-1}
\end{pmatrix},  \quad b_n' = \begin{pmatrix}
  0 & a_{n-1}' \\
  -I_{n-1} & 0
\end{pmatrix},\quad \ b_n'^{-1} = \begin{pmatrix}
  0 & I_{n-1} \\
  a_{n-1}'^{-1} & 0
\end{pmatrix}.
$$
(Observe that here the exponent $-1$ is just a notation and it
does not correspond to the inverse in algebraic sense.) Put
$A_n'=aa_n', A_n'^{-1}=aa_n'^{-1}, B_n'=bb_n'$ and $B_n'^{-1} =
bb_n'^{-1}$. Then
$$
\Delta_1' = A_1' + A_1'^{-1}+B_1' +B_1'^{-1} = \begin{pmatrix}
  0 & 2b \\
  -2b & 0
\end{pmatrix}
$$
and, for each $n\geq 2$,
$$
\Delta_n' = A_n' + A_n'^{-1}+B_n' +B_n'^{-1} = \begin{pmatrix}
  a(b_{n-1}'+b_{n-1}'^{-1}) & b(a_{n-1}'+I_{n-1}) \\
  b(a_{n-1}'^{-1}-I_{n-1}) & 0
\end{pmatrix}.
$$
\begin{prop}\label{orientedrules}
$\Delta_n'$ induces a good orientation on the Schreier graph
$\Sigma_n$ of the Basilica group.
\end{prop}
\begin{proof}
It is easy to show by induction that $\Delta_n'$ is anti-symmetric
and that each entry of $\Delta_n'$ coincides, up to the sign, with
the corresponding entry of the adjacency matrix $\Delta_n$ of
$\Sigma_n$, where loops are deleted. We also prove the assertion
about the orientation by induction. For $n=1,2$ we have $\Delta_1'
=
\begin{pmatrix}
  0 & 2b \\
  -2b & 0
\end{pmatrix}$ and $\Delta_2' = \begin{pmatrix}
  0 & 2a & 2b & 0 \\
  -2a & 0 & 0 & 2b \\
  -2b & 0 & 0 & 0 \\
  0 & -2b & 0 & 0
\end{pmatrix}$, which correspond to\unitlength=0,4mm
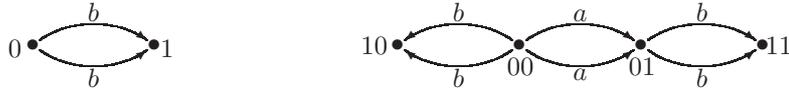
\begin{center}
\begin{picture}(300,30)
\letvertex A=(30,15)\letvertex B=(70,15)\letvertex C=(150,15)\letvertex D=(190,15)
\letvertex E=(230,15)\letvertex F=(270,15)

\put(22,11){0}\put(72,11){1}\put(138,12){10}\put(186,5){00}\put(226,5){01}\put(271,12){11}

\drawvertex(A){$\bullet$}\drawvertex(B){$\bullet$}
\drawvertex(C){$\bullet$}\drawvertex(D){$\bullet$}
\drawvertex(E){$\bullet$}\drawvertex(F){$\bullet$}

\drawcurvededge(A,B){$b$} \drawcurvededge[b](A,B){$b$}
\drawcurvededge(D,C){$b$} \drawcurvededge[b](D,C){$b$}
\drawcurvededge(D,E){$a$} \drawcurvededge[b](D,E){$a$}
\drawcurvededge(E,F){$b$} \drawcurvededge[b](E,F){$b$}
\end{picture}
\end{center}
Now look at the second block $b(a_{n-1}'+I_{n-1}) =
b\begin{pmatrix}
  b_{n-2}'+I_{n-2} & 0 \\
  0 & 2I_{n-2}
\end{pmatrix}$ of $\Delta'_n$. The matrix $2bI_{n-2}$ corresponds to the $2$-cycles
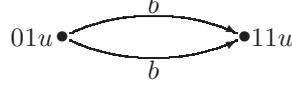
\begin{center}
\begin{picture}(200,40)

\letvertex A=(70,25)\letvertex B=(130,25)

\put(53,22){$01u$}\put(132,22){$11u$}

\drawvertex(A){$\bullet$}\drawvertex(B){$\bullet$}

\drawcurvededge(A,B){$b$}\drawcurvededge[b](A,B){$b$}
\end{picture}
\end{center}
which come from the $a$-loops of $\Sigma_{n-1}$ centered at $1u$
and so they are $2^{n-2}$. The block $b(b_{n-2}'+I_{n-2})$
corresponds to the $b$-cycles of length $2^k\geq 4$. These cycles
come from the $a$-cycles of level $n-1$ but they have double
length. In particular, $bb_{n-2}'$ corresponds to the $b$-cycles
at level $n-2$ (well oriented by induction), that correspond to
the $a$-cycles at level $n-1$ with the same good orientation given
by the substitutional rule
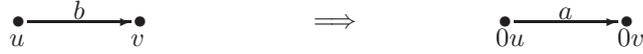
\begin{center}
\begin{picture}(400,20)
\letvertex D=(100,15)\letvertex E=(140,15)\letvertex F=(260,15)\letvertex G=(300,15)

\put(197,13){$\Longrightarrow$}

\put(97,7){$u$}\put(137,7){$v$}\put(257,7){$0u$}\put(297,7){$0v$}

\drawvertex(D){$\bullet$}
\drawvertex(E){$\bullet$}\drawvertex(F){$\bullet$}
\drawvertex(G){$\bullet$}

\drawedge(D,E){$b$} \drawedge(F,G){$a$}
\end{picture}
\end{center}
Such a cycle labeled $a$ with vertices $u_1, u_2, \ldots,
u_{2^{k-1}}$ (of length $2^{k-1}\geq 2$) at level $n-1$ gives rise
to a $b$-cycle of length $2^{k}$ in $\Sigma_n$ following the third
substitutional rule. In this new cycle, by induction, there is an
odd number of clockwise oriented edges of type
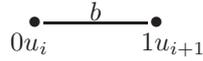
\begin{center}
\begin{picture}(400,20)
\letvertex D=(180,15)\letvertex E=(220,15)
\put(172,6){$0u_i$}\put(215,6){$1u_{i+1}$}
\drawvertex(D){$\bullet$} \drawvertex(E){$\bullet$}
\drawundirectededge(D,E){$b$}
\end{picture}
\end{center}
All the other edges have the same orientation (given by the matrix
$bI_{n-2}$). Since these edges are in even number, this implies
that the $b$-cycle is well oriented. A similar argument can be
developed for edges labeled by $a$ and this completes the proof.
\end{proof}

\begin{teo}
The partition function of the dimer model on the Schreier graph
$\Sigma_n$ of the Basilica group is
$$
\Phi_n(a,b)=
  \begin{cases}
    2^{\frac{2^n+1}{3}}b^{2^{n-1}} & n \text{ odd}, \\
    2^{\frac{2^n+2}{3}}b^{2^{n-1}} & n \text{ even}.
  \end{cases}
$$
\end{teo}

\begin{proof}
For small $n$ the assertion can be directly verified. Suppose now
$n\geq 5$. Observe that we have $\det(\Delta_n') =
b^{2^n}(\det(a_{n-1}'+I_{n-1}))^2$, since the matrices
$a_{n-1}'+I_{n-1}$ and $a_{n-1}'^{-1}-I_{n-1}$ have the same
determinant. Let us prove by induction on $n$ that, for every
$n\geq 5$, $(\det(a_{n-1}'+I_{n-1}))^2 = 2^{2^{n-1}}\cdot
2^{2l'}$, where $l'$ is the number of cycles labeled by $b$ in
$\Sigma_n$ having length greater than 2. One can verify by direct
computation that $\det(\Delta_5') = 2^{22}$ and $\det(\Delta_6') =
2^{44}$. Now
\begin{eqnarray*}
\det(\Delta_n') &=& (\det(a_{n-1}'+I_{n-1}))^2 = \left|
\begin{matrix}
  b_{n-2}'+I_{n-2} & 0 \\
  0 & 2I_{n-2}
\end{matrix}\right|^2 = 2^{2^{n-1}}\cdot (\det(b_{n-2}'+I_{n-2}))^2\\ &=& 2^{2^{n-1}}\left|
\begin{matrix}
  I_{n-3} & a_{n-3}' \\
  -I_{n-3} & I_{n-3}
\end{matrix}\right|^2  = 2^{2^{n-1}} (\det(a_{n-3}'+I_{n-3}))^2 =
2^{2^{n-1}}\cdot 2^{2^{n-3}}\cdot 2^{2l''},
\end{eqnarray*}
where the last equality follows by induction and $l''$ is the
number of $b$-cycles in $\Sigma_{n-2}$ having length greater than
2. Now observe that $l''$ is also equal to the number of
$a$-cycles of length greater than 2 in $\Sigma_{n-1}$ but also to
the number of $b$-cycles of length greater than 4 in $\Sigma_n$.
We already proved that $b^n_4 = b^{n-2}_2 = 2^{n-4}$ and so
$2^{2^{n-3}} = 2^{2b^n_4}$. Similarly $2^{2^{n-1}} = 2^{2b^n_2}$.
Then one gets the assertion by using computations made in
Proposition \ref{computationcycles}.
\end{proof}

\begin{cor}
The thermodynamic limit is $\frac{1}{3}\log 2 + \frac{1}{2}\log
b$. In particular, the entropy of absorption of diatomic molecules
per site is $\frac{1}{3}\log 2$.
\end{cor}


\section{The dimer model on the Schreier graphs of the Hanoi Towers group
$H^{(3)}$}\label{SECTION4}

We present here a more sophisticated example of dimers computation on Schreier graphs -- the Schreier graphs
of the action of the Hanoi Towers group $H^{(3)}$ on the rooted ternary tree.

\subsection{The Schreier graphs}

The group $H^{(3)}$ is generated by the automorphisms of the
ternary rooted tree having the following self-similar form
\cite{hanoi}:
$$
a= (01)(id,id,a)\qquad b=(02)(id,b,id) \qquad c=(12)(c,id,id),
$$
where $(01), (02)$ and $(12)$ are transpositions in $Sym(3)$.
Observe that $a,b,c$ are involutions. The associated Schreier
graphs are self-similar in the sense of \cite{wagner2}, that is,
each $\Sigma_{n+1}$ contains three copies of $\Sigma_n$ glued
together by three edges. These graphs can be recursively constructed via the following
substitutional rules \cite{hanoi}: \unitlength=0,4mm
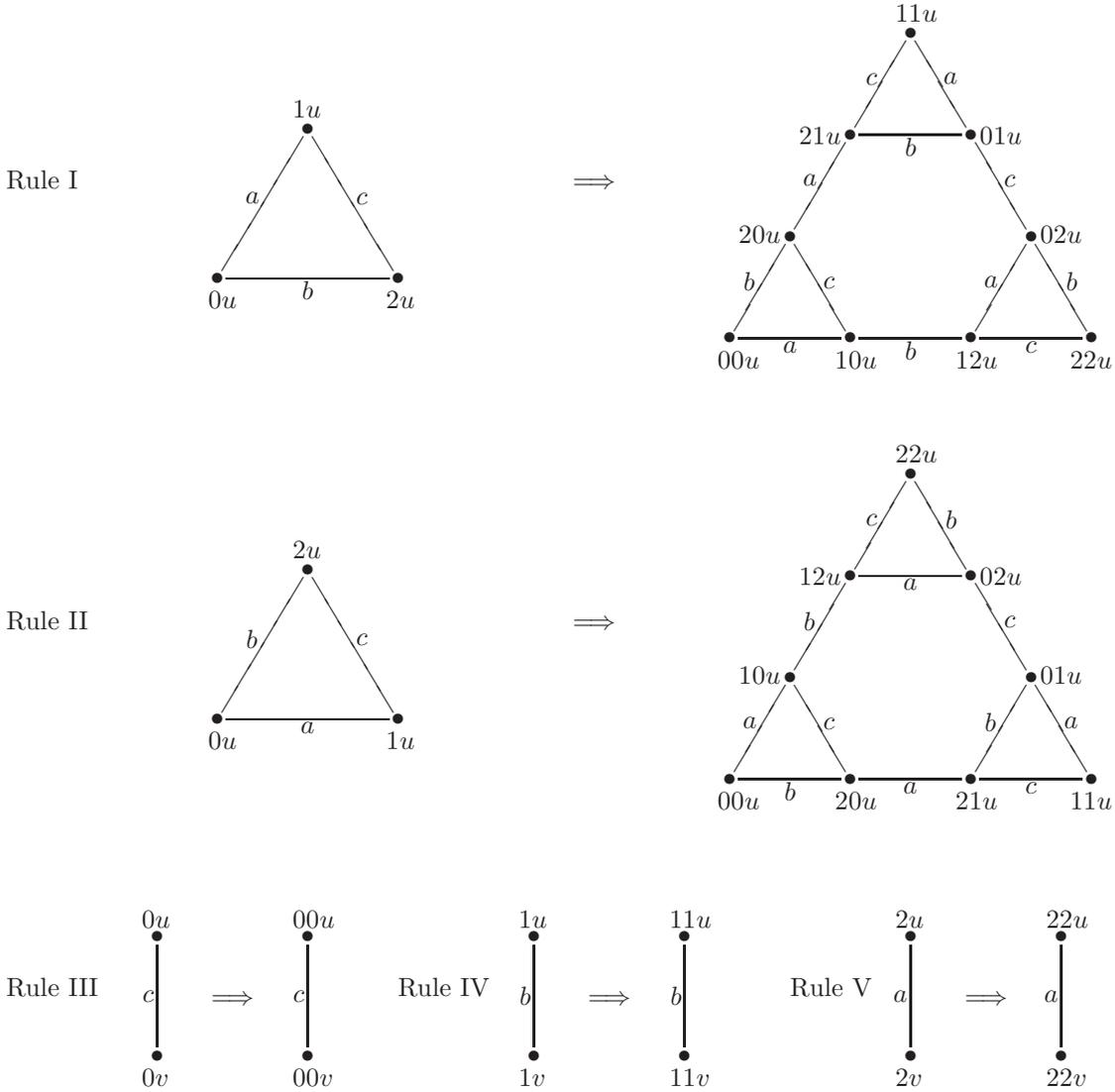
\begin{center}
\begin{picture}(400,115)
\letvertex A=(240,10)\letvertex B=(260,44)
\letvertex C=(280,78)\letvertex D=(300,112)
\letvertex E=(320,78)\letvertex F=(340,44)
\letvertex G=(360,10)\letvertex H=(320,10)\letvertex I=(280,10)

\letvertex L=(70,30)\letvertex M=(130,30)
\letvertex N=(100,80)

\put(236,0){$00u$}\put(243,42){$20u$}\put(263,75){$21u$}
\put(295,116){$11u$}\put(323,75){$01u$}\put(343,42){$02u$}\put(353,0){$22u$}
\put(315,0){$12u$}\put(275,0){$10u$}

\put(67,20){$0u$}\put(126,20){$2u$}\put(95,84){$1u$}\put(188,60){$\Longrightarrow$}
\put(0,60){Rule I}

\drawvertex(A){$\bullet$}\drawvertex(B){$\bullet$}
\drawvertex(C){$\bullet$}\drawvertex(D){$\bullet$}
\drawvertex(E){$\bullet$}\drawvertex(F){$\bullet$}
\drawvertex(G){$\bullet$}\drawvertex(H){$\bullet$}
\drawvertex(I){$\bullet$}
\drawundirectededge(A,B){$b$}\drawundirectededge(B,C){$a$}\drawundirectededge(C,D){$c$}
\drawundirectededge(D,E){$a$}\drawundirectededge(E,C){$b$}\drawundirectededge(E,F){$c$}\drawundirectededge(F,G){$b$}
\drawundirectededge(B,I){$c$}\drawundirectededge(H,F){$a$}\drawundirectededge(H,I){$b$}
\drawundirectededge(I,A){$a$}\drawundirectededge(G,H){$c$}

\drawvertex(L){$\bullet$}
\drawvertex(M){$\bullet$}\drawvertex(N){$\bullet$}
\drawundirectededge(M,L){$b$}\drawundirectededge(N,M){$c$}\drawundirectededge(L,N){$a$}
\end{picture}
\end{center}

\begin{center}
\begin{picture}(400,135)
\letvertex A=(240,10)\letvertex B=(260,44)
\letvertex C=(280,78)\letvertex D=(300,112)
\letvertex E=(320,78)\letvertex F=(340,44)
\letvertex G=(360,10)\letvertex H=(320,10)\letvertex I=(280,10)

\letvertex L=(70,30)\letvertex M=(130,30)
\letvertex N=(100,80)

\put(236,0){$00u$}\put(243,42){$10u$}\put(263,75){$12u$}
\put(295,116){$22u$}\put(323,75){$02u$}\put(343,42){$01u$}\put(353,0){$11u$}
\put(315,0){$21u$}\put(275,0){$20u$}

\put(67,20){$0u$}\put(126,20){$1u$}\put(95,84){$2u$}\put(188,60){$\Longrightarrow$}
\put(0,60){Rule II}
\drawvertex(A){$\bullet$}\drawvertex(B){$\bullet$}
\drawvertex(C){$\bullet$}\drawvertex(D){$\bullet$}
\drawvertex(E){$\bullet$}\drawvertex(F){$\bullet$}
\drawvertex(G){$\bullet$}\drawvertex(H){$\bullet$}
\drawvertex(I){$\bullet$}
\drawundirectededge(A,B){$a$}\drawundirectededge(B,C){$b$}\drawundirectededge(C,D){$c$}
\drawundirectededge(D,E){$b$}\drawundirectededge(E,C){$a$}\drawundirectededge(E,F){$c$}\drawundirectededge(F,G){$a$}
\drawundirectededge(B,I){$c$}\drawundirectededge(H,F){$b$}\drawundirectededge(H,I){$a$}
\drawundirectededge(I,A){$b$}\drawundirectededge(G,H){$c$}

\drawvertex(L){$\bullet$}
\drawvertex(M){$\bullet$}\drawvertex(N){$\bullet$}
\drawundirectededge(M,L){$a$}\drawundirectededge(N,M){$c$}\drawundirectededge(L,N){$b$}
\end{picture}
\end{center}
\begin{center}
\begin{picture}(400,80)
\letvertex A=(50,10)\letvertex B=(100,10)
\letvertex C=(175,10)\letvertex D=(225,10)
\letvertex E=(300,10)\letvertex F=(350,10)
\letvertex G=(50,50)\letvertex H=(100,50)
\letvertex I=(175,50)\letvertex L=(225,50)
\letvertex M=(300,50)\letvertex N=(350,50)

\put(45,53){$0u$}\put(45,0){$0v$}
\put(95,0){$00v$}\put(95,53){$00u$}\put(170,0){$1v$}\put(170,53){$1u$}\put(220,0){$11v$}
\put(220,53){$11u$}\put(295,0){$2v$}\put(295,53){$2u$}\put(345,0){$22v$}\put(345,53){$22u$}

\put(68,27){$\Longrightarrow$}\put(193,27){$\Longrightarrow$}\put(318,27){$\Longrightarrow$}
\put(0,30){Rule III} \put(130,30){Rule IV} \put(260,30){Rule V}

\drawvertex(A){$\bullet$}\drawvertex(B){$\bullet$}
\drawvertex(C){$\bullet$}\drawvertex(D){$\bullet$}
\drawvertex(E){$\bullet$}\drawvertex(F){$\bullet$}
\drawvertex(G){$\bullet$}\drawvertex(H){$\bullet$}
\drawvertex(I){$\bullet$}\drawvertex(L){$\bullet$}
\drawvertex(M){$\bullet$}\drawvertex(N){$\bullet$}

\drawundirectededge(A,G){$c$}\drawundirectededge(B,H){$c$}\drawundirectededge(C,I){$b$}
\drawundirectededge(D,L){$b$}\drawundirectededge(E,M){$a$}\drawundirectededge(F,N){$a$}
\end{picture}
\end{center}
The starting point is the Schreier graph $\Sigma_1$ of the first
level.\unitlength=0,3mm
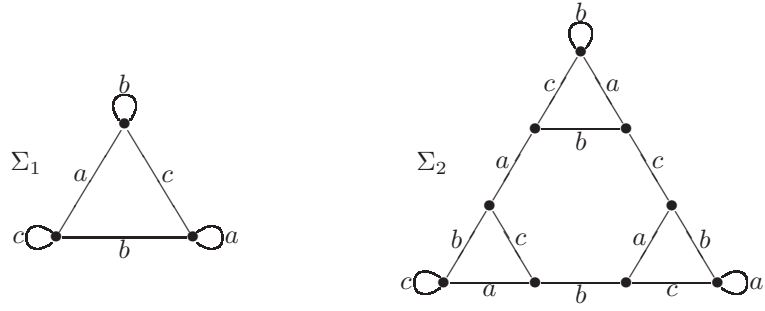
\begin{center}
\begin{picture}(400,125)
\letvertex A=(240,10)\letvertex B=(260,44)
\letvertex C=(280,78)\letvertex D=(300,112)
\letvertex E=(320,78)\letvertex F=(340,44)
\letvertex G=(360,10)\letvertex H=(320,10)\letvertex I=(280,10)

\letvertex L=(70,30)\letvertex M=(130,30)
\letvertex N=(100,80)

\put(228,60){$\Sigma_2$} \put(50,60){$\Sigma_1$}

\drawvertex(A){$\bullet$}\drawvertex(B){$\bullet$}
\drawvertex(C){$\bullet$}\drawvertex(D){$\bullet$}
\drawvertex(E){$\bullet$}\drawvertex(F){$\bullet$}
\drawvertex(G){$\bullet$}\drawvertex(H){$\bullet$}
\drawvertex(I){$\bullet$}
\drawundirectededge(A,B){$b$}\drawundirectededge(B,C){$a$}\drawundirectededge(C,D){$c$}
\drawundirectededge(D,E){$a$}\drawundirectededge(E,C){$b$}\drawundirectededge(E,F){$c$}\drawundirectededge(F,G){$b$}
\drawundirectededge(B,I){$c$}\drawundirectededge(H,F){$a$}\drawundirectededge(H,I){$b$}
\drawundirectededge(I,A){$a$}\drawundirectededge(G,H){$c$}

\drawvertex(L){$\bullet$}
\drawvertex(M){$\bullet$}\drawvertex(N){$\bullet$}
\drawundirectededge(M,L){$b$}\drawundirectededge(N,M){$c$}\drawundirectededge(L,N){$a$}

\drawundirectedloop[l](A){$c$}\drawundirectedloop(D){$b$}\drawundirectedloop[r](G){$a$}\drawundirectedloop[r](M){$a$}

\drawundirectedloop(N){$b$}\drawundirectedloop[l](L){$c$}
\end{picture}
\end{center}
In fact, the substitutional rules determine not only the graphs $\Sigma_n$ but the graphs together with a particular embedding in the plane. Throughout the paper we will always consider the graphs embedded in the plane, as drawn on the Figures, up to translations.
Observe that, for each $n\geq 1$, the graph $\Sigma_n$ has three
loops, at the vertices $0^n,1^n$ and $2^n$, labeled by $c,b$ and
$a$, respectively. Moreover, these are the only loops in
$\Sigma_n$. Since the number of vertices of the Schreier graph $\Sigma_n$ is
$3^n$ (and so an odd number), we let a dimer covering of
$\Sigma_n$ cover either zero or two outmost vertices:
we will consider covered by a loop the vertices not covered by any dimer. For this reason we do not erase the loops in this example.\\
The two subsections below correspond to the calculation of the
dimers on Hanoi Schreier graphs by two different methods:
combinatorial (Section \ref{SECTIONCOMBINATORIAL}), using the
self-similar structure of the graph; and via the Kasteleyn theory
(Section \ref{2210}), using self-similarity of the group $H^{(3)}$
in the construction of the oriented adjacency matrix.

\subsection{A combinatorial approach}\label{SECTIONCOMBINATORIAL}

There are four possible dimer configurations on $\Sigma_1$:
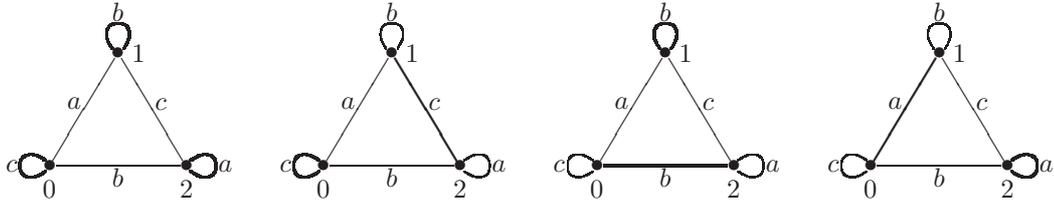
\begin{center}
\begin{picture}(480,90)
\letvertex A=(30,20)\letvertex B=(90,20)
\letvertex C=(60,70)\letvertex D=(150,20)
\letvertex E=(210,20)\letvertex F=(180,70)
\letvertex G=(270,20)\letvertex H=(330,20)\letvertex I=(300,70)
\letvertex L=(390,20)\letvertex M=(450,20)
\letvertex N=(420,70)

\put(27,6){0}\put(87,6){2}\put(66,66){1}

\put(147,6){0}\put(207,6){2}\put(186,66){1}

\put(267,6){0}\put(327,6){2}\put(306,66){1}

\put(387,6){0}\put(447,6){2}\put(426,66){1}

\drawvertex(A){$\bullet$}\drawvertex(B){$\bullet$}
\drawvertex(C){$\bullet$}\drawvertex(D){$\bullet$}
\drawvertex(E){$\bullet$}\drawvertex(F){$\bullet$}
\drawvertex(G){$\bullet$}\drawvertex(H){$\bullet$}
\drawvertex(I){$\bullet$}\drawvertex(L){$\bullet$}
\drawvertex(M){$\bullet$}\drawvertex(N){$\bullet$}
\thicklines

\drawundirectedloop[l](A){$c$}\drawundirectedloop[r](B){$a$}\drawundirectedloop(C){$b$}\drawundirectedloop[l](D){$c$}\drawundirectedloop(I){$b$}
\drawundirectedloop[r](M){$a$}

\drawundirectededge(F,E){$c$}\drawundirectededge(H,G){$b$}\drawundirectededge(L,N){$a$}

\thinlines
\drawundirectededge(B,A){$b$}\drawundirectededge(C,B){$c$}\drawundirectededge(A,C){$a$}\drawundirectededge(D,F){$a$}
\drawundirectededge(E,D){$b$}\drawundirectededge(I,H){$c$}\drawundirectededge(G,I){$a$}
\drawundirectededge(N,M){$c$}\drawundirectededge(M,L){$b$}

\drawundirectedloop[r](E){$a$}\drawundirectedloop(F){$b$}\drawundirectedloop[r](H){$a$}\drawundirectedloop[l](G){$c$}
\drawundirectedloop(N){$b$}\drawundirectedloop[l](L){$c$}
\end{picture}
\end{center}
At level $2$, we have eight possible dimer configurations:
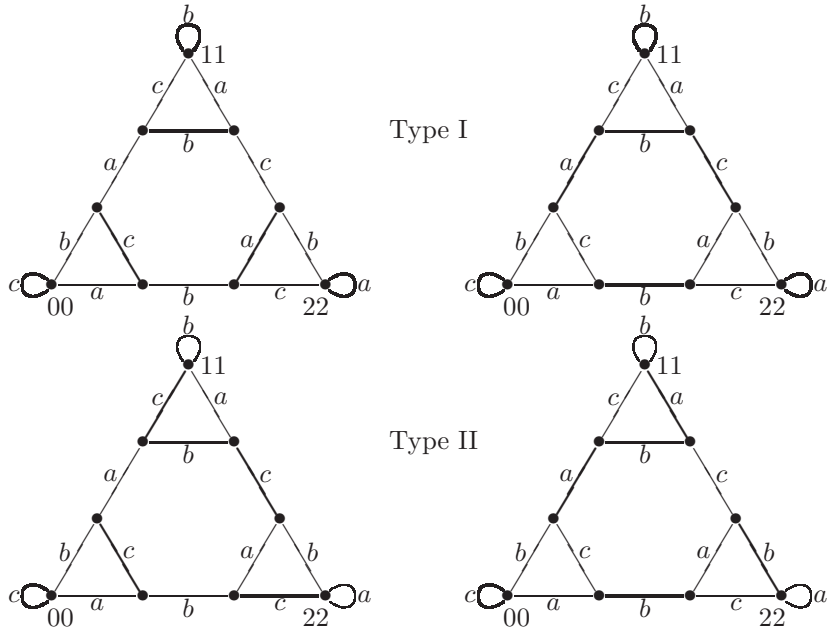
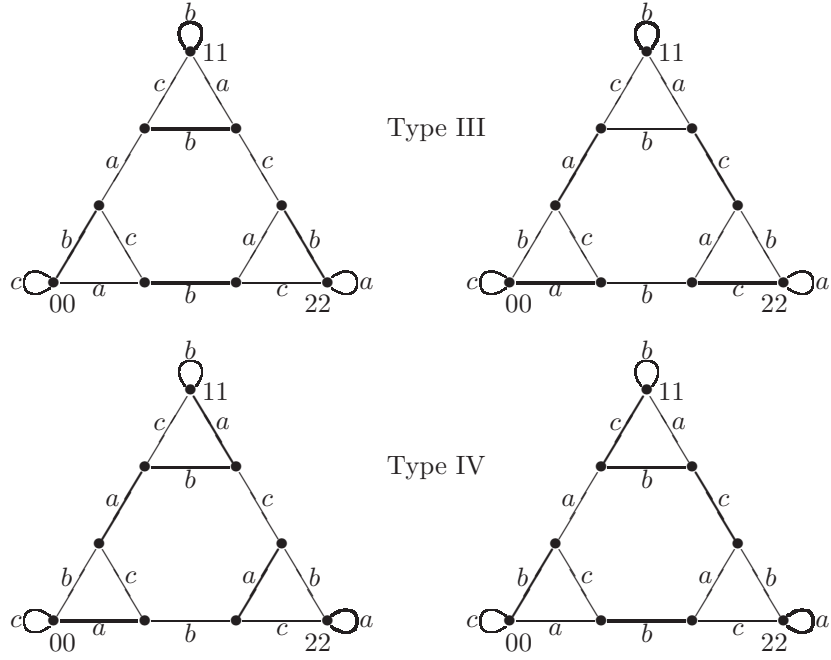
\begin{center}
\begin{picture}(400,125)
\letvertex A=(40,10)\letvertex B=(60,44)
\letvertex C=(80,78)\letvertex D=(100,112)
\letvertex E=(120,78)\letvertex F=(140,44)
\letvertex G=(160,10)\letvertex H=(120,10)\letvertex I=(80,10)
\put(38,-3){$00$}\put(150,-3){22}\put(105,108){11}

\drawvertex(A){$\bullet$}\drawvertex(B){$\bullet$}
\drawvertex(C){$\bullet$}\drawvertex(D){$\bullet$}
\drawvertex(E){$\bullet$}\drawvertex(F){$\bullet$}
\drawvertex(G){$\bullet$}\drawvertex(H){$\bullet$}
\drawvertex(I){$\bullet$}

\put(188,75){Type I}

\thicklines
\drawundirectedloop[l](A){$c$}\drawundirectedloop(D){$b$}\drawundirectedloop[r](G){$a$}
\drawundirectededge(E,C){$b$}\drawundirectededge(B,I){$c$}\drawundirectededge(H,F){$a$}

\thinlines
\drawundirectededge(A,B){$b$}\drawundirectededge(B,C){$a$}\drawundirectededge(C,D){$c$}
\drawundirectededge(D,E){$a$}\drawundirectededge(E,F){$c$}
\drawundirectededge(F,G){$b$}
\drawundirectededge(H,I){$b$}\drawundirectededge(I,A){$a$}\drawundirectededge(G,H){$c$}

\letvertex a=(240,10)\letvertex b=(260,44)
\letvertex c=(280,78)\letvertex d=(300,112)
\letvertex e=(320,78)\letvertex f=(340,44)
\letvertex g=(360,10)\letvertex h=(320,10)\letvertex i=(280,10)
\put(238,-3){$00$}\put(350,-3){22}\put(305,108){11}
\drawvertex(a){$\bullet$}\drawvertex(b){$\bullet$}
\drawvertex(c){$\bullet$}\drawvertex(d){$\bullet$}
\drawvertex(e){$\bullet$}\drawvertex(f){$\bullet$}
\drawvertex(g){$\bullet$}\drawvertex(h){$\bullet$}
\drawvertex(i){$\bullet$}

\thicklines

\drawundirectedloop[l](a){$c$}\drawundirectedloop(d){$b$}\drawundirectedloop[r](g){$a$}
\drawundirectededge(b,c){$a$}\drawundirectededge(e,f){$c$}
\drawundirectededge(h,i){$b$}

\thinlines
\drawundirectededge(a,b){$b$}\drawundirectededge(c,d){$c$}
\drawundirectededge(d,e){$a$}\drawundirectededge(e,c){$b$}
\drawundirectededge(f,g){$b$}\drawundirectededge(b,i){$c$}\drawundirectededge(h,f){$a$}
\drawundirectededge(i,a){$a$}\drawundirectededge(g,h){$c$}
\end{picture}
\end{center}

\begin{center}
\begin{picture}(400,125)
\letvertex A=(40,10)\letvertex B=(60,44)
\letvertex C=(80,78)\letvertex D=(100,112)
\letvertex E=(120,78)\letvertex F=(140,44)
\letvertex G=(160,10)\letvertex H=(120,10)\letvertex I=(80,10)

\drawvertex(A){$\bullet$}\drawvertex(B){$\bullet$}
\drawvertex(C){$\bullet$}\drawvertex(D){$\bullet$}
\drawvertex(E){$\bullet$}\drawvertex(F){$\bullet$}
\drawvertex(G){$\bullet$}\drawvertex(H){$\bullet$}
\drawvertex(I){$\bullet$}

\put(188,75){Type II}
\put(38,-3){$00$}\put(150,-3){22}\put(105,108){11}
\put(238,-3){$00$}\put(350,-3){22}\put(305,108){11}

 \thicklines
\drawundirectededge(C,D){$c$}\drawundirectedloop[l](A){$c$}
\drawundirectededge(E,F){$c$}\drawundirectededge(B,I){$c$}\drawundirectededge(G,H){$c$}

\thinlines
\drawundirectededge(A,B){$b$}\drawundirectededge(B,C){$a$}
\drawundirectededge(D,E){$a$} \drawundirectededge(F,G){$b$}
\drawundirectededge(H,I){$b$}\drawundirectededge(I,A){$a$}
\drawundirectedloop(D){$b$}\drawundirectedloop[r](G){$a$}
\drawundirectededge(E,C){$b$}\drawundirectededge(H,F){$a$}

\letvertex a=(240,10)\letvertex b=(260,44)
\letvertex c=(280,78)\letvertex d=(300,112)
\letvertex e=(320,78)\letvertex f=(340,44)
\letvertex g=(360,10)\letvertex h=(320,10)\letvertex i=(280,10)

\drawvertex(a){$\bullet$}\drawvertex(b){$\bullet$}
\drawvertex(c){$\bullet$}\drawvertex(d){$\bullet$}
\drawvertex(e){$\bullet$}\drawvertex(f){$\bullet$}
\drawvertex(g){$\bullet$}\drawvertex(h){$\bullet$}
\drawvertex(i){$\bullet$}

\thicklines
\drawundirectedloop[l](a){$c$}\drawundirectededge(d,e){$a$}
\drawundirectededge(b,c){$a$}\drawundirectededge(f,g){$b$}
\drawundirectededge(h,i){$b$}

\thinlines
\drawundirectededge(a,b){$b$}\drawundirectededge(c,d){$c$}
\drawundirectededge(e,c){$b$}
\drawundirectededge(b,i){$c$}\drawundirectededge(h,f){$a$}
\drawundirectededge(i,a){$a$}\drawundirectededge(g,h){$c$}
\drawundirectedloop(d){$b$}\drawundirectedloop[r](g){$a$}
\drawundirectededge(e,f){$c$}
\end{picture}
\end{center}

\begin{center}
\begin{picture}(400,125)
\letvertex A=(40,10)\letvertex B=(60,44)
\letvertex C=(80,78)\letvertex D=(100,112)
\letvertex E=(120,78)\letvertex F=(140,44)
\letvertex G=(160,10)\letvertex H=(120,10)\letvertex I=(80,10)

\drawvertex(A){$\bullet$}\drawvertex(B){$\bullet$}
\drawvertex(C){$\bullet$}\drawvertex(D){$\bullet$}
\drawvertex(E){$\bullet$}\drawvertex(F){$\bullet$}
\drawvertex(G){$\bullet$}\drawvertex(H){$\bullet$}
\drawvertex(I){$\bullet$}

\put(186,75){Type III}
\put(38,-3){$00$}\put(150,-3){22}\put(105,108){11}
\put(238,-3){$00$}\put(350,-3){22}\put(305,108){11} \thicklines
\drawundirectededge(E,C){$b$}
\drawundirectedloop(D){$b$}\drawundirectededge(A,B){$b$}
\drawundirectededge(H,I){$b$}\drawundirectededge(F,G){$b$}

\thinlines \drawundirectededge(B,C){$a$}
\drawundirectededge(D,E){$a$} \drawundirectededge(I,A){$a$}
\drawundirectedloop[r](G){$a$} \drawundirectededge(H,F){$a$}
\drawundirectededge(C,D){$c$}\drawundirectedloop[l](A){$c$}
\drawundirectededge(E,F){$c$}\drawundirectededge(B,I){$c$}\drawundirectededge(G,H){$c$}

\letvertex a=(240,10)\letvertex b=(260,44)
\letvertex c=(280,78)\letvertex d=(300,112)
\letvertex e=(320,78)\letvertex f=(340,44)
\letvertex g=(360,10)\letvertex h=(320,10)\letvertex i=(280,10)

\drawvertex(a){$\bullet$}\drawvertex(b){$\bullet$}
\drawvertex(c){$\bullet$}\drawvertex(d){$\bullet$}
\drawvertex(e){$\bullet$}\drawvertex(f){$\bullet$}
\drawvertex(g){$\bullet$}\drawvertex(h){$\bullet$}
\drawvertex(i){$\bullet$}

\thicklines
\drawundirectedloop(d){$b$}\drawundirectededge(b,c){$a$}
\drawundirectededge(e,f){$c$}\drawundirectededge(i,a){$a$}\drawundirectededge(g,h){$c$}

\thinlines
\drawundirectededge(a,b){$b$}\drawundirectededge(c,d){$c$}
\drawundirectededge(e,c){$b$}
\drawundirectededge(b,i){$c$}\drawundirectededge(h,f){$a$}
\drawundirectedloop[r](g){$a$}
\drawundirectedloop[l](a){$c$}\drawundirectededge(d,e){$a$}
\drawundirectededge(f,g){$b$} \drawundirectededge(h,i){$b$}
\end{picture}
\end{center}

\begin{center}
\begin{picture}(400,125)
\letvertex A=(40,10)\letvertex B=(60,44)
\letvertex C=(80,78)\letvertex D=(100,112)
\letvertex E=(120,78)\letvertex F=(140,44)
\letvertex G=(160,10)\letvertex H=(120,10)\letvertex I=(80,10)

\drawvertex(A){$\bullet$}\drawvertex(B){$\bullet$}
\drawvertex(C){$\bullet$}\drawvertex(D){$\bullet$}
\drawvertex(E){$\bullet$}\drawvertex(F){$\bullet$}
\drawvertex(G){$\bullet$}\drawvertex(H){$\bullet$}
\drawvertex(I){$\bullet$}
\put(38,-3){$00$}\put(150,-3){22}\put(105,108){11}
\put(238,-3){$00$}\put(350,-3){22}\put(305,108){11}
\put(186,75){Type IV}

\thicklines
\drawundirectedloop[r](G){$a$}\drawundirectededge(D,E){$a$}
\drawundirectededge(B,C){$a$}\drawundirectededge(H,F){$a$}\drawundirectededge(I,A){$a$}

\thinlines
\drawundirectededge(E,C){$b$}\drawundirectedloop(D){$b$}\drawundirectededge(A,B){$b$}
\drawundirectededge(H,I){$b$}\drawundirectededge(F,G){$b$}
\drawundirectededge(C,D){$c$}\drawundirectedloop[l](A){$c$}
\drawundirectededge(E,F){$c$}\drawundirectededge(B,I){$c$}\drawundirectededge(G,H){$c$}

\letvertex a=(240,10)\letvertex b=(260,44)
\letvertex c=(280,78)\letvertex d=(300,112)
\letvertex e=(320,78)\letvertex f=(340,44)
\letvertex g=(360,10)\letvertex h=(320,10)\letvertex i=(280,10)

\drawvertex(a){$\bullet$}\drawvertex(b){$\bullet$}
\drawvertex(c){$\bullet$}\drawvertex(d){$\bullet$}
\drawvertex(e){$\bullet$}\drawvertex(f){$\bullet$}
\drawvertex(g){$\bullet$}\drawvertex(h){$\bullet$}
\drawvertex(i){$\bullet$}

\thicklines
\drawundirectedloop[r](g){$a$}\drawundirectededge(c,d){$c$}
\drawundirectededge(e,f){$c$}\drawundirectededge(a,b){$b$}\drawundirectededge(h,i){$b$}

\thinlines
\drawundirectededge(e,c){$b$}\drawundirectedloop(d){$b$}\drawundirectededge(b,c){$a$}
\drawundirectededge(i,a){$a$}\drawundirectededge(g,h){$c$}
\drawundirectededge(b,i){$c$}\drawundirectededge(h,f){$a$}
\drawundirectedloop[l](a){$c$}\drawundirectededge(d,e){$a$}
\drawundirectededge(f,g){$b$}
\end{picture}
\end{center}
More generally, we will say that a dimer covering is of type I if
it contains all the three loops, of type II if it only contains
the leftmost loop (at vertex $0^n$), of type III if it only
contains the upmost loop (at vertex $1^n$) and of type IV if it
only contains the rightmost loop (at vertex $2^n$).\\ \indent For
$\Sigma_n$, $n\geq 1 $, let us denote by $\Phi^i_n(a,b,c)$ the
partition function of the dimer coverings of type $i$, for $i=$ I,
II, III, IV, so that
$\Phi_n=\Phi_n^I+\Phi_n^{II}+\Phi_n^{III}+\Phi_n^{IV}$. In what
follows we omit the variables $a,b,c$ in the partition functions.

\begin{teo}\label{numerohanoi}
The functions $\{\Phi_n\}$, $n=1,...,4$, satisfy the system of equations
\begin{eqnarray}\label{generalsystem}
\begin{cases}
\Phi^I_{n+1} = \left(\Phi^I_n\right)^3\cdot\frac{1}{abc} + \Phi_n^{II}\Phi_n^{III}\Phi_n^{IV} \\
\Phi^{II}_{n+1} =
\left(\Phi_n^{II}\right)^3\cdot\frac{1}{c}+\Phi_n^I\Phi_n^{III}\Phi_n^{IV}\cdot\frac{1}{ab}\\
\Phi^{III}_{n+1}
=\left(\Phi_n^{III}\right)^3\cdot\frac{1}{b}+\Phi_n^I\Phi_n^{II}\Phi_n^{IV}\cdot\frac{1}{ac}\\
\Phi^{IV}_{n+1} = \left(\Phi_n^{IV}\right)^3\cdot\frac{1}{a}
+\Phi_n^I\Phi_n^{II}\Phi_n^{III}\cdot\frac{1}{bc}
\end{cases},
\end{eqnarray}
with the initial conditions
\begin{eqnarray*}
\begin{cases}
\Phi^I_{1} = abc \\
\Phi^{II}_{1} = c^2 \\
\Phi^{III}_{1} = b^2 \\
\Phi^{IV}_{1} = a^2
\end{cases}.
\end{eqnarray*}
\end{teo}
\begin{proof}
We prove the assertion by induction on $n$. The initial conditions
can be easily verified. The induction step follows from the
substitutional rules. More precisely, in $\Sigma_n$ we have a copy
$T_0$ of $\Sigma_{n-1}$ reflected with respect to the bisector of
the angle with vertex  $0^{n-1}$; a copy $T_1$ of $\Sigma_{n-1}$
reflected with respect to the bisector of the angle with vertex
$1^{n-1}$; a copy $T_2$ of $\Sigma_{n-1}$ reflected with respect
to the bisector of the angle with vertex $2^{n-1}$.
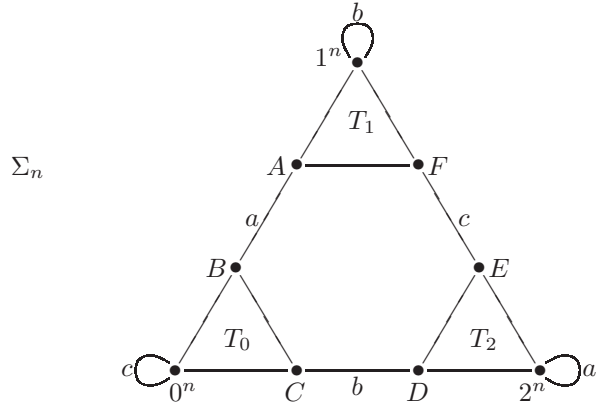
\unitlength=0,4mm
\begin{center}
\begin{picture}(400,125)
\letvertex A=(140,10)\letvertex B=(160,44)
\letvertex C=(180,78)\letvertex D=(200,112)
\letvertex E=(220,78)\letvertex F=(240,44)
\letvertex G=(260,10)\letvertex H=(220,10)\letvertex I=(180,10)

\put(86,75){$\Sigma_n$} \put(156,18){$T_0$} \put(237,18){$T_2$}
\put(197,90){$T_1$}

\put(186,111){$1^n$} \put(170,75){$A$} \put(223,75){$F$}
\put(150,41){$B$} \put(243,41){$E$} \put(138,0){$0^n$}
\put(176,0){$C$} \put(216,0){$D$} \put(253,0){$2^n$}

\put(163,58){$a$} \put(233,58){$c$} \put(198,2){$b$}

\drawvertex(A){$\bullet$}\drawvertex(B){$\bullet$}
\drawvertex(C){$\bullet$}\drawvertex(D){$\bullet$}
\drawvertex(E){$\bullet$}\drawvertex(F){$\bullet$}
\drawvertex(G){$\bullet$}\drawvertex(H){$\bullet$}
\drawvertex(I){$\bullet$}

\drawundirectededge(D,E){}
\drawundirectededge(B,C){}\drawundirectededge(H,F){}\drawundirectededge(I,A){}
\drawundirectededge(E,C){}\drawundirectededge(A,B){}
\drawundirectededge(H,I){}\drawundirectededge(F,G){}
\drawundirectededge(C,D){}
\drawundirectededge(E,F){}\drawundirectededge(B,I){}\drawundirectededge(G,H){}

\drawundirectedloop[r](G){$a$} \drawundirectedloop(D){$b$}
\drawundirectedloop[l](A){$c$}

\end{picture}
\end{center}
Using this information, let us analyze the dimer coverings of $\Sigma_n$ as constructed from dimer coverings
of $T_i$, $i=0,1,2$, that can be in turn interpreted as dimer coverings of $\Sigma_{n-1}$. \\
\indent First suppose that the dimer covering of $\Sigma_n$
contains only one loop, and without loss of generality assume it
is at $0^n$. There are then two possible cases.
\begin{itemize}
\item  The copy of $\Sigma_{n-1}$ corresponding to $T_0$ was covered using three loops (type I). Then the covering
of $\Sigma_n$ must cover the edges connecting
 $T_0$ to  $T_1$ and $T_2$, (labeled $a$ and $b$
respectively). So in the copy of $\Sigma_{n-1}$ corresponding to $T_1$ we had  a dimer covering with
only a loop in $A$ (type IV, by reflection) and in $T_2$ a
covering with only a loop in $D$ (type III, by reflection). These
coverings cover vertices $E$ and $F$ and so the edge labeled $c$
joining $E$ and $F$ does not belong to the cover of $\Sigma_n$. We
describe this situation as:
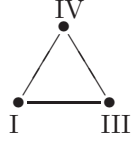
\begin{center}
\begin{picture}(200,43)
\letvertex A=(85,10)\letvertex B=(115,10)
\letvertex C=(100,35)

\put(82,0){I} \put(112,0){III} \put(97,38){IV}

\drawvertex(A){$\bullet$}\drawvertex(B){$\bullet$}
\drawvertex(C){$\bullet$}\drawundirectededge(A,B){}
\drawundirectededge(B,C){}\drawundirectededge(C,A){}
\end{picture}
\end{center}
\item The copy of $\Sigma_{n-1}$ corresponding to $T_0$ was covered with only one loop in $0^{n-1}$ (type II), so
that the vertices $B$ and $C$ are covered. This implies that the
edges joining $A,B$ and $C,D$, labeled $a$ and $b$ respectively
are not covered in $\Sigma_n$. So in the copy of $\Sigma_{n-1}$
corresponding to $T_1$ there was no loop at $1^{n-1}$ (or $A$),
which implies that in it only the loop at $F$ was covered (type
II, by reflection). Similarly the covering of the copy of
$\Sigma_{n-1}$ corresponding to $T_2$ could only contain a loop in
$E$ (type II, by reflection). Consequently, in $\Sigma_n$, the
edge joining $E$ and $F$ (labeled $c$) must belong to the
covering. Schematically this situation can be described as
follows.
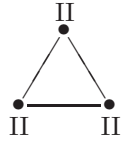
\begin{center}
\begin{picture}(200,43)
\letvertex A=(85,10)\letvertex B=(115,10)
\letvertex C=(100,35)

\put(82,0){II} \put(112,0){II} \put(97,38){II}

\drawvertex(A){$\bullet$}\drawvertex(B){$\bullet$}
\drawvertex(C){$\bullet$}

\drawundirectededge(A,B){}
\drawundirectededge(B,C){}\drawundirectededge(C,A){}
\end{picture}
\end{center}
\end{itemize}
\indent Now suppose that the covering of $\Sigma_n$ contains loops at
$0^n,1^n,2^n$. There are again two possible cases.
\begin{itemize}
\item We have on $T_0$ a dimer covering with three loops (type I), so that the
dimer covering of $\Sigma_n$ must use the edges joining the
copy $T_0$ to the copies $T_1$ and $T_2$. So in the copy of $\Sigma_{n-1}$ corresponding to $T_1$ (and similarly for $T_2$) we necessarily have a covering with three loops (type I).
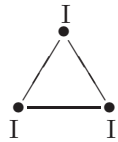
\begin{center}
\begin{picture}(200,43)
\letvertex A=(85,10)\letvertex B=(115,10)
\letvertex C=(100,35)

\put(82,0){I} \put(114,0){I} \put(99,38){I}

\drawvertex(A){$\bullet$}\drawvertex(B){$\bullet$}
\drawvertex(C){$\bullet$}

\drawundirectededge(A,B){}
\drawundirectededge(B,C){}\drawundirectededge(C,A){}
\end{picture}
\end{center}
\item We have on $T_0$ a dimer covering with only one loop in $0^{n-1}$ (type II), so
that the vertices $B$ and $C$ are covered. This implies that the
edges joining $A,B$ and $C,D$ are not covered in $\Sigma_n$. So in
$T_1$ we cannot have a loop at $A$, which implies that the
corresponding copy of $\Sigma_{n-1}$ was covered with only a loop
in $1^{n-1}$ (type III). Similarly $T_2$ was  covered with only a
loop in $2^{n-1}$ (type IV).
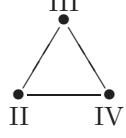
\begin{center}
\begin{picture}(200,43)
\letvertex A=(85,10)\letvertex B=(115,10)
\letvertex C=(100,35)

\put(82,0){II} \put(110,0){IV} \put(95,38){III}

\drawvertex(A){$\bullet$}\drawvertex(B){$\bullet$}
\drawvertex(C){$\bullet$}

\drawundirectededge(A,B){}
\drawundirectededge(B,C){}\drawundirectededge(C,A){}
\end{picture}
\end{center}
\end{itemize}
The claim follows.
\end{proof}

\begin{cor}\label{corollaryentropyHanoi}
For each $n\geq 1$, the number of dimer coverings of type I, II,
III, IV of the Schreier graph $\Sigma_n$ is the same and it is
equal to $2^{\frac{3^{n-1}-1}{2}}$. Hence, the number of dimer
coverings of the Schreier graph $\Sigma_n$ is equal to
$2^{\frac{3^{n-1}+3}{2}}$. The entropy of absorption of diatomic
molecules per site is $\frac{1}{6}\log 2$.
\end{cor}

\begin{proof}
By construction in the proof of Theorem \ref{numerohanoi}
we see that the number of configurations of type I of
$\Sigma_n$ is given by $2h^3$, where $h$ is (by inductive hypothesis) the common value of the
number of configurations of type I, II, III, IV of $\Sigma_{n-1}$, equal to $2^{\frac{3^{n-2}-1}{2}}$. So the
number of configurations of type I in $\Sigma_n$ is equal to
$$
2\cdot 2^{\frac{3(3^{n-2}-1)}{2}} = 2^{\frac{3^{n-1}-1}{2}}.
$$
Clearly the same count holds for the coverings of type II, III and
IV of $\Sigma_n$, and this completes the proof.
\end{proof}
\begin{os}\rm Analogues of Theorem 3.1. can be deduced for the dimers partition function on Schreier graphs of any
self-similar (automata) group with bounded activity (generated by
a bounded automaton). Indeed, V. Nekrashevych in \cite{volo}
 introduces an inductive procedure (that he
calls \lq\lq inflation") that produces a sequence of graphs that
differ from the Schreier graphs only in a bounded number of edges
(see also \cite{bondarenkothesis}, where this construction is used
to study growth of infinite Schreier graphs). This inductive
procedure allows to describe the partition function for the dimer
model on them by writing a system of recursive equations as in
(\ref{generalsystem}).
\end{os}
Unfortunately, we were not able to find an explicit solution of
the system \eqref{generalsystem}. We will come back to these
equations in Section \ref{statistiques} where we study some
statistics for the dimer coverings on $\Sigma_n$. Meanwhile, in
the next Subsection \ref{2210}, we will attempt to compute the
partition function in a different way, using the Kasteleyn theory.\\
\indent In the rest of this subsection we present the solution of
the system \eqref{generalsystem} in the particular case, when all
the weights on the edges are put to be equal, and deduce the
thermodynamic limit.

\begin{prop}\label{a=b=c}
The partition function for $a=b=c$ is
$$
\Phi_n(a,a,a)= 2^{\frac{3^{n-1}-1}{2}}\cdot
a^{\frac{3^n+1}{2}}(a+3).
$$
In this case, the thermodynamic limit is $ \frac{1}{6}\log 2 +
\frac{1}{2}\log a$.
\end{prop}

\begin{proof} By putting $a=b=c$, the system (\ref{generalsystem}) reduces to
\begin{eqnarray}\label{systemalla}
\begin{cases}
\Phi^I_{n+1} = \left(\Phi^I_n\right)^3\cdot\frac{1}{a^3} + (\Phi_n^{II})^3\\
\Phi^{II}_{n+1} =
\left(\Phi_n^{II}\right)^3\cdot\frac{1}{a}+\Phi_n^I(\Phi_n^{II})^2\cdot\frac{1}{a^2}
\end{cases}
\end{eqnarray}
with initial conditions $\Phi^I_{1} = a^3$ and $\Phi^{II}_{1} =
a^2$, since $\Phi_n^{II} = \Phi_n^{III} = \Phi_n^{IV}$. One can
prove by induction that $\frac{\Phi_n^{I}}{\Phi_n^{II}}=a$, so
that the first equation in (\ref{systemalla}) becomes
$$
\Phi_{n+1}^{I}=\frac{2(\Phi_n^{I})^3}{a^3},
$$
giving
$$
\begin{cases}
\Phi^I_n = 2^{\frac{3^{n-1}-1}{2}}\cdot a^{\frac{3^n+3}{2}}\\
\Phi^{II}_n = 2^{\frac{3^{n-1}-1}{2}}\cdot a^{\frac{3^n+1}{2}}.
\end{cases}
$$
The partition function is then obtained as
$$
\Phi_n=\Phi_n^I+3\Phi_n^{II} = 2^{\frac{3^{n-1}-1}{2}}\cdot
a^{\frac{3^n+1}{2}}(a+3).
$$
The thermodynamic limit is
$$
\lim_{n\to \infty}\frac{\log(\Phi_n)}{3^n} = \frac{1}{6}\log 2 +
\frac{1}{2}\log a.
$$
Finally, by putting $a=1$, we get the entropy of absorption of
diatomic molecules as in Corollary \ref{corollaryentropyHanoi}.
\end{proof}
Next, we deduce the
existence of the thermodynamic limit of the function
$\Phi_n(1,1,c)$, for $c\geq 1$. Observe that $\Phi_n(1,1,c) =
\Phi_n^I(1,1,c) + \Phi_n^{II}(1,1,c) + 2\Phi_n^{III}(1,1,c)$,
since $\Phi_n^{III}(1,1,c) = \Phi_n^{IV}(1,1,c)$. A similar
argument holds for the functions $\Phi_n(a,1,1)$ and
$\Phi_n(1,b,1)$.

\begin{prop}\label{COROLLARYEXISTENCE}
For every $c\geq 1$, the thermodynamic limit $\lim_{n\to
\infty}\frac{\log(\Phi_n(1,1,c))}{3^n}$ exists.
\end{prop}

\begin{proof}
It is clear that, for every $c\geq 1$, the sequence
$\varepsilon_n:=\frac{\log(\Phi_n(1,1,c))}{3^n}$ is positive. We
claim that $\varepsilon_n$ is decreasing. We have:
\begin{eqnarray*}
\frac{\varepsilon_{n+1}}{\varepsilon_n}&=&
\frac{\log(\Phi_{n+1}(1,1,c))}{3\cdot\log(\Phi_n(1,1,c))}\\
&=&\frac{3\cdot\log(\Phi_n(1,1,c))+\log\left(\frac{\Phi_{n+1}(1,1,c)}
{\Phi_n(1,1,c)^3}\right)}{3\cdot\log(\Phi_n(1,1,c))}\\
&=& 1+ \frac{\log\left(\frac{\Phi_{n+1}(1,1,c)}
{\Phi_n(1,1,c)^3}\right)}{3\cdot\log(\Phi_n(1,1,c))}.
\end{eqnarray*}

Since $\log(\Phi_n(1,1,c))>0$ for every $c\geq 1$, it suffices to
prove that $\frac{\Phi_{n+1}(1,1,c)}{\Phi_n(1,1,c)^3}$ is less or
equal to 1 for every $c\geq 1$.

\begin{eqnarray*}
\frac{\Phi_{n+1}(1,1,c)}{\Phi_n(1,1,c)^3}&=&
\frac{\Phi^{I}_{n+1}(1,1,c)+\Phi^{II}_{n+1}(1,1,c)+2\Phi^{III}_{n+1}(1,1,c)}{(\Phi^{I}_{n}(1,1,c)+\Phi^{II}_{n}(1,1,c)+2\Phi^{III}_{n}(1,1,c))^3}\\
&=& \frac{\left(\frac{(\Phi^{I}_{n})^3}{c}+
\Phi^{II}_{n}(\Phi^{III}_{n})^2+\frac{
(\Phi^{II}_{n})^3}{c}+\Phi^{I}_{n}(\Phi^{III}_{n})^2 +
2(\Phi^{III}_{n})^3+2\frac{\Phi^{I}_{n}\Phi^{II}_{n}\Phi^{III}_{n}}{c}\right)}{\left(
\Phi^{I}_{n}+\Phi^{II}_{n}+2\Phi^{III}_{n} \right)^3}\\
&\leq & \frac{\left((\Phi^{I}_{n})^3+
\Phi^{II}_{n}(\Phi^{III}_{n})^2+
(\Phi^{II}_{n})^3+\Phi^{I}_{n}(\Phi^{III}_{n})^2 +
2(\Phi^{III}_{n})^3+2\Phi^{I}_{n}\Phi^{II}_{n}\Phi^{III}_{n}\right)}{\left(
\Phi^{I}_{n}+\Phi^{II}_{n}+2\Phi^{III}_{n} \right)^3}\\
&\leq& 1
\end{eqnarray*}

This implies the existence of $\lim_{n\to
\infty}\frac{\log(\Phi_n(1,1,c))}{3^n}$.
\end{proof}

\subsection{Partition function by Kasteleyn method}\label{2210}

In order to define a matrix inducing a good orientation, in the
sense of Kasteleyn, on the Schreier graphs of $H^{(3)}$, we
introduce the matrices
$$
a_1 = \begin{pmatrix}
  0 & 1 & 0 \\
  -1 & 0 & 0 \\
  0 & 0 & 1
\end{pmatrix}, \qquad b_1 =\begin{pmatrix}
  0 & 0 & -1 \\
  0 & 1 & 0 \\
  1 & 0 & 0
\end{pmatrix}, \qquad  c_1 =  \begin{pmatrix}
  1 & 0 & 0 \\
  0 & 0 & 1 \\
  0 & -1 & 0
\end{pmatrix}.
$$
Then, for every $n$ even, we put
$$
a_n = \begin{pmatrix}
  0 & -I_{n-1} & 0 \\
  I_{n-1} & 0 & 0 \\
  0 & 0 & a_{n-1}
\end{pmatrix}, \qquad b_n =\begin{pmatrix}
  0 & 0 & I_{n-1} \\
  0 & b_{n-1} & 0 \\
  -I_{n-1} & 0 & 0
\end{pmatrix}, \qquad c_n =  \begin{pmatrix}
  c_{n-1} & 0 & 0 \\
  0 & 0 & -I_{n-1} \\
  0 & I_{n-1} & 0
\end{pmatrix},
$$
and for every $n>1$ odd, we put
$$
a_n = \begin{pmatrix}
  0 & I_{n-1} & 0 \\
  -I_{n-1} & 0 & 0 \\
  0 & 0 & a_{n-1}
\end{pmatrix}, \qquad b_n =\begin{pmatrix}
  0 & 0 & -I_{n-1} \\
  0 & b_{n-1} & 0 \\
  I_{n-1} & 0 & 0
\end{pmatrix}, \qquad c_n =  \begin{pmatrix}
  c_{n-1} & 0 & 0 \\
  0 & 0 & I_{n-1} \\
  0 & -I_{n-1} & 0
\end{pmatrix},
$$
where $a_n, b_n,c_n$ and $I_n$ are square matrices of size $3^n$.
Now we put $A_n=aa_n, B_n=bb_n$ and $C_n=cc_n$ for each $n\geq 1$
and define $\Delta_n=A_n+B_n+C_n$, so that
$$
\Delta_1 = \begin{pmatrix}
  c & a & -b \\
  -a & b & c \\
  b & -c & a
\end{pmatrix}
$$
and, for each $n > 1$,
$$
\Delta_n =\begin{pmatrix}
  C_{n-1} & -aI_{n-1} & bI_{n-1} \\
  aI_{n-1} & B_{n-1} & -cI_{n-1} \\
  -bI_{n-1} & cI_{n-1} & A_{n-1}
\end{pmatrix} \  \mbox{for }n \mbox{ even, } \qquad \Delta_n =\begin{pmatrix}
  C_{n-1} & aI_{n-1} & -bI_{n-1} \\
  -aI_{n-1} & B_{n-1} & cI_{n-1} \\
  bI_{n-1} & -cI_{n-1} & A_{n-1}
\end{pmatrix} \  \mbox{for }n \mbox{ odd}.
$$
We want to prove that, for each $n\geq1$, the oriented adjacency
matrix $\Delta_n$ induces a good orientation on $\Sigma_n$. Then
we will apply Kasteleyn theory to get the partition function of
the dimer model on $\Sigma_n$. One can easily verify that also in
this case the entries of the matrix $\Delta_n$ coincide, in
absolute value, with the entries of the (unoriented weighted)
adjacency matrix of the Schreier graphs of the group. The problem
related to loops and their orientation will be discussed later.
The figures below describe the orientation induced on $\Sigma_1$
and $\Sigma_2$ by the matrices $\Delta_1$ and $\Delta_2$,
respectively.\unitlength=0,4mm
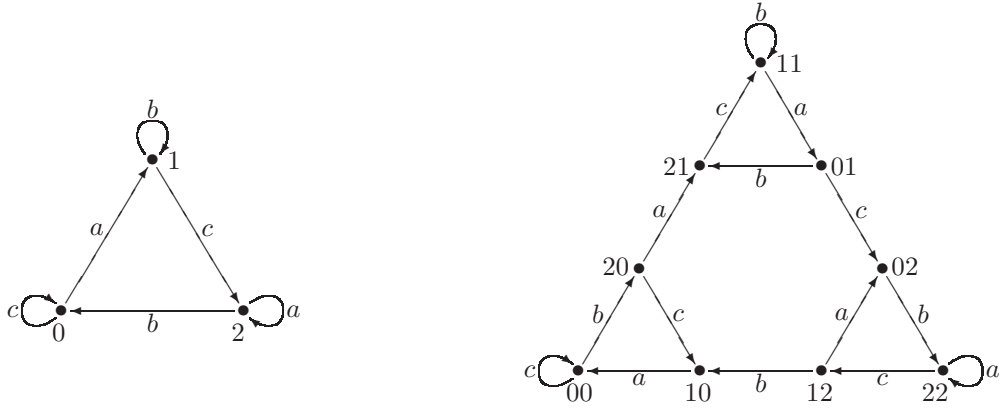
\begin{center}
\begin{picture}(400,125)
\letvertex A=(240,10)\letvertex B=(260,44)
\letvertex C=(280,78)\letvertex D=(300,112)
\letvertex E=(320,78)\letvertex F=(340,44)
\letvertex G=(360,10)\letvertex H=(320,10)\letvertex I=(280,10)

\letvertex L=(70,30)\letvertex M=(130,30)
\letvertex N=(100,80)

\put(236,0){$00$}\put(248,42){$20$}\put(268,75){$21$}
\put(305,109){$11$}\put(323,75){$01$}\put(343,42){$02$}\put(353,0){$22$}
\put(315,0){$12$}\put(275,0){$10$}

\put(67,20){$0$}\put(126,20){$2$}\put(105,77){$1$}

\drawvertex(A){$\bullet$}\drawvertex(B){$\bullet$}
\drawvertex(C){$\bullet$}\drawvertex(D){$\bullet$}
\drawvertex(E){$\bullet$}\drawvertex(F){$\bullet$}
\drawvertex(G){$\bullet$}\drawvertex(H){$\bullet$}
\drawvertex(I){$\bullet$}

\drawedge(A,B){$b$}\drawedge(B,C){$a$}\drawedge(C,D){$c$}
\drawedge(D,E){$a$}\drawedge(E,C){$b$}\drawedge(E,F){$c$}\drawedge(F,G){$b$}
\drawedge(B,I){$c$}\drawedge(H,F){$a$}\drawedge(H,I){$b$}
\drawedge(I,A){$a$}\drawedge(G,H){$c$}

\drawvertex(L){$\bullet$}
\drawvertex(M){$\bullet$}\drawvertex(N){$\bullet$}
\drawedge(M,L){$b$}\drawedge(N,M){$c$}\drawedge(L,N){$a$}

\drawloop[l](A){$c$}\drawloop(D){$b$}\drawloop[r](G){$a$}\drawloop[r](M){$a$}

\drawloop(N){$b$}\drawloop[l](L){$c$}
\end{picture}
\end{center}
\begin{prop}
For each $n\geq 1$, the matrix $\Delta_n$ induces a good
orientation on $\Sigma_n$.
\end{prop}
\begin{proof}
Observe that in $\Sigma_1$ the sequence of labels $a,b,c$ appears
in anticlockwise order. Following the substitutional rules, we
deduce that for every $n$ odd we can read in each elementary
triangle the sequence $a,b,c$ in anticlockwise order. On the other
hand, for $n$ even, the occurrences of $a,b,c$ in each elementary
triangle of $\Sigma_n$ follow a clockwise order. We prove our
claim by induction on $n$. For $n=1$, the matrix $\Delta_1$
induces on $\Sigma_1$ the orientation shown in the picture above,
so that the assertion is true for $n=1$. Now observe that, for
every $n$ odd, the blocks $\pm aI_{n-1}, \pm bI_{n-1}, \pm
cI_{n-1}$ in $\Delta_n$ ensure that each elementary triangle in
$\Sigma_n$ has the same orientation given by
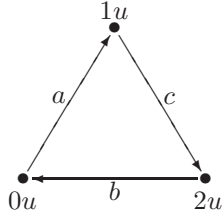
\begin{center}
\begin{picture}(400,60)
\letvertex L=(170,10)\letvertex M=(230,10)
\letvertex N=(200,60)

\put(165,0){$0u$}\put(226,0){$2u$}\put(195,63){$1u$}

\drawvertex(L){$\bullet$}
\drawvertex(M){$\bullet$}\drawvertex(N){$\bullet$}
\drawedge(M,L){$b$}\drawedge(N,M){$c$}\drawedge(L,N){$a$}
\end{picture}
\end{center}
For $n$ even, the sequence $a,b,c$ is clockwise and the blocks
$\pm aI_{n-1}, \pm bI_{n-1}, \pm cI_{n-1}$ in $\Delta_n$ ensure
that the orientation induced on the edges is clockwise as the
following picture shows:
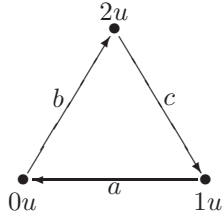
\begin{center}
\begin{picture}(400,57)
\letvertex L=(170,5)\letvertex M=(230,5)
\letvertex N=(200,55)

\put(165,-5){$0u$}\put(226,-5){$1u$}\put(195,58){$2u$}

\drawvertex(L){$\bullet$}
\drawvertex(M){$\bullet$}\drawvertex(N){$\bullet$}
\drawedge(M,L){$a$}\drawedge(N,M){$c$}\drawedge(L,N){$b$}
\end{picture}
\end{center}
So we conclude that for every $n$ all the elementary triangles of
$\Sigma_n$ are clockwise oriented. Now construct the graph
$\Sigma_{n+1}$ from $\Sigma_n$ and suppose $n$ odd (the same proof
works in the case $n$ even). Rule I gives
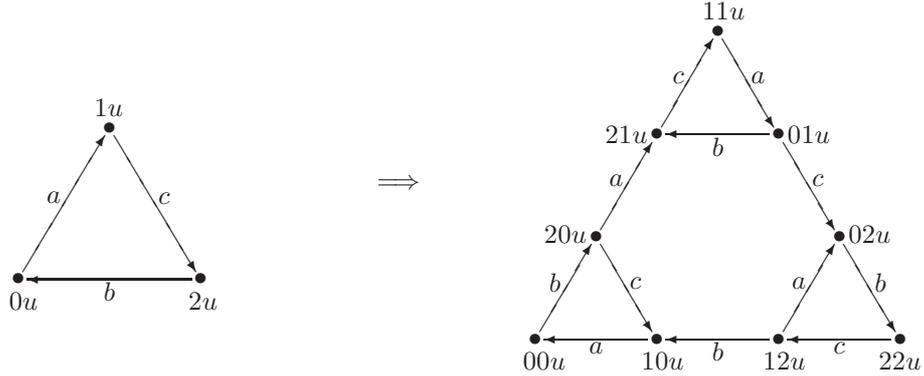
\begin{center}
\begin{picture}(400,115)
\letvertex A=(240,10)\letvertex B=(260,44)
\letvertex C=(280,78)\letvertex D=(300,112)
\letvertex E=(320,78)\letvertex F=(340,44)
\letvertex G=(360,10)\letvertex H=(320,10)\letvertex I=(280,10)

\letvertex L=(70,30)\letvertex M=(130,30)
\letvertex N=(100,80)

\put(236,0){$00u$}\put(243,42){$20u$}\put(263,75){$21u$}
\put(295,116){$11u$}\put(323,75){$01u$}\put(343,42){$02u$}\put(353,0){$22u$}
\put(315,0){$12u$}\put(275,0){$10u$}

\put(67,20){$0u$}\put(126,20){$2u$}\put(95,84){$1u$}\put(188,60){$\Longrightarrow$}

\drawvertex(A){$\bullet$}\drawvertex(B){$\bullet$}
\drawvertex(C){$\bullet$}\drawvertex(D){$\bullet$}
\drawvertex(E){$\bullet$}\drawvertex(F){$\bullet$}
\drawvertex(G){$\bullet$}\drawvertex(H){$\bullet$}
\drawvertex(I){$\bullet$}
\drawedge(A,B){$b$}\drawedge(B,C){$a$}\drawedge(C,D){$c$}
\drawedge(D,E){$a$}\drawedge(E,C){$b$}\drawedge(E,F){$c$}\drawedge(F,G){$b$}
\drawedge(B,I){$c$}\drawedge(H,F){$a$}\drawedge(H,I){$b$}
\drawedge(I,A){$a$}\drawedge(G,H){$c$}

\drawvertex(L){$\bullet$}
\drawvertex(M){$\bullet$}\drawvertex(N){$\bullet$}
\drawedge(M,L){$b$}\drawedge(N,M){$c$}\drawedge(L,N){$a$}
\end{picture}
\end{center}
In order to understand why the edges not belonging to an elementary
triangle have this orientation, we observe that the edge
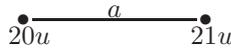
\begin{center}
\begin{picture}(400,2)
\letvertex N=(170,5)\letvertex K=(230,5)

\put(165,-3){$20u$}\put(225,-3){$21u$}

\drawvertex(N){$\bullet$} \drawvertex(K){$\bullet$}
\drawundirectededge(N,K){$a$}
\end{picture}
\end{center}
has the same orientation as the edge
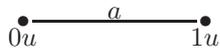
\begin{center}
\begin{picture}(400,2)
\letvertex N=(170,5)\letvertex K=(230,5)

\put(165,-3){$0u$}\put(225,-3){$1u$}

\drawvertex(N){$\bullet$} \drawvertex(K){$\bullet$}
\drawundirectededge(N,K){$a$}
\end{picture}
\end{center}
since the entry $(20u,21u)$ of the matrix $\Delta_{n+1}$ is the
same as the entry $(0u,1u)$ of the matrix $\Delta_n$. Similarly
for the other two edges joining vertices $01u,02u$ and $10u,12u$.
This implies that each elementary hexagon has a good orientation.
Now note that in $\Sigma_{n+1}$ we have $3^{n-i}$ elementary
cycles of length $3\cdot 2^i$, for each $i=0,1, \ldots, n$. We
already know that in $\Sigma_{n+1}$ all the elementary triangles
and hexagons have a good orientation. Observe that each cycle in
$\Sigma_n$ having length $k = 3\cdot 2^m$, with $m\geq 1$, gives
rise in $\Sigma_{n+1}$ to a cycle of length $2k$. In this new
cycle of $\Sigma_{n+1}$, $k$ edges join vertices starting with the
same letter and keep the same orientation as in $\Sigma_n$ (so
they are well oriented by induction); the remaining $k$ edges
belong to elementary triangles and have the form
\begin{center}
\begin{picture}(400,3)
\letvertex N=(170,5)\letvertex K=(230,5)

\put(165,-3){$xu$}\put(225,-3){$\overline{x}u$}

\drawvertex(N){$\bullet$} \drawvertex(K){$\bullet$}
\drawundirectededge(K,N){$ $}
\end{picture}
\end{center}
where $x \neq \overline{x}$ and $x, \overline{x}\in \{0,1,2\}$,
$u\in \Sigma_n$. Since the last $k$ edges belong to elementary
triangles, they are oriented in the same direction and, since $k$
is even, they give a good orientation to the cycle. The same
argument works for each elementary cycle and so the proof is
completed.
\end{proof}

The matrix $\Delta_n$ cannot be directly used to find the
partition function because it is not anti-symmetric (there are
three nonzero entries in the diagonal corresponding to loops) and
it is of odd size. Let $\Gamma_{n,c}$ be the matrix obtained from $\Delta_n$ by
deleting the row and the column indexed by $0^n$ and where the
entries $(1^n,1^n)$ and $(2^n,2^n)$ are replaced by $0$, so that
the partition function of dimer coverings of type II is given by
$c\sqrt{\det(\Gamma_{n,c})}$. Similarly, we define
$\Gamma_{n,b},\Gamma_{n,a}$ for dimer coverings of type III, IV,
respectively. Now let $\Lambda_n$ be the matrix obtained from
$\Delta_n$ by deleting the three rows and the three columns
indexed by $0^n,1^n$ and $2^n$, so that the partition function of
the dimer coverings of type I is $abc\sqrt{\det{\Lambda_n}}$. This
gives
\begin{eqnarray}\label{partitionassum}
\Phi_n(a,b,c)
=c\sqrt{\det(\Gamma_{n,c})}+b\sqrt{\det(\Gamma_{n,b})}+a\sqrt{\det(\Gamma_{n,a})}+abc\sqrt{\det{\Lambda_n}}.
\end{eqnarray}
In order to compute $\det(\Gamma_{n,c})$ (the case of
$\Gamma_{n,b},\Gamma_{n,a}$ and $\Lambda_n$ is analogous), we put
$$
a'_1=\begin{pmatrix}
   0 & a & 0 \\
   -a & 0 & 0 \\
  0 & 0 & 0
\end{pmatrix}, \qquad
b'_1=\begin{pmatrix}
  0 & 0 & -b \\
  0 & 0 & 0 \\
  b & 0 & 0
\end{pmatrix}, \qquad
c'_1=\begin{pmatrix}
  1 & 0 & 0\\
  0 & 0 & c \\
  0 & -c & 0
\end{pmatrix}.
$$
Then for every $n>1$ odd we put
$$
a'_n = \begin{pmatrix}
  0 & aI_{n-1} & 0 \\
  -aI_{n-1} & 0 & 0 \\
  0 & 0 & a'_{n-1}
\end{pmatrix}, \quad b'_n =\begin{pmatrix}
  0 & 0 & -bI_{n-1} \\
  0 & b'_{n-1} & 0 \\
  bI_{n-1} & 0 & 0
\end{pmatrix}, \quad c'_n = \begin{pmatrix}
  c'_{n-1} & 0 & 0 \\
  0 & 0 & cI_{n-1} \\
  0 & -cI_{n-1} & 0
\end{pmatrix},
$$
and for every $n$ even we put
$$
a'_n = \begin{pmatrix}
  0 & -aI_{n-1} & 0 \\
  aI_{n-1} & 0 & 0 \\
  0 & 0 & a'_{n-1}
\end{pmatrix}, \quad b'_n =\begin{pmatrix}
  0 & 0 & bI_{n-1} \\
  0 & b'_{n-1} & 0 \\
  -bI_{n-1} & 0 & 0
\end{pmatrix}, \quad c'_n = \begin{pmatrix}
  c'_{n-1} & 0 & 0 \\
  0 & 0 & -cI_{n-1} \\
  0 & cI_{n-1} & 0
\end{pmatrix}.
$$
Finally, set
$$
\overline{A}_1= \overline{B}_1=\begin{pmatrix}
   0 & 0 & 0 \\
   0 & 0 & 0 \\
  0 & 0 & 0
\end{pmatrix}, \qquad
\overline{C}_1=\begin{pmatrix}
  0 & 0 & 0 \\
  0 & 0 & c \\
  0 & -c & 0
\end{pmatrix}.
$$
Then for every $n>1$ odd we put
$$
\overline{A}_n = \begin{pmatrix}
  0 & aI_{n-1}^0 & 0 \\
  -aI_{n-1}^0 & 0 & 0 \\
  0 & 0 & a'_{n-1}
\end{pmatrix}, \qquad \overline{B}_n =\begin{pmatrix}
  0 & 0 & -bI_{n-1}^0 \\
  0 & b'_{n-1} & 0 \\
  bI_{n-1}^0 & 0 & 0
\end{pmatrix}, \qquad \overline{C}_n = c_n',
$$
and for every $n$ even we put
$$
\overline{A}_n = \begin{pmatrix}
  0 & -aI_{n-1}^0 & 0 \\
  aI_{n-1}^0 & 0 & 0 \\
  0 & 0 & a'_{n-1}
\end{pmatrix}, \qquad \overline{B}_n =\begin{pmatrix}
  0 & 0 & bI_{n-1}^0 \\
  0 & b'_{n-1} & 0 \\
  -bI_{n-1}^0 & 0 & 0
\end{pmatrix}, \qquad \overline{C}_n = c_n',
$$
with
$$
I_n^0=I_n-\begin{pmatrix}
  1 & 0 & \cdots & 0 \\
  0 & 0 & \cdots & 0 \\
  \vdots & 0 & \ddots & 0 \\
  0 & 0 & 0 & 0
\end{pmatrix}.
$$
Finally, let $\overline{\Delta}_n =
\overline{A}_n+\overline{B}_n+\overline{C}_n$ for each $n\geq 1$,
so that
$$
\overline{\Delta}_n =\begin{pmatrix}
  c'_{n-1} & aI_{n-1}^0 & -bI_{n-1}^0 \\
  -aI_{n-1}^0 & b'_{n-1} & cI_{n-1} \\
  bI_{n-1}^0 & -cI_{n-1} & a'_{n-1}
\end{pmatrix} \  \mbox{for }n \mbox{ odd}, \qquad
\overline{\Delta}_n =\begin{pmatrix}
  c'_{n-1} & -aI_{n-1}^0 & bI_{n-1}^0 \\
  aI_{n-1}^0 & b'_{n-1} & -cI_{n-1} \\
  -bI_{n-1}^0 & cI_{n-1} & a'_{n-1}
\end{pmatrix} \  \mbox{for }n \mbox{ even}.
$$
The introduction of the matrices $I_n^0$ guarantees that
$\det(\overline{\Delta}_n)=\det(\Gamma_{n,c})$, since we have
performed all the necessary cancellations in $\Delta_n$.
Geometrically this corresponds to erasing the loops rooted at the
vertices $1^n$ and $2^n$ and the edges connecting the vertex $0^n$ to
the rest of $\Sigma_n$.\\
\indent Next, we define a rational
function $F:\mathbb{R}^6\longrightarrow \mathbb{R}^6$ as follows
$$
F(x_1,x_2,x_3,x_4,x_5,x_6)=\left(x_1,x_2,x_3,
\frac{x_1x_4^3+x_2x_3x_5x_6}{x_1x_2x_3+x_4x_5x_6},\frac{x_2x_5^3+x_1x_3x_4x_6}{x_1x_2x_3+x_4x_5x_6},\frac{x_3x_6^3+x_1x_2x_4x_5}{x_1x_2x_3+x_4x_5x_6}
\right).
$$
Denote $F^{(k)}(\underline{x})$ the $k$-th iteration of the
function $F$, and $F_i$ the $i$-th projection of $F$ so that
$$
F(\underline{x})=(F_1(\underline{x}), \ldots, F_6(\underline{x})).
$$
Set
$$
F^{(k)}(a,b,c,a,b,c)=(a,b,c,a^{(k)}, b^{(k)}, c^{(k)}).
$$
\begin{teo}\label{PROPOSITIONPARTITION}
For each $n\geq 3$, the partition function $\Phi_n(a,b,c)$ of the
dimer model on the Schreier graph $\Sigma_n$ of the Hanoi Tower
group $H^{(3)}$ is
\begin{eqnarray*}
\Phi_n(a,b,c)&=& \prod_{k=0}^{n-3}\left(abc+a^{(k)}b^{(k)}c^{(k)}
\right)^{ 3^{n-k-2}} \left(abc\left(a^{(n-2)}b^{(n-2)} +
a^{(n-2)}c^{(n-2)} + b^{(n-2)}c^{(n-2)}\right.\right.\\ &+&
\left.\left. a^{(n-2)}b^{(n-2)}c^{(n-2)} + abc\right) +
a^2(a^{(n-2)})^3 + b^2(b^{(n-2)})^3 + c^2(c^{(n-2)})^3 \right).
\end{eqnarray*}
\end{teo}

\begin{proof}
We explicitly analyze the case of dimer coverings of type II. It
follows from the discussion above that
$\Phi_n^{II}(a,b,c)=c\sqrt{\det(\overline{\Delta}_n)}$. More
precisely, the factor $c$ corresponds to the label of the loop at
$0^n$, and the factor $\sqrt{\det(\overline{\Delta}_n)}$ is the
absolute value of the Pfaffian of the oriented adjacency matrix of
the graph obtained from $\Sigma_n$ by deleting the edges
connecting the vertex $0^n$ to the rest of the graph and the loops
rooted at $1^n$ and $2^n$. The cases corresponding to coverings of
type III and IV are analogous. If we expand twice the matrix
$\overline{\Delta}_n$ using the recursion formula and perform the
permutations $(17)$ and $(58)$ for both rows and columns, we get
the matrix (for $n$ odd)\footnotesize
$$
\begin{pmatrix}
  M_{11} & M_{12} \\
  M_{21} & M_{22}
\end{pmatrix}= \left(\begin{array}{cccccc|ccc}
  0 & 0 & 0 & -cI_{n-2} & -aI_{n-2} & 0 & bI_{n-2}^0 & 0 & 0 \\
  0 & 0 & -cI_{n-2} & 0 & -bI_{n-2} & 0 & 0 & aI_{n-2} & 0 \\
  0 & cI_{n-2} & 0 & 0 & 0 & aI_{n-2} & 0 & 0 & -bI_{n-2} \\
  cI_{n-2} & 0 & 0 & 0 & 0 & bI_{n-2} & -aI_{n-2}^0 & 0 & 0 \\
  aI_{n-2} & bI_{n-2} & 0 & 0 & 0 & 0 & 0 & -cI_{n-2} & 0 \\
  0 & 0 & -aI_{n-2} & -bI_{n-2} & 0 & 0 & 0 & 0 & cI_{n-2}\\
  \hline
  -bI_{n-2}^0 & 0 & 0 & aI_{n-2}^0 & 0 & 0 & c'_{n-2} & 0 & 0 \\
  0 & -aI_{n-2} & 0 & 0 & cI_{n-2} & 0 & 0 & b'_{n-2} & 0 \\
  0 & 0 & bI_{n-2} & 0 & 0 & -cI_{n-2} & 0 & 0 & a'_{n-2}
\end{array}\right).
$$\normalsize
Note that each entry is a square matrix of size $3^{n-2}$. Hence,
the Schur complement formula gives
\begin{eqnarray}\label{matrice per hanoi}
\det(\overline{\Delta}_n) &=& \det(M_{11})\cdot
\det(M_{22}-M_{21}M_{11}^{-1}M_{12})\\
&=& (2abc)^{2\cdot 3^{n-2}}\left|\begin{matrix}
  c'_{n-2} & -\frac{a^4+b^2c^2}{2abc}I_{n-2}^0 & \frac{b^4+a^2c^2}{2abc}I_{n-2}^0 \\
  \frac{a^4+b^2c^2}{2abc}I_{n-2}^0 & b_{n-2}' & -\frac{c^4+a^2b^2}{2abc} \\
  -\frac{b^4+a^2c^2}{2abc}I_{n-2}^0 & \frac{c^4+a^2b^2}{2abc} & a_{n-2}'
\end{matrix}\right| . \nonumber
\end{eqnarray}
The matrix obtained in (\ref{matrice per hanoi}) has the same
shape as $\overline{\Delta}_{n-1}$, so we can use recursion by
defining
$$
\overline{\Delta}_k(x_1,\ldots, x_6)= \begin{pmatrix}
  c_{k-1}' & -F_4(x_1, \ldots, x_6)I_{k-1}^0 & F_5(x_1, \ldots, x_6)I_{k-1}^0 \\
 F_4(x_1, \ldots, x_6)I_{k-1}^0 & b_{k-1}' & -F_6(x_1, \ldots, x_6) \\
  -F_5(x_1, \ldots, x_6)I_{k-1}^0 & F_6(x_1, \ldots, x_6) & a_{k-1}'
\end{pmatrix}.
$$
Hence, (\ref{matrice per hanoi}) becomes
\begin{eqnarray*}
\det(\overline{\Delta}_n)&=&(2abc)^{2\cdot
3^{n-2}}\cdot\det(\overline{\Delta}_{n-1}(a,b,c,a,b,c))\\
&=& (2abc)^{2\cdot 3^{n-2}}(abc+a^{(1)}b^{(1)}c^{(1)})^{2\cdot
3^{n-3}}\cdot\det\left(
\overline{\Delta}_{n-2}(F^{(1)}(a,b,c,a,b,c))\right)\\
&=& \prod_{k=0}^{n-3}\left(abc+a^{(k)}b^{(k)}c^{(k)}
\right)^{2\cdot
3^{n-k-2}}\cdot\det\left(\overline{\Delta}_2(F^{(n-3)}(a,b,c,a,b,c))
\right)\\
&=& \prod_{k=0}^{n-3}\left(abc+a^{(k)}b^{(k)}c^{(k)}
\right)^{2\cdot 3^{n-k-2}}\cdot\left( ab(a^{(n-2)} b^{(n-2)})
+c(c^{(n-2)})^3 \right)^2.
\end{eqnarray*}
A similar recurrence holds for coverings of type III and
IV. For dimer coverings of type I, by using the Schur complement again, we get
$$
\det(\Lambda_n)=\prod_{k=0}^{n-3}\left(abc+a^{(k)}b^{(k)}c^{(k)}
\right)^{2\cdot 3^{n-k-2}}\cdot\left( abc+a^{(n-2)}
b^{(n-2)}c^{(n-2)} \right)^2.
$$
Then we use \eqref{partitionassum} and the proof is completed.
\end{proof}

\begin{os}\rm
The proof above, with $a=b=c=1$, gives the
number of dimers coverings of $\Sigma_n$ that we had already computed to be
$2^{\frac{3^{n-1}+3}{2}}$ in Corollary \ref{corollaryentropyHanoi}.
\end{os}

\section{The dimer model on the Sierpi\'{n}ski gasket}\label{SECTIONSIERPINSKI}

In this section we study the dimer model on a sequence of graphs
$\{\Gamma_n\}_{n\geq 1}$ forming finite approximations to the
well-known Sierpi\'{n}ski gasket. The local limit of these graphs
is an infinite graph known as the infinite Sierpi\'{n}ski
triangle. The graphs $\Gamma_n$ are not Schreier graphs of a
self-similar group. However, they are self-similar in the sense of
\cite{wagner2}, and their structure is very similar to that of the
Schreier graphs $\Sigma_n$ of the group $H^{(3)}$, studied in the
previous section. More precisely, one can obtain $\Gamma_n$ from
$\Sigma_n$ by contracting
the edges joining two different elementary triangles.  \\
\indent One can think of a few natural ways to label the edges of
these graphs with weights of three types. The one that springs
first into mind is the \lq\lq directional" weight, where the edges
are labeled $a,b,c$ according to their direction in the graph
drawn on the plane, see the picture of $\Gamma_3$ with the
directional labeling below. Note that the labeled graph $\Gamma_n$
is obtained from the labeled graph $\Gamma_{n-1}$ by taking three
translated copies of the latter (and identifying three pairs of
corners, see the picture).
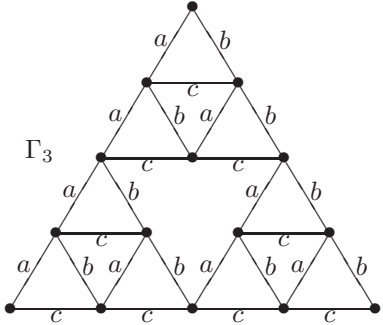
\unitlength=0,2mm\begin{center}
\begin{picture}(400,210)
\put(90,110){$\Gamma_3$}
\letvertex A=(200,210)\letvertex B=(170,160)\letvertex C=(140,110)
\letvertex D=(110,60)\letvertex E=(80,10)\letvertex F=(140,10)\letvertex G=(200,10)
\letvertex H=(260,10)\letvertex I=(320,10)
\letvertex L=(290,60)\letvertex M=(260,110)\letvertex N=(230,160)
\letvertex O=(200,110)\letvertex P=(170,60)\letvertex Q=(230,60)

\drawvertex(A){$\bullet$}\drawvertex(B){$\bullet$}
\drawvertex(C){$\bullet$}\drawvertex(D){$\bullet$}
\drawvertex(E){$\bullet$}\drawvertex(F){$\bullet$}
\drawvertex(G){$\bullet$}\drawvertex(H){$\bullet$}
\drawvertex(I){$\bullet$}\drawvertex(L){$\bullet$}\drawvertex(M){$\bullet$}
\drawvertex(N){$\bullet$}\drawvertex(O){$\bullet$}
\drawvertex(P){$\bullet$}\drawvertex(Q){$\bullet$}

\drawundirectededge(E,D){$a$} \drawundirectededge(D,C){$a$}
\drawundirectededge(C,B){$a$} \drawundirectededge(B,A){$a$}
\drawundirectededge(A,N){$b$} \drawundirectededge(N,M){$b$}
\drawundirectededge(M,L){$b$} \drawundirectededge(L,I){$b$}
\drawundirectededge(I,H){$c$} \drawundirectededge(H,G){$c$}
\drawundirectededge(G,F){$c$} \drawundirectededge(F,E){$c$}
\drawundirectededge(N,B){$c$} \drawundirectededge(O,C){$c$}
\drawundirectededge(M,O){$c$} \drawundirectededge(P,D){$c$}
\drawundirectededge(L,Q){$c$} \drawundirectededge(B,O){$b$}
\drawundirectededge(O,N){$a$} \drawundirectededge(C,P){$b$}
\drawundirectededge(P,G){$b$} \drawundirectededge(D,F){$b$}
\drawundirectededge(Q,M){$a$} \drawundirectededge(G,Q){$a$}
\drawundirectededge(H,L){$a$} \drawundirectededge(F,P){$a$}
\drawundirectededge(Q,H){$b$}
\end{picture}
\end{center}

The dimer model on $\Gamma_n$ with this labeling was previously studied in \cite{wagner1}: the authors wrote down a recursion between
levels $n$ and $n+1$, obtaining a system of equations involving the partition functions, but did not arrive at an explicit solution. Unfortunately, we were not able to compute the generating function of the dimer covers corresponding to this \lq\lq directional" weight function either.  Below we describe two other natural labelings of $\Gamma_n$ for which we were able to compute the partition functions: we refer to them as the \lq\lq Schreier\rq\rq\  labeling and the \lq\lq rotation-invariant\rq\rq\  labeling.

\subsection{The \lq\lq Schreier\rq\rq labeling}\label{firstmodelll}
In the \lq\lq Schreier\rq\rq labeling, at a given corner of
labeled $\Gamma_n$ we have a copy of labeled $\Gamma_{n-1}$
reflected with respect to the bisector of the corresponding angle,
see the picture below. \unitlength=0,2mm
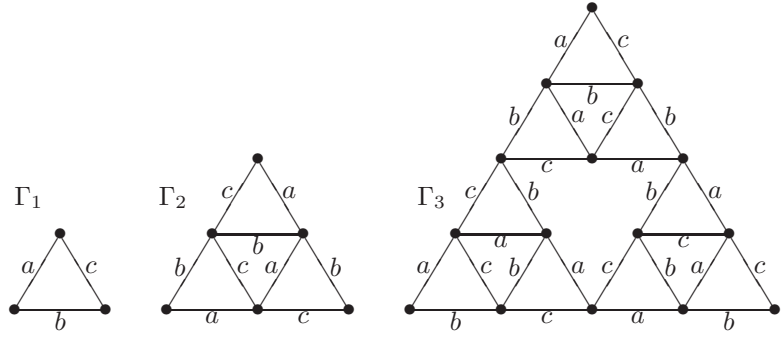
\begin{center}
\begin{picture}(500,220)
\letvertex a=(30,60)\letvertex b=(0,10)\letvertex c=(60,10)

\letvertex d=(160,110)\letvertex e=(130,60)\letvertex f=(100,10)\letvertex g=(160,10)

\letvertex h=(220,10)\letvertex i=(190,60)

\drawvertex(a){$\bullet$}\drawvertex(b){$\bullet$}
\drawvertex(c){$\bullet$}\drawvertex(d){$\bullet$}
\drawvertex(e){$\bullet$}\drawvertex(f){$\bullet$}
\drawvertex(g){$\bullet$}\drawvertex(h){$\bullet$}
\drawvertex(i){$\bullet$}

\drawundirectededge(b,a){$a$} \drawundirectededge(c,b){$b$}
\drawundirectededge(a,c){$c$} \drawundirectededge(e,d){$c$}
\drawundirectededge(f,e){$b$} \drawundirectededge(g,f){$a$}

\drawundirectededge(h,g){$c$} \drawundirectededge(i,h){$b$}
\drawundirectededge(d,i){$a$} \drawundirectededge(i,e){$b$}
\drawundirectededge(e,g){$c$} \drawundirectededge(g,i){$a$}
\put(0,80){$\Gamma_1$}

\put(265,80){$\Gamma_3$}

\put(95,80){$\Gamma_2$}
\letvertex A=(380,210)\letvertex B=(350,160)\letvertex C=(320,110)

\letvertex D=(290,60)\letvertex E=(260,10)\letvertex F=(320,10)\letvertex G=(380,10)

\letvertex H=(440,10)\letvertex I=(500,10)
\letvertex L=(470,60)\letvertex M=(440,110)\letvertex N=(410,160)

\letvertex O=(380,110)\letvertex P=(350,60)\letvertex Q=(410,60)

\drawvertex(A){$\bullet$}\drawvertex(B){$\bullet$}
\drawvertex(C){$\bullet$}\drawvertex(D){$\bullet$}
\drawvertex(E){$\bullet$}\drawvertex(F){$\bullet$}
\drawvertex(G){$\bullet$}\drawvertex(H){$\bullet$}
\drawvertex(I){$\bullet$}\drawvertex(L){$\bullet$}\drawvertex(M){$\bullet$}
\drawvertex(N){$\bullet$}\drawvertex(O){$\bullet$}
\drawvertex(P){$\bullet$}\drawvertex(Q){$\bullet$}

\drawundirectededge(E,D){$a$} \drawundirectededge(D,C){$c$}
\drawundirectededge(C,B){$b$} \drawundirectededge(B,A){$a$}
\drawundirectededge(A,N){$c$} \drawundirectededge(N,M){$b$}
\drawundirectededge(M,L){$a$} \drawundirectededge(L,I){$c$}
\drawundirectededge(I,H){$b$} \drawundirectededge(H,G){$a$}
\drawundirectededge(G,F){$c$} \drawundirectededge(F,E){$b$}
\drawundirectededge(N,B){$b$} \drawundirectededge(O,C){$c$}
\drawundirectededge(M,O){$a$} \drawundirectededge(P,D){$a$}
\drawundirectededge(L,Q){$c$} \drawundirectededge(B,O){$a$}
\drawundirectededge(O,N){$c$} \drawundirectededge(C,P){$b$}
\drawundirectededge(P,G){$a$} \drawundirectededge(D,F){$c$}
\drawundirectededge(Q,M){$b$} \drawundirectededge(G,Q){$c$}
\drawundirectededge(H,L){$a$}\drawundirectededge(F,P){$b$}
\drawundirectededge(Q,H){$b$}
\end{picture}
\end{center}
It turns out that this labeling of the graph $\Gamma_n$ can be alternatively described
by considering the labeled Schreier graph $\Sigma_n$ of
the Hanoi Towers group and then  contracting the edges
connecting copies of $\Sigma_{n-1}$ in $\Sigma_n$, hence the name \lq\lq Schreier\rq\rq labeling.\\
\indent For every $n\geq 1$, the number of vertices of $\Gamma_n$
is $|V(\Gamma_n)| = \frac{3}{2}(3^{n-1}+1)$. This implies that,
for $n$ odd, $|V(\Gamma_n)|$ is odd and so we allow dimer
coverings touching either two or none of the corners. If $n$ is
even, $|V(\Gamma_n)|$ is even, and so we allow  dimer coverings
touching either one or three corners. We say that a dimer covering
of $\Gamma_n$ is:
\begin{itemize}
\item of type $f$, if it covers no corner;
\item of type $g^{ab}$ (respectively $g^{ac},g^{bc}$), if it covers the corner of
$\Gamma_n$ where two edges $a$ and $b$ (respectively $a$ and $c$, $b$ and $c$) meet, but does not cover
any other corner;
\item of type $h^{ab}$ (respectively $h^{ac},h^{bc}$), if it does not cover the corner of
$\Gamma_n$ where two edges $a$ and $b$ (respectively $a$ and $c$, $b$ and $c$) meet, but it covers the
remaining two corners;
\item of type $t$, if it covers all three corners.
\end{itemize}
Observe that for $n$ odd we can only have configurations of type
$f$ and $h$, and for $n$ even we can only have configurations of type
$g$ and $t$.\\
\indent From now on, we will denote by
$f_n,g^{ab}_n,g^{ac}_n,g^{bc}_n,h^{ab}_n,h^{ac}_n,h^{bc}_n,t_n$
the summand in the partition function $\Phi_n(a,b,c)$ counting the
coverings of the corresponding type. For instance, for $n=1$, the
only nonzero terms are $f_1=1$, $h^{ab}_1=c$, $h^{ac}_1=b$,
$h^{bc}_1=a$, so that we get
$$
\Phi_1(a,b,c) = 1 + a + b + c.
$$
For $n=2$, the only nonzero terms are $t_2=2abc$, $g^{ab}_2=2ab$,
$g^{ac}_2=2ac$, $g^{bc}_2=2bc$, so that we get
$$
\Phi_2(a,b,c) = 2(abc+ab+ac+bc).
$$
In the following pictures, the dark bullet next to a vertex means
that this vertex is covered by a dimer in that configuration.
Since the graph $\Gamma_{n+1}$ consists of three copies of
$\Gamma_n$, a dimer covering of $\Gamma_{n+1}$ can be constructed
from three coverings of $\Gamma_n$.  For example, a configuration
of type $f$ for $\Gamma_{n+1}$ can be obtained using three
configurations of $\Gamma_n$ of type $g^{ab},g^{ac},g^{bc}$, as
the first of the following pictures shows.\unitlength=0,3mm
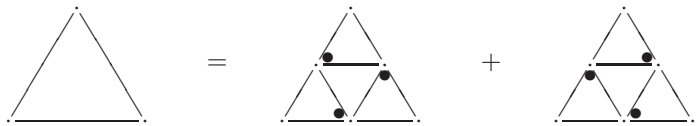
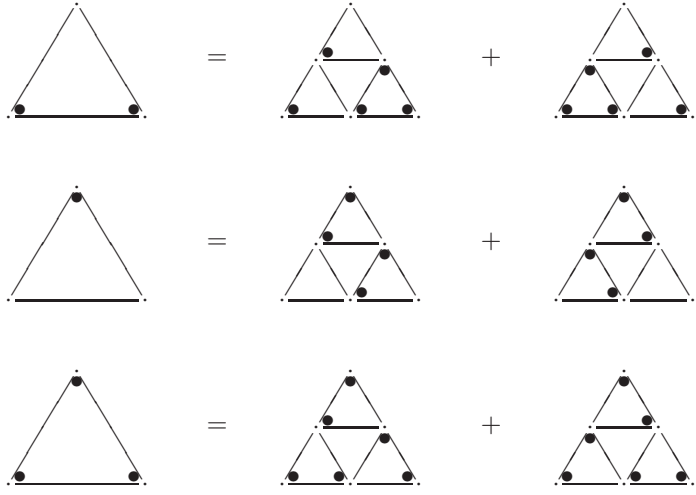
\begin{center}
\begin{picture}(400,70)
\letvertex A=(80,60)\letvertex B=(50,10)\letvertex C=(110,10)

\letvertex D=(200,60)\letvertex E=(185,35)\letvertex F=(170,10)\letvertex G=(200,10)

\letvertex H=(230,10)\letvertex I=(215,35)
\letvertex L=(320,60)\letvertex M=(305,35)\letvertex N=(290,10)

\letvertex O=(320,10)\letvertex P=(350,10)\letvertex Q=(335,35)
\letvertex R=(190,38)\letvertex S=(215,30)\letvertex T=(195,13)
\letvertex U=(330,38)\letvertex V=(305,30)\letvertex Z=(325,13)
\drawvertex(R){$\bullet$}\drawvertex(S){$\bullet$}
\drawvertex(T){$\bullet$}\drawvertex(U){$\bullet$}
\drawvertex(V){$\bullet$}\drawvertex(Z){$\bullet$}

\put(80,60){\circle*{1}}\put(50,10){\circle*{1}}\put(110,10){\circle*{1}}
\put(200,60){\circle*{1}}\put(185,35){\circle*{1}}\put(170,10){\circle*{1}}\put(200,10){\circle*{1}}
\put(230,10){\circle*{1}}\put(215,35){\circle*{1}}
\put(320,60){\circle*{1}}\put(305,35){\circle*{1}}\put(290,10){\circle*{1}}
\put(320,10){\circle*{1}}\put(350,10){\circle*{1}}\put(335,35){\circle*{1}}
\put(137,33){$=$}\put(257,33){$+$}

\drawundirectededge(A,B){} \drawundirectededge(B,C){}
\drawundirectededge(C,A){} \drawundirectededge(D,E){}
\drawundirectededge(E,F){} \drawundirectededge(F,G){}
\drawundirectededge(G,H){} \drawundirectededge(H,I){}
\drawundirectededge(I,D){} \drawundirectededge(E,G){}
\drawundirectededge(G,I){} \drawundirectededge(I,E){}
\drawundirectededge(L,M){} \drawundirectededge(M,N){}
\drawundirectededge(N,O){} \drawundirectededge(O,P){}
\drawundirectededge(P,Q){} \drawundirectededge(Q,L){}
\drawundirectededge(M,Q){} \drawundirectededge(M,O){}
\drawundirectededge(O,Q){}
\end{picture}
\end{center}
\begin{center}
\begin{picture}(400,70)

\letvertex A=(80,60)\letvertex B=(50,10)\letvertex C=(110,10)

\letvertex D=(200,60)\letvertex E=(185,35)\letvertex F=(170,10)\letvertex G=(200,10)

\letvertex H=(230,10)\letvertex I=(215,35)
\letvertex L=(320,60)\letvertex M=(305,35)\letvertex N=(290,10)

\letvertex O=(320,10)\letvertex P=(350,10)\letvertex Q=(335,35)
\letvertex u=(345,13)\letvertex r=(205,13)\letvertex Z=(315,13)\letvertex T=(175,13)\letvertex s=(225,13)
\letvertex S=(295,13)
\letvertex a=(55,13)\letvertex b=(105,13)
\letvertex V=(215,30)\letvertex t=(305,30)
\letvertex R=(190,38)\letvertex U=(330,38)

\drawvertex(a){$\bullet$}\drawvertex(b){$\bullet$}
\drawvertex(u){$\bullet$}\drawvertex(r){$\bullet$}
\drawvertex(s){$\bullet$}\drawvertex(t){$\bullet$}
\drawvertex(R){$\bullet$}\drawvertex(S){$\bullet$}
\drawvertex(T){$\bullet$}\drawvertex(U){$\bullet$}
\drawvertex(V){$\bullet$}\drawvertex(Z){$\bullet$}

\put(80,60){\circle*{1}}\put(50,10){\circle*{1}}\put(110,10){\circle*{1}}
\put(200,60){\circle*{1}}\put(185,35){\circle*{1}}\put(170,10){\circle*{1}}\put(200,10){\circle*{1}}
\put(230,10){\circle*{1}}\put(215,35){\circle*{1}}
\put(320,60){\circle*{1}}\put(305,35){\circle*{1}}\put(290,10){\circle*{1}}
\put(320,10){\circle*{1}}\put(350,10){\circle*{1}}\put(335,35){\circle*{1}}

\put(137,33){$=$}\put(257,33){$+$}

\drawundirectededge(A,B){} \drawundirectededge(B,C){}
\drawundirectededge(C,A){} \drawundirectededge(D,E){}
\drawundirectededge(E,F){} \drawundirectededge(F,G){}
\drawundirectededge(G,H){} \drawundirectededge(H,I){}
\drawundirectededge(I,D){} \drawundirectededge(E,G){}
\drawundirectededge(G,I){} \drawundirectededge(I,E){}
\drawundirectededge(L,M){} \drawundirectededge(M,N){}
\drawundirectededge(N,O){} \drawundirectededge(O,P){}
\drawundirectededge(P,Q){} \drawundirectededge(Q,L){}
\drawundirectededge(M,Q){} \drawundirectededge(M,O){}
\drawundirectededge(O,Q){}
\end{picture}
\end{center}
\begin{center}
\begin{picture}(400,70)
\letvertex A=(80,60)\letvertex B=(50,10)\letvertex C=(110,10)

\letvertex D=(200,60)\letvertex E=(185,35)\letvertex F=(170,10)\letvertex G=(200,10)

\letvertex H=(230,10)\letvertex I=(215,35)
\letvertex L=(320,60)\letvertex M=(305,35)\letvertex N=(290,10)

\letvertex O=(320,10)\letvertex P=(350,10)\letvertex Q=(335,35)

\letvertex R=(205,13)\letvertex S=(315,13)\letvertex T=(215,30)
\letvertex U=(305,30)\letvertex V=(190,38)\letvertex Z=(330,38)
\letvertex u=(200,55)\letvertex v=(320,55)\letvertex a=(80,55)

\drawvertex(a){$\bullet$}
\drawvertex(R){$\bullet$}\drawvertex(S){$\bullet$}
\drawvertex(T){$\bullet$}\drawvertex(U){$\bullet$}
\drawvertex(V){$\bullet$}\drawvertex(Z){$\bullet$}
\drawvertex(u){$\bullet$}\drawvertex(v){$\bullet$}
\put(80,60){\circle*{1}}\put(50,10){\circle*{1}}\put(110,10){\circle*{1}}
\put(200,60){\circle*{1}}\put(185,35){\circle*{1}}\put(170,10){\circle*{1}}\put(200,10){\circle*{1}}
\put(230,10){\circle*{1}}\put(215,35){\circle*{1}}
\put(320,60){\circle*{1}}\put(305,35){\circle*{1}}\put(290,10){\circle*{1}}
\put(320,10){\circle*{1}}\put(350,10){\circle*{1}}\put(335,35){\circle*{1}}

\put(137,33){$=$}\put(257,33){$+$}

\drawundirectededge(A,B){} \drawundirectededge(B,C){}
\drawundirectededge(C,A){} \drawundirectededge(D,E){}
\drawundirectededge(E,F){} \drawundirectededge(F,G){}
\drawundirectededge(G,H){} \drawundirectededge(H,I){}
\drawundirectededge(I,D){} \drawundirectededge(E,G){}
\drawundirectededge(G,I){} \drawundirectededge(I,E){}
\drawundirectededge(L,M){} \drawundirectededge(M,N){}
\drawundirectededge(N,O){} \drawundirectededge(O,P){}
\drawundirectededge(P,Q){} \drawundirectededge(Q,L){}
\drawundirectededge(M,Q){} \drawundirectededge(M,O){}
\drawundirectededge(O,Q){}
\end{picture}
\end{center}
\begin{center}
\begin{picture}(400,70)

\letvertex A=(80,60)\letvertex B=(50,10)\letvertex C=(110,10)

\letvertex D=(200,60)\letvertex E=(185,35)\letvertex F=(170,10)\letvertex G=(200,10)

\letvertex H=(230,10)\letvertex I=(215,35)
\letvertex L=(320,60)\letvertex M=(305,35)\letvertex N=(290,10)

\letvertex O=(320,10)\letvertex P=(350,10)\letvertex Q=(335,35)

\letvertex u=(55,13)\letvertex v=(105,13)\letvertex z=(80,55)
\letvertex R=(175,13)\letvertex S=(195,13)\letvertex T=(225,13)
\letvertex U=(295,13)\letvertex V=(325,13)\letvertex Z=(345,13)
\letvertex a=(305,30)\letvertex b=(215,30)\letvertex c=(190,38)
\letvertex d=(330,38)\letvertex e=(200,55)\letvertex f=(320,55)

\drawvertex(a){$\bullet$}\drawvertex(f){$\bullet$}
\drawvertex(b){$\bullet$}\drawvertex(e){$\bullet$}
\drawvertex(c){$\bullet$}\drawvertex(d){$\bullet$}
\drawvertex(R){$\bullet$}\drawvertex(S){$\bullet$}
\drawvertex(T){$\bullet$}\drawvertex(U){$\bullet$}
\drawvertex(V){$\bullet$}\drawvertex(Z){$\bullet$}
\drawvertex(u){$\bullet$}\drawvertex(v){$\bullet$}
\drawvertex(z){$\bullet$}
\put(80,60){\circle*{1}}\put(50,10){\circle*{1}}\put(110,10){\circle*{1}}
\put(200,60){\circle*{1}}\put(185,35){\circle*{1}}\put(170,10){\circle*{1}}\put(200,10){\circle*{1}}
\put(230,10){\circle*{1}}\put(215,35){\circle*{1}}
\put(320,60){\circle*{1}}\put(305,35){\circle*{1}}\put(290,10){\circle*{1}}
\put(320,10){\circle*{1}}\put(350,10){\circle*{1}}\put(335,35){\circle*{1}}

\put(137,33){$=$}\put(257,33){$+$}

\drawundirectededge(A,B){} \drawundirectededge(B,C){}
\drawundirectededge(C,A){} \drawundirectededge(D,E){}
\drawundirectededge(E,F){} \drawundirectededge(F,G){}
\drawundirectededge(G,H){} \drawundirectededge(H,I){}
\drawundirectededge(I,D){} \drawundirectededge(E,G){}
\drawundirectededge(G,I){} \drawundirectededge(I,E){}
\drawundirectededge(L,M){} \drawundirectededge(M,N){}
\drawundirectededge(N,O){} \drawundirectededge(O,P){}
\drawundirectededge(P,Q){} \drawundirectededge(Q,L){}
\drawundirectededge(M,Q){} \drawundirectededge(M,O){}
\drawundirectededge(O,Q){}
\end{picture}
\end{center}
By using these recursions and arguments similar to those in the proof
of Theorem \ref{numerohanoi}, one can show that for $n$ odd (respectively, for $n$ even), the number of coverings of types
$f,h^{ab},h^{ac},h^{bc}$ (respectively of types $t,g^{ab},g^{ac},g^{bc}$) is the same, and so is equal to the quarter of the
total number of dimer coverings of $\Gamma_n$.

\begin{teo}
For each $n$, the partition function of the dimer model on
$\Gamma_n$ is:
$$
\begin{cases}
    \Phi_n(a,b,c)= 2(4abc)^{\frac{3^{n-1}-3}{4}}(abc+ab+ac+bc)& \text{for } n \text{ even}\\
    \Phi_n(a,b,c)= (4abc)^{\frac{3^{n-1}-1}{4}}(1+a+b+c)& \text{for } n \text{ odd}.
\end{cases}
$$
Hence, the number of dimer coverings of $\Gamma_n$ is
$2^{\frac{3^{n-1}+3}{2}}$.
\end{teo}

\begin{proof}
The recursion shows that, for $n$ odd, one has:
\begin{eqnarray}\label{pari}
  \begin{cases}
  f_n = 2g^{ab}_{n-1}g^{ac}_{n-1}g^{bc}_{n-1}\\
  h^{ab}_n= 2t_{n-1}g^{ac}_{n-1}g^{bc}_{n-1}\\
  h^{ac}_n = 2t_{n-1}g^{ab}_{n-1}g^{bc}_{n-1}\\
  h^{bc}_n = 2t_{n-1}g^{ab}_{n-1}g^{ac}_{n-1}
  \end{cases}
\end{eqnarray}
Similarly, for $n$ even, one has:
\begin{eqnarray}\label{dispari}
\begin{cases}
  t_n = 2h^{ab}_{n-1}h^{ac}_{n-1}h^{bc}_{n-1}\\
  g^{ab}_n= 2f_{n-1}h^{ac}_{n-1}h^{bc}_{n-1}\\
  g^{ac}_n = 2f_{n-1}h^{ab}_{n-1}h^{bc}_{n-1}\\
  g^{bc}_n = 2f_{n-1}h^{ab}_{n-1}h^{ac}_{n-1}
  \end{cases}
\end{eqnarray}
The solutions of systems (\ref{pari}) and (\ref{dispari}), with
initial conditions $f_1=1, h^{ab}_1=c, h^{ac}_1=b, h^{bc}_1=a$,
can be computed by induction on $n$. We find:
$$
\begin{cases}
t_n=2abc(4abc)^{\frac{3^{n-1}-3}{4}} \\
g^{ab}_n=2ab(4abc)^{\frac{3^{n-1}-3}{4}}\\
g^{ac}_n=2ac(4abc)^{\frac{3^{n-1}-3}{4}}\\
g^{bc}_n=2bc(4abc)^{\frac{3^{n-1}-3}{4}}
\end{cases} \ \ \ n \mbox{ even,} \ \ \ \
\begin{cases}
f_n=(4abc)^{\frac{3^{n-1}-1}{4}} \\
h^{ab}_n=c(4abc)^{\frac{3^{n-1}-1}{4}}\\
h^{ac}_n=b(4abc)^{\frac{3^{n-1}-1}{4}}\\
h^{bc}_n=a(4abc)^{\frac{3^{n-1}-1}{4}}
\end{cases} \ \ \ n \mbox{ odd.}
$$
The assertion follows from the fact that $\Phi_n(a,b,c)= f_n +
h^{ab}_n+h^{ac}_n+h^{bc}_n$ for $n$ odd and $\Phi_n(a,b,c)=
t_n+g^{ab}_n+g^{ac}_n+g^{bc}_n$ for $n$ even. The number of dimer
coverings of $\Gamma_n$ is obtained as $\Phi_n(1,1,1)$.
\end{proof}
\begin{cor}
The thermodynamic limit is $\frac{1}{6}\log(4abc)$. In particular,
the entropy of absorption of diatomic molecules per site is
$\frac{1}{3}\log 2$.
\end{cor}
The number of dimer coverings and the value of the entropy have already appeared in \cite{taiwan}, where the
dimers on $\Gamma_n$ with the weight function constant 1 were considered.\\
\indent Note also that the number of dimer coverings found for Sierpi\'nski graphs $\Gamma_n$ coincides with
the number of dimer coverings for the Schreier graphs $\Sigma_n$ of the group $H^{(3)}$ (see Section \ref{SECTION4}).


\subsection{\lq\lq Rotation-invariant\rq\rq labeling}

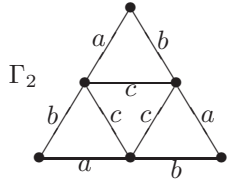
\unitlength=0,2mm
\begin{center}
\begin{picture}(400,115)
\put(120,60){$\Gamma_2$}

\letvertex D=(200,110)\letvertex E=(170,60)\letvertex F=(140,10)\letvertex G=(200,10)

\letvertex H=(260,10)\letvertex I=(230,60)

\drawvertex(D){$\bullet$}
\drawvertex(E){$\bullet$}\drawvertex(F){$\bullet$}
\drawvertex(G){$\bullet$}\drawvertex(H){$\bullet$}
\drawvertex(I){$\bullet$}

\drawundirectededge(E,D){$a$} \drawundirectededge(F,E){$b$}
\drawundirectededge(G,F){$a$}

\drawundirectededge(H,G){$b$} \drawundirectededge(I,H){$a$}
\drawundirectededge(D,I){$b$} \drawundirectededge(I,E){$c$}
\drawundirectededge(E,G){$c$} \drawundirectededge(G,I){$c$}
\end{picture}
\end{center}

This labeling is, for $n\geq 2$, invariant under rotation by
$\frac{2\pi}{3}$. For $n\geq 3$, the copy of $\Gamma_{n-1}$ at the
left (respectively upper, right) corner of $\Gamma_n$ is rotated
by $0$ (respectively $2\pi/3$, $4\pi/3$).

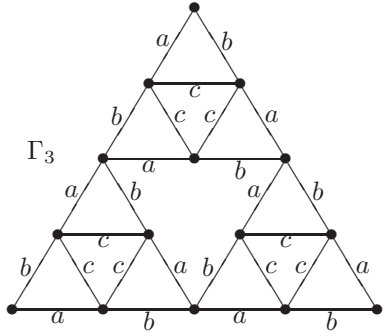
\begin{center}
\begin{picture}(400,205)

\put(90,110){$\Gamma_3$}

\letvertex A=(200,210)\letvertex B=(170,160)\letvertex C=(140,110)

\letvertex D=(110,60)\letvertex E=(80,10)\letvertex F=(140,10)\letvertex G=(200,10)

\letvertex H=(260,10)\letvertex I=(320,10)
\letvertex L=(290,60)\letvertex M=(260,110)\letvertex N=(230,160)

\letvertex O=(200,110)\letvertex P=(170,60)\letvertex Q=(230,60)

\drawvertex(A){$\bullet$}\drawvertex(B){$\bullet$}
\drawvertex(C){$\bullet$}\drawvertex(D){$\bullet$}
\drawvertex(E){$\bullet$}\drawvertex(F){$\bullet$}
\drawvertex(G){$\bullet$}\drawvertex(H){$\bullet$}
\drawvertex(I){$\bullet$}\drawvertex(L){$\bullet$}\drawvertex(M){$\bullet$}
\drawvertex(N){$\bullet$}\drawvertex(O){$\bullet$}
\drawvertex(P){$\bullet$}\drawvertex(Q){$\bullet$}

\drawundirectededge(E,D){$b$} \drawundirectededge(D,C){$a$}
\drawundirectededge(C,B){$b$} \drawundirectededge(B,A){$a$}
\drawundirectededge(A,N){$b$} \drawundirectededge(N,M){$a$}
\drawundirectededge(M,L){$b$} \drawundirectededge(L,I){$a$}
\drawundirectededge(I,H){$b$} \drawundirectededge(H,G){$a$}
\drawundirectededge(G,F){$b$} \drawundirectededge(F,E){$a$}
\drawundirectededge(N,B){$c$} \drawundirectededge(O,C){$a$}
\drawundirectededge(M,O){$b$} \drawundirectededge(P,D){$c$}
\drawundirectededge(L,Q){$c$} \drawundirectededge(B,O){$c$}
\drawundirectededge(O,N){$c$} \drawundirectededge(C,P){$b$}
\drawundirectededge(P,G){$a$} \drawundirectededge(D,F){$c$}
\drawundirectededge(Q,M){$a$} \drawundirectededge(G,Q){$b$}
\drawundirectededge(H,L){$c$}\drawundirectededge(F,P){$c$}
\drawundirectededge(Q,H){$c$}
\end{picture}
\end{center}
We distinguish here the following types of dimers coverings: we
say that a dimer covering is of type $g$ (respectively of type
$h$, of type $f$, of type $t$) if exactly one (respectively
exactly two, none, all) of the three corners of $\Gamma_n$ is
(are) covered.
 Observe that by symmetry of the labeling
we do not need to define $g^{ab}, g^{ac}, g^{bc}, h^{ab}, h^{ac},
h^{bc}$. It is easy to
check that in this model for $n$ odd we can only have
configurations of type $f$ and $h$, for $n$ even we can only have
configurations of type $g$ and $t$.
By using recursion, one also checks that for $n$ even, the
number of coverings of type $t$ is one third of the number of
coverings of type $g$, and that for $n$ odd, the number of
coverings of type $f$ is one third of the number of dimer
coverings of type $h$.\\
\indent Next, we compute the partition function $\Phi_n(a,b,c)$
associated with the \lq\lq rotation-invariant\rq\rq labeling. We
will denote by $f_n,g_n,h_n,t_n$ the summands in $\Phi_n(a,b,c)$
corresponding to coverings of types $f,g,h,t$. For instance, for
$n=2$, we have $g_2=3c(a+b)$, $t_2=a^3+b^3$, so that
$$
\Phi_2(a,b,c)= a^3+b^3+3c(a+b).
$$
\begin{teo}\label{partition2011}
The partition function of $\Gamma_n$, for each $n\geq 2$, is given
by:
$$
\begin{cases}
    \Phi_n(a,b,c)= 2^{\frac{3^{n-2}-1}{2}}(a^3+b^3)^{\frac{3^{n-2}-1}{4}}(ac+bc)^{\frac{3^{n-1}-3}{4}}(a^3+b^3+3c(a+b))& \text{for } n \text{ even} \\
    \Phi_n(a,b,c)= 2^{\frac{3^{n-2}-1}{2}}(a^3+b^3)^{\frac{3^{n-2}-3}{4}}(ac+bc)^{\frac{3^{n-1}-1}{4}}(3(a^3+b^3)+c(a+b))& \text{for } n \text{ odd}.
\end{cases}
$$
\end{teo}

\begin{proof}
Similarly to how we computed the partition function for the \lq\lq
Schreier\rq\rq labeling, we get, for $n\geq 3$ odd:
\begin{eqnarray}\label{parisec}
  \begin{cases}
  f_n = 2\left(\frac{g_{n-1}}{3}\right)^3\\
  h_n= 6t_{n-1}\left(\frac{g_{n-1}}{3}\right)^2
\end{cases}
\end{eqnarray}
and for $n$ even:
\begin{eqnarray}\label{disparisec}
\begin{cases}
  t_n = 2\left(\frac{h_{n-1}}{3}\right)^3 \\
  g_n= 6f_{n-1}\left(\frac{h_{n-1}}{3}\right)^2.
  \end{cases}
\end{eqnarray}
The solutions of systems (\ref{parisec}) and (\ref{disparisec}),
with the initial conditions $g_2=3c(a+b)$ and $t_2=a^3+b^3$, can
be computed by induction on $n$: one gets, for $n$ even,
$$
\begin{cases}
t_n=2^{\frac{3^{n-2}-1}{2}}(a^3+b^3)^{\frac{3^{n-2}+3}{4}}(ac+bc)^{\frac{3^{n-1}-3}{4}}\\
g_n=3\cdot
2^{\frac{3^{n-2}-1}{2}}(a^3+b^3)^{\frac{3^{n-2}-1}{4}}(ac+bc)^{\frac{3^{n-1}+1}{4}}
\end{cases}
$$
and for $n\geq 3$ odd
$$
\begin{cases}
f_n=2^{\frac{3^{n-2}-1}{2}}(a^3+b^3)^{\frac{3^{n-2}-3}{4}}(ac+bc)^{\frac{3^{n-1}+3}{4}}\\
h_n=3\cdot
2^{\frac{3^{n-2}-1}{2}}(a^3+b^3)^{\frac{3^{n-2}+1}{4}}(ac+bc)^{\frac{3^{n-1}-1}{4}}.
\end{cases}
$$
The assertion follows from the fact that $\Phi_n(a,b,c)= f_n +
h_n$ for $n$ odd and $\Phi_n(a,b,c)= t_n+g_n$ for $n$ even.
\end{proof}

\begin{cor}
The thermodynamic limit is $\frac{1}{9}\log 2+\frac{1}{18}\log
(a^3+b^3)+\frac{1}{6}\log c(a+b)$.
\end{cor}

By putting $a=b=c=1$, one finds the same value of entropy as in
the \lq\lq Schreier\rq\rq labeling, as expected.


\section{Statistics}\label{statistiques}

In this section we study the statistics of occurrences of edges
with a given label in a random dimer covering, for the Schreier
graphs of $H^{(3)}$ and for the Sierpi\'{n}ski triangles. We
compute the mean and the variance and, in some cases, we are able
to find the asymptotic behavior of the moment generating function
of the associated normalized random variable.

\subsection{Schreier graphs of the Hanoi Towers group}

Denote by $c_n$ (by symmetry, $a_n$ and $b_n$ can be studied in the same way) the random variable that counts the number of occurrences of edges labeled
$c$ in a random dimer covering on $\Sigma_n$. In order to study it, introduce the function $\Phi_n^i(c) = \Phi_n^i(1,1,c)$, for
$i=I, II, III, IV$, and observe that
$\Phi_n^{III}(c)=\Phi_n^{IV}(c)$. Moreover, denote by $\mu_{n,i}$
and $\sigma^2_{n,i}$ the mean and the variance of $c_n$ in a
random dimer covering of type $i$, respectively. Note that we have
$\mu_{n,III}= \mu_{n,IV}$ and $\sigma^2_{n,III}=\sigma^2_{n,IV}$.
\begin{teo}
For each $n\geq 1$,
$$
\mu_{n,I}=\frac{3^{n-1}+1}{2}, \qquad
\mu_{n,II}=\frac{3^{n-1}+3}{2}, \qquad
\mu_{n,III}=\frac{3^{n-1}-1}{2},
$$
$$
\sigma^2_{n,I} = \frac{3^n-6n+3}{4}, \qquad \sigma_{n,II}^2 =
\frac{3^n+10n-13}{4}, \qquad \sigma_{n,III}^2 = \frac{3^n-2n-1}{4}.
$$
\end{teo}
\begin{proof}
For $a=b=1$, the system \eqref{generalsystem} reduces to
$$
\begin{cases}
\Phi^I_{n+1} = \frac{\left(\Phi^I_n\right)^3}{c} + \Phi_n^{II}\left(\Phi_n^{III}\right)^2 \\
\Phi^{II}_{n+1} =
\frac{\left(\Phi_n^{II}\right)^3}{c}+\Phi_n^{I}\left(\Phi_n^{III}\right)^2\\
\Phi^{III}_{n+1} =
\left(\Phi_n^{III}\right)^3+\frac{\Phi_n^{I}\Phi_n^{II}\Phi_n^{III}}{c}
\end{cases}
$$
with initial conditions $\Phi^I_1(c)=c$, $\Phi^{II}_1(c)=c^2$ and
$\Phi^{III}_1(c)=1$. Now put, for every $n\geq 1$,
$q_n=\frac{\Phi_n^{II}}{\Phi_n^I}$ and
$r_n=\frac{\Phi_n^{III}}{\Phi_n^I}$. Observe that both $q_n$ and
$r_n$ are functions of the only variable $c$. In particular, for
each $n$, one has $q_n(1)=r_n(1)=1$, since the number of dimer
covering is the same for each type of configuration. By computing
the quotient $\Phi_{n+1}^{II}/\Phi_{n+1}^I$ and dividing each term
by $(\Phi_n^I)^3$, one gets
\begin{eqnarray}\label{uno1}
q_{n+1} = \frac{\frac{q_n^3}{c}+r_n^2}{\frac{1}{c}+q_nr_n^2}.
\end{eqnarray}
Similarly, one has
\begin{eqnarray}\label{due2}
r_{n+1} = \frac{r_n^3+\frac{q_nr_n}{c}}{\frac{1}{c}+q_nr_n^2}.
\end{eqnarray}
Using (\ref{uno1}) and (\ref{due2}), one can show by induction
that $q_n'(1)=1$ and $r_n'(1)=-1$, for every $n\geq 1$. Moreover,
$q_n''(1)=4(n-1)$ and $r_n''(1)=n+1$.\\ From the first equation of
the system, one gets
$\Phi_n^I(c)=\frac{(\Phi_{n-1}^I)^3}{c}\left(1+q_{n-1}(c)r_{n-1}^2(c)\right)$.
By applying the logarithm and using recursion, we have:
$$
\log(\Phi_n^I(c))=3^{n-1}\log(\Phi_1^I(c))-\sum_{k=0}^{n-2}3^k\log
c+\sum_{k=1}^{n-1}3^{n-1-k}\log(1+cq_k(c)r_k^2(c)).
$$
Taking the derivative and putting $c=1$, one gets
$$
\mu_{n,I}=\frac{\Phi_n^{I'}(c)}{\Phi_n^I(c)}\left|_{c=1}=\frac{3^{n-1}+1}{2},\right.
$$
what is one third of the total number of edges involved in such a
covering, as it was to be expected because of the symmetry of the
graph and of the labeling. Taking once more derivative, one gets
$$
\frac{\Phi_n^{I''}(c)\Phi_n^I(c)-(\Phi_n^{I'}(c))^2}{(\Phi_n^I(c))^2}\left|_{c=1}=\frac{3^{n-1}-6n+1}{4}.\right.
$$
Hence,
$$ \sigma_{n,I}^2=\frac{\Phi_n^{I'}(1)}{\Phi_n^I(1)} +
\frac{\Phi_n^{I''}(1)\Phi_n^I(1)}{(\Phi_n^I(1))^2}-\frac{(\Phi_n^{I'}(1))^2}{(\Phi_n^I(1))^2}
=\frac{3^n-6n+3}{4}.
$$
In a similar way one can find $\mu_{n,II},
\sigma_{n,II}^2,\mu_{n,III},\sigma_{n,III}^2$.
\end{proof}

Observe that one has $\mu_{n,II} > \mu_{n,I} >\mu_{n,III}$: this
corresponds to the fact that the distribution of labels $a,b,c$ is
uniform in a configuration of type I, but not in the other ones.
In fact, a configuration of type II has a loop labeled $c$, but a
configuration of type III (resp. IV) has a loop labeled $b$ (resp.
$a$): so the label $c$ is \lq\lq dominant\rq\rq in type II,
whereas the label $b$ (resp. $a$) is \lq\lq dominant\rq\rq in type
III (resp. IV).

\subsection{Sierpi\'{n}ski triangles}

\begin{teo}\label{Theoprobabilityfunctions}
For Sierpi\'{n}ski triangles with the \lq\lq Schreier\rq\rq labeling, for each $n\geq 1$, the random variable $c_n$ has
$$
\mu_n=\frac{3^{n-1}}{4}, \qquad \sigma^2_n=\frac{3}{16}.
$$
Moreover, the associated probability density function is
$$
f(x) =
\begin{cases}
\frac{3}{4}\delta(x+\frac{1}{\sqrt{3}})+\frac{1}{4}\delta(x-\sqrt{3})& n \ \text{odd}\\
\frac{3}{4}\delta(x-\frac{1}{\sqrt{3}})+\frac{1}{4}\delta(x+\sqrt{3})&n\
\text{even},
\end{cases}
$$
where $\delta$ denotes the Dirac function.
\end{teo}
\begin{proof}
Putting $a=b=1$, one gets
$$
\begin{cases}
\Phi_n(c)= (4c)^{\frac{3^{n-1}-1}{4}}(c+3)& \text{for } n \text{
odd}\\
\Phi_n(c)= 2(4c)^{\frac{3^{n-1}-3}{4}}(3c+1)& \text{for } n \text{
even}.
\end{cases}
$$
The mean and the variance of $c_n$ can be computed as in the
previous theorem, by using logarithmic derivatives. Now let $C_n =
\frac{c_n-\mu_n}{\sigma_n}$ be the normalized random variable,
then the moment generating function of $C_n$ is given by
$$
\mathbb{E}(e^{sC_n}) =
e^{-\mu_ns/\sigma_n}\mathbb{E}(e^{sx_n/\sigma_n}) =
e^{-\mu_ns/\sigma_n}\frac{\Phi_n(e^{s/\sigma_n})}{\Phi_n(1)}.
$$
We get
$$
\mathbb{E}(e^{sC_n}) =
\begin{cases}
\frac{e^{\sqrt{3}s}+3e^{-\frac{s}{\sqrt{3}}}}{4} & n\ \text{odd}\\
\frac{e^{-\sqrt{3}s}+3e^{\frac{s}{\sqrt{3}}}}{4} & n\ \text{even}.
\end{cases}
$$
and the claim follows.
\end{proof}

Observe that the moment generating functions that we have found
only depend on the parity of $n$. The following theorem gives an
interpretation of the probability density functions given in
Theorem \ref{Theoprobabilityfunctions}.
\begin{teo}\label{spiegazionedensity}
For $n$ odd, the normalized random variable $C_n$ is equal to
$\sqrt{3}$ in each covering of type $h^{ab}$ and to
$\frac{-1}{\sqrt{3}}$ in each covering of type $f,h^{ac},h^{bc}$.
For $n$ even, the normalized random variable $C_n$ is equal to
$-\sqrt{3}$ in each covering of type $g^{ab}$ and to
$\frac{1}{\sqrt{3}}$ in each covering of type $t,g^{ac},g^{bc}$.
\end{teo}

\begin{proof}
The assertion can be proved by induction. For $n=1,2$ a direct
computation shows that the assertion is true. We give here only
the proof for $n>2$ odd. The following pictures show how to get a
labeled dimer covering for $\Gamma_n$, $n$ odd, starting from
three dimer coverings of $\Gamma_{n-1}$. One can easily check that
these recursions hold, by using the definition of the labeling of
$\Gamma_n$. \unitlength=0,4mm
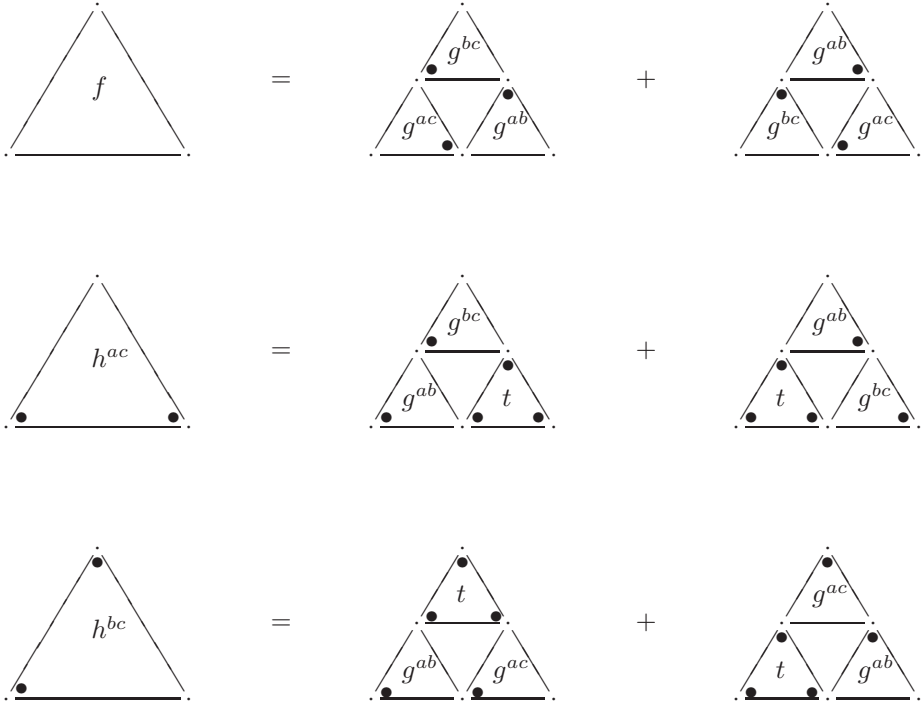
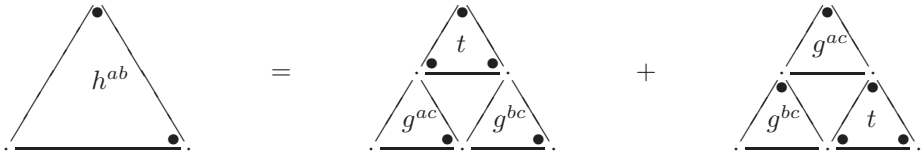
\begin{center}
\begin{picture}(400,70)
\letvertex A=(80,60)\letvertex B=(50,10)\letvertex C=(110,10)

\letvertex D=(200,60)\letvertex E=(185,35)\letvertex F=(170,10)\letvertex G=(200,10)

\letvertex H=(230,10)\letvertex I=(215,35)
\letvertex L=(320,60)\letvertex M=(305,35)\letvertex N=(290,10)

\letvertex O=(320,10)\letvertex P=(350,10)\letvertex Q=(335,35)
\letvertex R=(190,38)\letvertex S=(215,30)\letvertex T=(195,13)
\letvertex U=(330,38)\letvertex V=(305,30)\letvertex Z=(325,13)
\drawvertex(R){$\bullet$}\drawvertex(S){$\bullet$}
\drawvertex(T){$\bullet$}\drawvertex(U){$\bullet$}
\drawvertex(V){$\bullet$}\drawvertex(Z){$\bullet$}

\put(80,60){\circle*{1}}\put(50,10){\circle*{1}}\put(110,10){\circle*{1}}
\put(200,60){\circle*{1}}\put(185,35){\circle*{1}}\put(170,10){\circle*{1}}\put(200,10){\circle*{1}}
\put(230,10){\circle*{1}}\put(215,35){\circle*{1}}
\put(320,60){\circle*{1}}\put(305,35){\circle*{1}}\put(290,10){\circle*{1}}
\put(320,10){\circle*{1}}\put(350,10){\circle*{1}}\put(335,35){\circle*{1}}
\put(137,33){$=$}\put(257,33){$+$}

\put(78,30){$f$}\put(195,42){$g^{bc}$}\put(180,17){$g^{ac}$}
\put(210,17){$g^{ab}$}\put(315,42){$g^{ab}$}\put(300,17){$g^{bc}$}\put(330,17){$g^{ac}$}

\drawundirectededge(A,B){} \drawundirectededge(B,C){}
\drawundirectededge(C,A){} \drawundirectededge(D,E){}
\drawundirectededge(E,F){} \drawundirectededge(F,G){}
\drawundirectededge(G,H){} \drawundirectededge(H,I){}
\drawundirectededge(I,D){} \drawundirectededge(E,G){}
\drawundirectededge(G,I){} \drawundirectededge(I,E){}
\drawundirectededge(L,M){} \drawundirectededge(M,N){}
\drawundirectededge(N,O){} \drawundirectededge(O,P){}
\drawundirectededge(P,Q){} \drawundirectededge(Q,L){}
\drawundirectededge(M,Q){} \drawundirectededge(M,O){}
\drawundirectededge(O,Q){}
\end{picture}
\end{center}
\begin{center}
\begin{picture}(400,70)

\letvertex A=(80,60)\letvertex B=(50,10)\letvertex C=(110,10)

\letvertex D=(200,60)\letvertex E=(185,35)\letvertex F=(170,10)\letvertex G=(200,10)

\letvertex H=(230,10)\letvertex I=(215,35)
\letvertex L=(320,60)\letvertex M=(305,35)\letvertex N=(290,10)

\letvertex O=(320,10)\letvertex P=(350,10)\letvertex Q=(335,35)
\letvertex u=(345,13)\letvertex r=(205,13)\letvertex Z=(315,13)\letvertex T=(175,13)\letvertex s=(225,13)
\letvertex S=(295,13)
\letvertex a=(55,13)\letvertex b=(105,13)
\letvertex V=(215,30)\letvertex t=(305,30)
\letvertex R=(190,38)\letvertex U=(330,38)

\put(78,30){$h^{ac}$}\put(195,42){$g^{bc}$}\put(180,17){$g^{ab}$}
\put(213,17){$t$}\put(315,42){$g^{ab}$}\put(303,17){$t$}\put(330,17){$g^{bc}$}

\drawvertex(a){$\bullet$}\drawvertex(b){$\bullet$}
\drawvertex(u){$\bullet$}\drawvertex(r){$\bullet$}
\drawvertex(s){$\bullet$}\drawvertex(t){$\bullet$}
\drawvertex(R){$\bullet$}\drawvertex(S){$\bullet$}
\drawvertex(T){$\bullet$}\drawvertex(U){$\bullet$}
\drawvertex(V){$\bullet$}\drawvertex(Z){$\bullet$}

\put(80,60){\circle*{1}}\put(50,10){\circle*{1}}\put(110,10){\circle*{1}}
\put(200,60){\circle*{1}}\put(185,35){\circle*{1}}\put(170,10){\circle*{1}}\put(200,10){\circle*{1}}
\put(230,10){\circle*{1}}\put(215,35){\circle*{1}}
\put(320,60){\circle*{1}}\put(305,35){\circle*{1}}\put(290,10){\circle*{1}}
\put(320,10){\circle*{1}}\put(350,10){\circle*{1}}\put(335,35){\circle*{1}}

\put(137,33){$=$}\put(257,33){$+$}

\drawundirectededge(A,B){} \drawundirectededge(B,C){}
\drawundirectededge(C,A){} \drawundirectededge(D,E){}
\drawundirectededge(E,F){} \drawundirectededge(F,G){}
\drawundirectededge(G,H){} \drawundirectededge(H,I){}
\drawundirectededge(I,D){} \drawundirectededge(E,G){}
\drawundirectededge(G,I){} \drawundirectededge(I,E){}
\drawundirectededge(L,M){} \drawundirectededge(M,N){}
\drawundirectededge(N,O){} \drawundirectededge(O,P){}
\drawundirectededge(P,Q){} \drawundirectededge(Q,L){}
\drawundirectededge(M,Q){} \drawundirectededge(M,O){}
\drawundirectededge(O,Q){}
\end{picture}
\end{center}
\begin{center}
\begin{picture}(400,70)
\letvertex A=(80,60)\letvertex B=(50,10)\letvertex C=(110,10)

\letvertex D=(200,60)\letvertex E=(185,35)\letvertex F=(170,10)\letvertex G=(200,10)

\letvertex H=(230,10)\letvertex I=(215,35)
\letvertex L=(320,60)\letvertex M=(305,35)\letvertex N=(290,10)

\letvertex O=(320,10)\letvertex P=(350,10)\letvertex Q=(335,35)

\letvertex a=(80,55)\letvertex b=(55,13)\letvertex d=(200,55)\letvertex e=(190,37)
\letvertex i=(211,37)\letvertex f=(175,12)\letvertex g=(205,12)

\letvertex l=(320,55)\letvertex m=(305,30)\letvertex n=(295,12)\letvertex o=(315,12)
\letvertex q=(335,30)

\put(78,30){$h^{bc}$}\put(198,42){$t$}\put(180,17){$g^{ab}$}
\put(210,17){$g^{ac}$}\put(315,42){$g^{ac}$}\put(303,17){$t$}\put(330,17){$g^{ab}$}

\drawvertex(a){$\bullet$}
\drawvertex(b){$\bullet$}\drawvertex(f){$\bullet$}
\drawvertex(d){$\bullet$}\drawvertex(g){$\bullet$}
\drawvertex(e){$\bullet$}\drawvertex(l){$\bullet$}
\drawvertex(i){$\bullet$}\drawvertex(m){$\bullet$}
\drawvertex(n){$\bullet$}
\drawvertex(o){$\bullet$}\drawvertex(q){$\bullet$}

\put(80,60){\circle*{1}}\put(50,10){\circle*{1}}\put(110,10){\circle*{1}}
\put(200,60){\circle*{1}}\put(185,35){\circle*{1}}\put(170,10){\circle*{1}}\put(200,10){\circle*{1}}
\put(230,10){\circle*{1}}\put(215,35){\circle*{1}}
\put(320,60){\circle*{1}}\put(305,35){\circle*{1}}\put(290,10){\circle*{1}}
\put(320,10){\circle*{1}}\put(350,10){\circle*{1}}\put(335,35){\circle*{1}}

\put(137,33){$=$}\put(257,33){$+$}

\drawundirectededge(A,B){} \drawundirectededge(B,C){}
\drawundirectededge(C,A){} \drawundirectededge(D,E){}
\drawundirectededge(E,F){} \drawundirectededge(F,G){}
\drawundirectededge(G,H){} \drawundirectededge(H,I){}
\drawundirectededge(I,D){} \drawundirectededge(E,G){}
\drawundirectededge(G,I){} \drawundirectededge(I,E){}
\drawundirectededge(L,M){} \drawundirectededge(M,N){}
\drawundirectededge(N,O){} \drawundirectededge(O,P){}
\drawundirectededge(P,Q){} \drawundirectededge(Q,L){}
\drawundirectededge(M,Q){} \drawundirectededge(M,O){}
\drawundirectededge(O,Q){}
\end{picture}
\end{center}
\begin{center}
\begin{picture}(400,70)
\letvertex A=(80,60)\letvertex B=(50,10)\letvertex C=(110,10)

\letvertex D=(200,60)\letvertex E=(185,35)\letvertex F=(170,10)\letvertex G=(200,10)

\letvertex H=(230,10)\letvertex I=(215,35)
\letvertex L=(320,60)\letvertex M=(305,35)\letvertex N=(290,10)

\letvertex O=(320,10)\letvertex P=(350,10)\letvertex Q=(335,35)

\letvertex a=(80,55)\letvertex c=(105,13)\letvertex d=(200,55)\letvertex e=(190,38)
\letvertex i=(210,38)\letvertex g=(195,13)\letvertex h=(225,13)
\letvertex l=(320,55)\letvertex m=(305,30)\letvertex q=(335,30)
\letvertex o=(325,13)\letvertex p=(345,13)

\put(78,30){$h^{ab}$}\put(198,42){$t$}\put(180,17){$g^{ac}$}
\put(210,17){$g^{bc}$}\put(315,42){$g^{ac}$}\put(300,17){$g^{bc}$}\put(333,17){$t$}

\drawvertex(a){$\bullet$}\drawvertex(m){$\bullet$}
\drawvertex(c){$\bullet$}\drawvertex(q){$\bullet$}
\drawvertex(d){$\bullet$}\drawvertex(o){$\bullet$}
\drawvertex(e){$\bullet$}\drawvertex(p){$\bullet$}
\drawvertex(i){$\bullet$}\drawvertex(g){$\bullet$}
\drawvertex(h){$\bullet$}\drawvertex(l){$\bullet$}

\put(80,60){\circle*{1}}\put(50,10){\circle*{1}}\put(110,10){\circle*{1}}
\put(200,60){\circle*{1}}\put(185,35){\circle*{1}}\put(170,10){\circle*{1}}\put(200,10){\circle*{1}}
\put(230,10){\circle*{1}}\put(215,35){\circle*{1}}
\put(320,60){\circle*{1}}\put(305,35){\circle*{1}}\put(290,10){\circle*{1}}
\put(320,10){\circle*{1}}\put(350,10){\circle*{1}}\put(335,35){\circle*{1}}

\put(137,33){$=$}\put(257,33){$+$}

\drawundirectededge(A,B){} \drawundirectededge(B,C){}
\drawundirectededge(C,A){} \drawundirectededge(D,E){}
\drawundirectededge(E,F){} \drawundirectededge(F,G){}
\drawundirectededge(G,H){} \drawundirectededge(H,I){}
\drawundirectededge(I,D){} \drawundirectededge(E,G){}
\drawundirectededge(G,I){} \drawundirectededge(I,E){}
\drawundirectededge(L,M){} \drawundirectededge(M,N){}
\drawundirectededge(N,O){} \drawundirectededge(O,P){}
\drawundirectededge(P,Q){} \drawundirectededge(Q,L){}
\drawundirectededge(M,Q){} \drawundirectededge(M,O){}
\drawundirectededge(O,Q){}
\end{picture}
\end{center}
If we look at the first three pictures, we see that the variable
$C_n$ in a dimer covering of type $f, h^{ac}, h^{bc}$ is given, by
induction, by the sum of two contributions $1/\sqrt{3}$ and one
contribution $-\sqrt{3}$, which gives $-1/\sqrt{3}$. The fourth
picture shows that the variable $C_n$ in a dimer covering of type
$h^{ab}$ is given, by induction, by the sum of three contributions
$1/\sqrt{3}$, which gives $\sqrt{3}$. A similar proof can be given
for $n$ even. The statement follows.
\end{proof}

\begin{teo}
For Sierpi\'nski graphs with the \lq\lq rotation-invariant\rq\rq\
labeling, for each $n\geq 2$, the random variables $a_n$ and $b_n$
have
$$
\mu_n=\frac{3^{n-1}}{4} \qquad \sigma^2_n=\frac{4\cdot
3^{n-1}+3}{4}
$$
and they are asymptotically normal. The random variable $c_n$ has
$$
\mu_n=\frac{3^{n-1}}{4} \qquad \sigma^2_n=\frac{3}{16}
$$
and the associated probability density function is
$$
f(x) =
\begin{cases}
\frac{3}{4}\delta(x-\frac{1}{\sqrt{3}})+\frac{1}{4}\delta(x+\sqrt{3})
\qquad \mbox{for }n \mbox{ even}\\
\frac{3}{4}\delta(x+\frac{1}{\sqrt{3}})+\frac{1}{4}\delta(x-\sqrt{3})
\qquad \mbox{for }n \mbox{ odd}.
\end{cases}
$$
\end{teo}
\begin{proof}
By putting $b=c=1$ in the partition functions given in Theorem
\ref{partition2011}, one gets
$$
\begin{cases}
\Phi_n(a)= 2^{\frac{3^{n-2}-1}{2}}(a^3+1)^{\frac{3^{n-2}-1}{4}}(a+1)^{\frac{3^{n-1}-3}{4}}(a^3+3a+4)& \text{for } n \text{ even} \\
\Phi_n(a)=
2^{\frac{3^{n-2}-1}{2}}(a^3+1)^{\frac{3^{n-2}-3}{4}}(a+1)^{\frac{3^{n-1}-1}{4}}(3a^3+a+4)&
\text{for } n \text{ odd}.
\end{cases}
$$
Similarly, one can find
$$
\begin{cases}
\Phi_n(c)= 2^{\frac{3^{n-1}-1}{2}}c^{\frac{3^{n-1}-3}{4}}(3c+1)& \text{for } n \text{ even} \\
\Phi_n(c)= 2^{\frac{3^{n-1}-1}{2}}c^{\frac{3^{n-1}-1}{4}}(c+3)&
\text{for } n \text{ odd}.
\end{cases}
$$
Then one proceeds as in the previously studied cases.
\end{proof}
A similar interpretation as in the case of the \lq\lq
Schreier\rq\rq labeling can be given.
\begin{teo}
For $n$ even, the normalized random variable $C_n$ is equal to
$-\sqrt{3}$ in each covering of type $t$ and to
$\frac{1}{\sqrt{3}}$ in each covering of type $g$. For $n$ odd,
the normalized random variable $C_n$ is equal to $\sqrt{3}$ in
each covering of type $f$ and to $-\frac{1}{\sqrt{3}}$ in each
covering of type $h$.
\end{teo}

\begin{os}\label{finaleremark}\rm
In \cite{wagner1} the authors study the statistical properties of
the dimer model on $\Gamma_n$ endowed with the \lq\lq
directional\rq\rq labeling: for $n$ even (which is the only case
allowing a perfect matching), they get the following expressions
for the mean and the variance of the number of labels $c$:
$$
\mu_n=\frac{3^{n-1}+1}{4} \qquad \sigma_n^2=\frac{3^{n-1}-3}{4}.
$$
Moreover, they show that the associated normalized random variable
tends weakly to the normal distribution.
\end{os}

\end{document}